\documentclass{elsarticle}

\makeatletter
\def\ps@pprintTitle{%
 \let\@oddhead\@empty
 \let\@evenhead\@empty
 \def\@oddfoot{\centerline{\thepage}}%
 \let\@evenfoot\@oddfoot}
\makeatother
\usepackage{graphicx}
\usepackage{color}
\usepackage{booktabs}
\usepackage{amsmath}
\usepackage{amssymb}
 \usepackage{amsthm}
\usepackage{subfigure}
\usepackage{pst-node}
 \usepackage{graphicx}
 \usepackage{amsmath}
 \usepackage{amssymb}
 \usepackage{amsthm}
 \usepackage{latexsym}
 \usepackage{textcomp}
\usepackage{wasysym}
 \usepackage{amscd}
\usepackage{color}
\usepackage{setspace}
\usepackage{txfonts}
\usepackage{mathrsfs}
\usepackage{stmaryrd}
\usepackage{fancyvrb} 
\usepackage{graphicx}
\usepackage{color}
\usepackage{booktabs}
\usepackage{amsmath}
\usepackage{amssymb}
 \usepackage{amsthm}
\usepackage{algorithmic}
\usepackage{algorithm}
\usepackage{subfigure}
\usepackage{cancel}
\usepackage{fullpage}
\usepackage{enumerate}
\usepackage{fullpage}
\usepackage{xcolor,colortbl}
\usepackage{multirow}

\definecolor{lightgray}{gray}{0.8}

\newenvironment{assumption2}[1]
{
{\bf Assumption #1.} 
\it 
}
{

 } 
 \newtheorem{thm}{Theorem}
\newtheorem{lem}[thm]{Lemma}
\newdefinition{rmk}{Remark}
\newproof{pf}{Proof}
\newproof{pot}{Proof of Theorem \ref{thm2}}
 
 \newcommand{\set}[1]{\left\{#1\right\}}
\newcommand{\norm}[1]{\left\|\,#1\,\right\|}

  \newcommand{\wh}[1]{\widehat{#1}}
\newcommand{\yt}{\tilde{y}_n}  
   
\newcommand{\rh}{\widehat{r}} 
  \newcommand{\Jt}{\tilde{J}_n}  
  \newcommand{\nt}{\tilde{N}_n}

\renewcommand{\O}{\mathcal{O}}

\newcommand{\X}{\mathscr{X}}

\newcommand{\R}{\mathbb{R}}
\renewcommand{\P}{\mathbb{P}}
\newcommand{\N}{\mathbb{N}}

\renewcommand{\yt}{\tilde{u}_n}  
   
\renewcommand{\rh}{\widehat{r}} 
\newcommand{\ut}{\tilde{u}_n}  
\newcommand{\Uh}{\widehat{U}}

\numberwithin{equation}{section}
\begin{document}
    \begin{frontmatter}
        \title{A new approach to constructing efficient stiffly accurate EPIRK methods}
        \author[ucm]{G. Rainwater\corref{cor1}}
        \ead{grainwater@ucmerced.edu}

        \author[ucm]{M. Tokman}
        \ead{mtokman@ucmerced.edu}

        \cortext[cor1]{Corresponding author.  Tel.: +1 925 209 3021; fax: +1 209 228 4060.}
        \address[ucm]{School of Natural Sciences, University of California, 5200 N. Lake Road, Merced, CA 95343}

        \begin{abstract}
                 The structural flexibility of the exponential propagation iterative methods of Runge-Kutta type (EPIRK) enables construction of particularly efficient exponential time integrators.  While the EPIRK methods have been shown to perform well on stiff problems, all of the schemes proposed up to now have been derived using classical order conditions. In this paper we extend the stiff order conditions and the convergence theory developed for the exponential Rosenbrock methods to the EPIRK integrators. We derive stiff order conditions for the EPIRK methods and develop algorithms to solve them to obtain specific schemes.  Moreover, we propose a new approach to constructing particularly efficient EPIRK integrators that are optimized to work with an adaptive Krylov algorithm.  We use a set of numerical examples to illustrate the computational advantages that the newly constructed EPIRK methods offer compared to previously proposed exponential integrators.

        \end{abstract}

    \end{frontmatter}

\section{Introduction}   \label{sec:introduction}
Stiff systems of differential equations of the form 		
\begin{equation}
\label{eqn:ivp}
u'(t)=f(u(t)),\qquad u(t_0)=u_0, \quad u(t) \in \mathbf{R}^N
\end{equation}
are routinely encountered in a wide variety of scientific and engineering applications.  Obtaining the numerical solution to this problem over a long time interval compared to the fastest scales in the system is a challenging task that has been traditionally addressed with the use of implicit time integrators. The implicit methods have better stability properties compared to explicit techniques and thus allow for numerical integration of (\ref{eqn:ivp}) with larger time steps. However, while an implicit method can outperform an explicit scheme, it too is affected by the stiffness of the problem which manifests itself in the solution of the implicit equations at each time step.  A general stiff system of type (\ref{eqn:ivp}) is typically solved with an implicit method that has a Newton iteration embedded within each time step.  Each Newton iteration in turn requires approximation of a product of a rational function of the Jacobian with a vector $(I - cf'(u))^{-1}v$ where $c$ is a constant, $I$ is an $N \times N$ matrix and $u$ and $v$ are $N$-dimensional vectors.  For a general stiff matrix $f'(u)$  a method of choice to approximate $(I - cf'(u))^{-1}v$  is usually a Krylov-projection based algorithm such as GMRES.  Stiffness of the matrix $f'(u)$ results in the slow convergence of any Krylov-projection-type algorithm. Developing an efficient preconditioner is often essential to making an implicit Newton-Krylov time integrator sufficiently fast.  However, construction of such preconditioner can often be a difficult and even impossible task. Consequently development of more efficient time integrators becomes an important problem in numerical analysis. 

Exponential integrators received attention over the past decade as an alternative to implicit methods for stiff systems of type \eqref{eqn:ivp}.  Just like implicit methods, exponential integrators possess good stability properties but they require evaluation of exponential-like, rather than rational, matrix function-vector products. Using a Krylov-projection based method to evaluate an exponential-like function can save significant amount of computational time compared to the rational function evaluation needed for an implicit integrator.   Exponential propagation iterative methods of Runge-Kutta-type (EPIRK) framework has been introduced to enable construction of particularly efficient exponential methods.  The general formulation of the EPIRK methods is 
\begin{equation}
\label{eqn:genEPIRK}
  \begin{aligned}
U_{ni}&= u_n+\alpha_{i1}\psi_{i1}(g_{i1}h_nA_{i1})h_nf(u_n)+h_n\sum_{j=2}^{i-1} \alpha_{ij}\psi_{ij}(g_{ij}h_nA_{ij})\Delta^{(j-1)}r(u_n),\quad i=2,\dots,s,\\
u_{n+1} &= u_n +\beta_{1}\psi_{s+1 1}(g_{s+1 1}h_nA_{i1})h_nf(u_n)+h_n\sum_{j=2}^s \beta_j\psi_{s+1j}(g_{s+1j}h_nA_{ij})\Delta^{(j-1)}r(u_n)\;\;\
  \end{aligned}
\end{equation}
where  $h_n=t_{n+1}-t_n$ is the time step and $\Delta^{(k)} $ denotes the $k$th forward difference vector. As described in \cite{tokmanOberwolfach} each matrix $A_{ij}$ can be either a full Jacobian $J_n=f'(u_n)$ or a part of the full Jacobian if the operator $f(u)$ can be partitioned in some meaningful way.  For example, we can set $A_{ij}=L$ when the right-hand-side operator in \eqref{eqn:ivp} can be partitioned as $f(u) = Lu + N(u)$ with stiffness contained in the linear portion $L$.  Function $r(u)$ can either be $r(u)=f(u)-f(u_n)-f'(u_n)(u-u_n)$ or $r(u) = N(u)$ for the partitioned operator $f(u)$.  To obtain a fully exponential EPIRK integrator, functions $\psi_{ij}(z)$ are chosen to be linear combinations of exponential-like functions
\begin{equation}
\label{eqn:psidef}
\psi_{ij}(z) = \sum_{k=1}^{K}p_{ijk}\varphi_{k}(z), \quad \varphi_k(z) = \int_0^1 e^{z(1-\theta)}\frac{\theta^{k-1}}{(k-1)!}d\theta.
\end{equation}
It is also possible to choose some of the these functions to be rational $\psi_{ij}(z)=1/(1-z)$ to derive implicit-exponential integrator \cite{RainwaterTokman, tokmanOberwolfach} which can be used in cases when an efficient preconditioner is available for the full Jacobian $J_n$ or its stiff part $L$.  The main advantages of the EPIRK framework \eqref{eqn:genEPIRK} are the flexibility of the choices for $A_{ij}$ and $\psi_{ij}(z)$ and the degrees of freedom in constructing particular integrators represented by coefficients $\alpha_{ij}$, $\beta_{ij}$ and $g_{ij}$.  In particular, as shown in \cite{tokmanadapt, RainwaterTokman} optimizing coefficients $g_{ij}$ shrinks the spectrum of the corresponding matrix $A_{ij}$ and therefore results in significant computational savings by speeding up convergence of the Krylov projection algorithm to approximate $\psi_{ij}(h_ng_{ij}A_{ij})v$.  A number of numerical studies showed that the EPIRK methods performed well on stiff problems \cite{Tokman2006, Tokman, RainwaterTokman}. However, the derivation of the EPIRK schemes and the general convergence theory were based on classical rather than stiff order conditions in previous publications.  Methods constructed using stiff order conditions are a subset of classically accurate schemes that do not suffer from order reduction for certain classes of problems.  In this paper we demonstrate that the stiff order conditions and convergence theory developed in \cite{Hoch,luanOstermann_ExpBseries, luanOstermann_stiffRK} can be extended to the EPIRK methods. We derive stiff order conditions for the EPIRK methods of nonsplit (or unpartitioned) type, i.e. the most general version of the EPIRK integrators with $A_{ij} = J_n$ and $r(u) = f(u)-f(u_n) - f'(u_n)(u-u_n)$. We present a systematic way to solve the resulting stiff order conditions and show how the flexibility of the EPIRK framework can be utilized to construct particularly efficient schemes. 

 The paper is organized as follows.  Section \ref{sec:stifftheory}  describes how the stiff order conditions and the convergence theory from \cite{luan_oster} can be extended to include the EPIRK methods.  This section also includes an explanation of the differences between the EPIRK framework and the exponential Rosenbrock methods for which the theory was originally developed.  In Section \ref{sec:construction}  we propose a new optimization approach and procedure to solve the stiff order conditions for EPIRK methods to derive a range of efficient fourth- and fifth-order schemes.  In particular, we construct EPIRK methods that can be particularly efficient when used together with an adaptive Krylov-projection algorithm, currently the most general and efficient way to  estimate the exponential matrix function vector products.  Finally, Section \ref{sec:numExper} contains numerical tests that validate the performance of the newly derived methods and demonstrate the relative efficiency of these techniques compared to previously proposed schemes.

%
%
%
%
\section{Stiffly accurate EPIRK methods} 
   \label{sec:stifftheory}

\subsection{EPIRK and exponential Rosenbrock methods} In this paper we focus on the nonsplit, or unpartitioned,  \cite{Tokman, RainwaterTokman} EPIRK schemes for the general problem \eqref{eqn:ivp}. The unpartitioned EPIRK methods are constructed from \eqref{eqn:psidef} by setting $A_{ij} = J_n$ to obtain:
\begin{equation}
\label{eqn:nonsplitEPIRK}
  \begin{aligned}
U_{ni}&= u_0+\alpha_{i1}\psi_{i1}(g_{i1}h_nJ_n)h_nf(u_n)+h_n\sum_{j=2}^{i-1} \alpha_{ij}\psi_{ij}(g_{ij}h_nJ_n)\Delta^{(j-1)}r(u_n),\quad i=2,\dots,s,\\
u_{n+1} &= u_n +\beta_{1} \psi_{s+1 \,1}(g_{s+1 1}h_nJ_n)h_nf(u_n)+h_n\sum_{j=2}^s \beta_j\psi_{s+1\,j}(g_{s+1,j}h_nJ_n)\Delta^{(j-1)}r(u_n)\;\;\
  \end{aligned}
\end{equation}
where $r(u)=f(u)-f(u_n)-J_n(u-u_n)$.  Classical order conditions were derived for EPIRK schemes in \cite{Tokman} and these methods were shown to be efficient for stiff problems \cite{Loffeld, TokmanTranquilli}.  The extension of the theory to stiff order conditions presented below will enable us to construct stiffly accurate EPIRK schemes that can be proved to be convergent even for unbounded operators $J_n$.  

The stiff order conditions and the corresponding convergence theory has been developed in \cite{hochosterexpros, luan_oster,luanOstermann_ExpBseries} for the exponential Rosenbrock methods.  While the original formulation of the exponential Rosenbrock methods was first proposed in \cite{pope}, the full development of this class of integrators, including derivation of the classical and stiff order conditions along with the convergence theory, have not been done until the resurgence of interest in exponential methods over the past several decades \cite{Hoch, hochosterexpros, luan_oster}.  Due to efficiency of implementation and theoretical considerations, in \cite{hochosterexpros} the original formulation of the exponential Rosenbrock methods have been recast in the following form
\begin{equation}
\label{eqn:EXPRBgenform}
 \begin{aligned}
U_{ni}&=u_n+c_ih_n\varphi_1(c_i h_n J_n)f(u_n)+h_n\sum_{j=2}^{i-1}a_{ij}(h_nJ_n)D_{nj}\\
u_{n+1} &= u_n +h_n\varphi_1(h_nJ_n)f(u_n)+h_n\sum_{j=2}^s b_i(h_nJ_n)D_{ni}\;\;\
  \end{aligned}
\end{equation}
where $N_{n}(u)=f(u)-J_n u$, $D_{nj}=N_n(U_{nj})-N_n(u_n)$ and coefficients $a_{ij}(z)$ and $b_{ij}(z)$ are functions comprised of linear combinations $\varphi_k(z)$.  The  structural difference between \eqref{eqn:nonsplitEPIRK} and \eqref{eqn:EXPRBgenform} is in the use of $g_{ij}$ coefficients in \eqref{eqn:nonsplitEPIRK} and in allowing the second term of the right-hand-side in each of the stages to have a more general function $\psi_{i1}(z)$ rather than restricting it to $\psi_{i1}(z)=\varphi_1(z)$ as in \eqref{eqn:EXPRBgenform}. Note that $D_{nj} = r(U_{nj})$.  Any exponential Rosenbrock method can be written in EPIRK form. Any EPIRK method can be written in an extended exponential Rosenbrock form if the differences mentioned above are taken into account.  To make it more straightforward to apply the stiff order conditions theory developed for exponential Rosenbrock methods to EPIRK integrators we re-write \eqref{eqn:nonsplitEPIRK} in the extended exponential Rosenbrock form using the expansion
\begin{equation}\label{eqn:forwardDiffvec}
\Delta^{(n)}r(u_n)=\sum_{i=2}^{n+1} \left( \begin{array}{c} n-1 \\ i-1 \end{array}\right)(-1)^{n-i-2}r(U_{ni})=\sum_{i=2}^{n} \left( \begin{array}{c} n-1 \\ i-1 \end{array}\right)(-1)^{n-i-2}r(U_{ni}),
\end{equation}
and collecting the terms corresponding to each $r(U_{ni})$ in every stage. Then (\ref{eqn:nonsplitEPIRK}) can be expressed as
\begin{equation}
\label{eqn:EPIRK}
  \begin{aligned}
U_{ni}&= u_n+\alpha_{i1}\psi_{i1}(g_{i1}h_nJ_n)h_nf(u_n)+h_n\sum_{j=2}^{i-1} a_{ij}(h_nJ_n)r(U_{nj}),\quad i=2,\dots,s,\\
u_{n+1} &= u_n +\beta_1\psi_{\!s+1 1}(g_{s+1 1}h_nJ_n)h_nf(u_n)+h_n\sum_{j=2}^{s} b_{j}(h_nJ_n)r(U_{nj})\;\;\
 \end{aligned}
\end{equation}
where 
\begin{equation}
\label{eqn:a_coefficients}
a_{ij}(z)=\sum_{k=j}^{i-1}\left( \begin{array}{c} k-1 \\ i-1 \end{array}\right)(-1)^{k-i-2}\alpha_{ik}\psi_{ik}(g_{ik}z)\quad \textrm{and} \quad b_{j}(z)=\sum_{k=j}^{s}\left( \begin{array}{c} k-1 \\ s \end{array}\right)(-1)^{k-s-3}\beta_{k}\psi_{\!s+1 k}(g_{\!s+1 k}z)
\end{equation}
We additionally define $\psi_{i1}(z) = \sum_{k=1}^s p_{i1k}\varphi_k(z)$. We now incorporated the $g_{ij}$ coefficients into the definitions of $a_{ij}(z)$ and $b_{ij}(z)$ and extended the second term of the right-hand-sides in stages to general $\psi_{ij}(z)$ functions.  Later we will show how these generalizations of the exponential Rosenbrock methods to EPIRK framework offer added flexibility that allows for derivation of more efficient methods.  To illustrate this reformulation we consider a simple example of a three-stage EPIRK method 
\begin{equation}
\label{eqn:3stageEPIRKform}
\begin{aligned}
U_{n2}&= u_n+\alpha_{21}\psi_{21}(g_{21}h_nJ_n)h_nf(u_n)\\
U_{n3}&= u_n+\alpha_{31}\psi_{31}(g_{31}h_nJ_n)h_nf(u_n) + h_n\alpha_{32}\psi_{32}(g_{32}h_nJ_n)\underbrace{\Delta r(u_n)}_{r(U_{n2})}\\
u_{n+1} &= u_n + \beta_1\psi_{4 1}(g_{4 1}h_nJ_n)h_nf(u_n)+h_n\beta_2 \psi_{42}(g_{42}h_nJ_n)\underbrace{\Delta r(u_n)}_{r(U_{n2})}+\beta_3\psi_{43}(g_{43}h_nJ_n)\underbrace{\Delta^2 r(u_n)}_{{r(U_{n3})-2r(U_{n2})}}.
  \end{aligned}
\end{equation}
This EPIRK method can be re-written in an extended exponential Rosenbrock way as 
\begin{equation}
\label{eqn:3stageEPIRKinEXPRBform}
  \begin{aligned}
U_{n2}&= u_n+\alpha_{21}\psi_{21}(g_{21}h_nJ_n)h_nf(u_n)\\
U_{n3}&= u_n+\alpha_{31}\psi_{31}(g_{31}h_nJ_n)h_nf(u_n) + h_n\underbrace{\alpha_{32}\psi_{32}(g_{32}h_nJ_n)}_{{\textcolor{black}{a_{32}(h_nJ_n)}}}r(U_{n2})\\
u_{n+1} &= u_n + \beta_1\psi_{4 1}(g_{4 1}h_nJ_n)h_nf(u_n)+h_n\underbrace{\left(\beta_3 \psi_{42}(g_{42}h_nJ_n)-2\beta_3\psi_{43}(g_{43}h_nJ_n)\right)}_{{\textcolor{black}{b_2(h_nJ_n)}}} r(U_{n2}) +h_n\underbrace{\beta_3 \psi_{43}(g_{43}h_nJ_n)}_{{\textcolor{black}{b_3(h_nJ_n)}}} r(U_{n3}).
  \end{aligned}.
  \end{equation}
Due to the close relationship between EPIRK and exponential Rosenbrock methods outlined above, most of the theory from \cite{luan_oster} applies to EPIRK directly.  But this generalization has to be handled with care since additional results are needed to account for the more general form of EPIRK schemes. Below we outline the theory for the reader's convenience and present more detail in places where the differences between EPIRK and exponential Rosenbrock methods result in distinct expressions and, ultimately, modified stiff order conditions. 

\subsection{Analytical framework.} \label{subsec:analyticFramework} As in \cite{luan_oster} the analysis is based on the theory of strongly continuous semigroups in a Banach space $\X$ with norm $\norm{\cdot}$. For the reader's convenience we state the assumptions from \cite{luan_oster} that form the base for the stiff order conditions and the convergence theory; we also outline the main ideas of the theory. For our analysis we consider \eqref{eqn:ivp} written in a linearized form
\begin{equation}
\label{eqn:linearizedForm}
u'(t)=f(u(t))=Lu(t)+N(u(t)),\qquad u(t_0)=u_0,
\end{equation}
with Jacobian 
\begin{equation}
\label{eqn:jacobian}
J=J(u)=\frac{\partial f}{\partial u}(u) = L+\frac{\partial N}{\partial u}(u).
\end{equation}
The following two assumptions are then made about operators $L$ and $N(u)$:

  \begin{assumption2}{1 (\cite{luan_oster})}\label{assump:AgenOfC0semigroup}
  The linear operator $L$ is the generator of a strongly continuous semigroup $e^{tL}$ in $\X$.
  \end{assumption2}

 \begin{assumption2}{2 (\cite{luan_oster})}\label{assump:SolutionAndNonlin}
We assume that (\ref{eqn:ivp}) possesses a sufficiently smooth solution $u:[0,T]\to \X$ with derivatives in $\X$ and that the nonlinearity $N:\X\to \X$ is sufficiently often Fr\'{e}chet differentiable in a strip along the exact solution.
 \end{assumption2}
Given these assumptions it can be shown that the Jacobian $J$ in \eqref{eqn:jacobian} is the generator of a strongly continuous semigroup (\cite{pazy},Chap. 3.1).  This implies that there exist constants $C$ and $\omega$ such that 
\begin{equation}
\label{eqn:expBounded}
\norm{e^{tJ}}_{\X\leftarrow\X}\leq Ce^{\omega t},\quad t\geq 0
\end{equation}
holds uniformly in a neighborhood of the exact solution.	Furthermore, it can be concluded from this result that $\varphi_k(h_nJ)$ and subsequently their linear combinations $a_{ij}(h_nJ)$ \& $b_{i}(h_nJ)$, are bounded operators.  Assumption 2 also implies that the Jacobian (\ref{eqn:jacobian}) satisfies the Lipschitz condition
\begin{equation}
\label{eqn:lipchitz_jac}
\norm{J(u)-J(v)}_{\X\leftarrow \X} = \norm{\frac{\partial N}{\partial u}(u) -\frac{\partial N}{\partial u}(v)}_{\X\leftarrow \X} \leq L\norm{u-v}
\end{equation}
in a neighborhood of the exact solution.

Note that problems with homogeneous or no-flow boundary conditions satisfy Assumptions 1 and 2.  In general, non-homogeneous  boundary conditions (including time-dependent) do not necessarily satisfy Assumption 1. As a simple example, consider a semigroup $T(t)=e^{tA}$ defined over the Banach space $\X=L^2(\R)$ of square-integrable functions.   Assumption 1 requires the semigroup to be strongly continuous. By the Hille-Yoshida theorem the necessary condition for the semigroup to be strongly continuous is that the domain $D(A)$ of the infinitesimal generator of the semigroup $A$ is dense in $\X$.  Thus, it is necessary to restrict the domain $D(A)$ in order to ensure that the semigroup $e^{tA}$ is strongly continuous.  If we assume that the solutions $w(t)$ of $w'(t)=Aw$ belong to a subspace $C_0^\infty(\R) \in \X$ it is possible to prove that the semigroup $e^{tA}$ is strongly continuous \cite{pazy} .  This is not necessarily the case for the subspace of functions $w(t)$ with non-homogeneous boundary conditions.  The analysis quantifying how much order reduction one can expect for the problems with non-homogeneous boundary conditions has been done for standard Runge-Kutta and Rosenbrock methods in \cite{OsterRoche1, OsterRoche2}. Developing similar fractional order convergence theory for exponential methods or developing exponential schemes which do not exhibit order reduction for problems with non-homogeneous boundary conditions are non-trivial tasks which we plan to address in our future research.  In this paper we address this issue by using numerical examples to illustrate how much order reduction one can expect if the boundary conditions are non-homogeneous.


The derivation of the stiff order conditions and the convergence theory for problems which satisfy Assumptions 1 and 2 then proceeds as follows. 
\subsection{Local error and stiff order conditions for EPIRK methods.}  Derivation of the stiff order conditions and proof of convergence require analysis of the local error and construction of expressions for the approximate and exact solutions.  To accommodate the difference between EPIRK and exponential Rosenbrock methods derivation of expressions for the numerical solution in \cite{luan_oster} has to be adjusted.  Thus, we choose to present this derivation in more detail and simply restate other results from \cite{luan_oster} that are used without alternations.  The key idea in the convergence theory is to express the error in terms of operators that are bounded. While the Jacobian operator $J_n$ can potentially be unbounded, expressions involving only derivatives of the solution $u_n^{(k)}$ and/or the nonlinearity $\partial^k N/\partial ^k u$ are bounded given Assumptions 1 and 2.  The following formulas that connect these two groups of operators help make these transitions. Consider linearization of \eqref{eqn:ivp} along the exact solution $\tilde{u}_n=u(t_n)$ to get
		\begin{equation}\label{eqn:linearizeAlongExact}
			u'(t)=\Jt u(t) + \nt (u(t)) 
		\end{equation}
		where
		\begin{equation}
		\Jt=\frac{\partial f}{\partial u}(\ut)=L+\frac{\partial N}{\partial u}(\ut), \quad \nt(u)=f(u)-\Jt u.    \label{eqn:J_tilde_g_tilde}
		\end{equation}
From (\ref{eqn:J_tilde_g_tilde}), we obtain 
		\begin{equation}
			\nt(u)=N(u)-\frac{\partial N}{\partial u}(\ut)u \quad \textrm{and} \quad \frac{\partial \nt}{\partial u}(\ut)=0 \label{eqn:identities_of_g_tilde}.
		\end{equation}
Using these identities along with the repeated differentiation of \eqref{eqn:linearizeAlongExact} we get:
 \begin{equation}
	 	\Jt \yt'=\yt'' \qquad  \& \qquad \Jt\yt''=\yt^{(3)}-\frac{\partial^2 \nt}{\partial u^2}(\yt)(\yt',\yt') \label{eqn:J_mult_identities},
	 \end{equation}	
	 where $\ut^{(k)}$ denotes the $k$th derivative of the exact solution of (\ref{eqn:linearizeAlongExact}). More generally, we have 
	 	\begin{equation}
	\Jt u^{(k)}(t)=u^{(k+1)}(t)-\frac{d^k}{dt^k}\nt(u(t))
	\end{equation}
	 which shows that $\Jt u^{(k)}$ is bounded (due to Assumption 2).  This result and identities (\ref{eqn:identities_of_g_tilde}) are key to deriving the stiff order conditions.

First following the procedure in \cite{luan_oster} we carry out one integration step with \eqref{eqn:EPIRK} with exact solution $\tilde{u}_n$ used for the initial value to express the numerical solution as 
\begin{equation}
		\label{eqn:EPIRKalongExact}
		  \begin{aligned}
		\wh{U}_{ni}&= \ut+\alpha_{i1}\psi_{i1}(g_{i1}h_n\Jt)h_nf(\ut)+h_n\sum_{j=2}^{i-1} a_{ij}(h_n\Jt)\rh_{nj},\quad i=2,\dots,s,\\
		\wh{u}_{n+1} &= \ut +\beta_1 \psi_{s+1 \,1}(g_{s+1 1}h_n\Jt)h_nf(\ut)+h_n\sum_{i=2}^{s} b_{i}(h_n\Jt)\rh_{ni}\;\;\
		  \end{aligned},
		\end{equation}
		with \begin{equation}\label{eqn:rhat}
		\rh_{ni}=\nt(\wh{U}_{ni})-\nt(\ut).
		\end{equation}
We now begin computing the Taylor expansion of (\ref{eqn:EPIRKalongExact}) by first calculating $\rh_{ni}$ as a Taylor series around $\ut$.  Using (\ref{eqn:identities_of_g_tilde}) we obtain 
\begin{equation}
\rh_{ni}=\sum_{q=2}^k \frac{h_n^q}{q!}\frac{\partial^q \nt}{\partial u^q}(\ut)\underbrace{(V_i,\dots,V_i)}_{q \textrm{ times}}+\mathcal{R}_{ki} \label{eqn:Expan_of_r}
\end{equation}
with 
\begin{equation} \label{eqn:V_i}
V_{i}=\frac{1}{h_n}\left(\Uh_{ni}-\ut\right)=\alpha_{i1}\psi_{i1}(g_{i1}h_n \Jt)+\sum_{j=2}^{i-1}a_{ij}(h_n \Jt)\rh_{nj} 
\end{equation}
and the remainder 
$$\mathcal{R}_{ki}=h_n^{k+1}\int_0^1 \frac{(1-\theta)^k}{k!}\frac{\partial^{k+1}\nt }{\partial u^{k+1}}(\ut +\theta h_n V_i)(\underbrace{V_i,\dots,V_i}_{k+1 \textrm{ times}})d\theta$$
which is bounded and of order $\mathcal{R}_{ki}=\O(h_n^{k+1})$ by Assumptions 1 and 2.  Note the expression \eqref{eqn:V_i} corresponds to the formula (3.7) in \cite{luan_oster} with functional coefficients $\psi_{i1}$, $a_{ij}$ generalized. 
By substituting (\ref{eqn:Expan_of_r}) into (\ref{eqn:EPIRKalongExact}) we obtain 
\begin{equation}
\label{eqn:NumericalExpan1}
\hat{u}_{n+1}=\ut+\beta_1\psi_{s+1 \,1}(g_{s+1 1} h_n \Jt)h_n f(\ut)+\sum_{i=2}^{s} b_i(h_n \Jt)\sum_{q=2}^k \frac{h_n^{q+1}}{q!}\frac{\partial^q \nt}{\partial u^q}\underbrace{(V_i,\dots,V_i)}_{q \textrm{ times}}+\O(h_n^{k+2}) 
\end{equation}
The following lemmas represent analogues of lemmas 3.1 and 3.2 in \cite{luan_oster}. These results allow us to obtain the expansion of (\ref{eqn:NumericalExpan1}) avoiding terms containing powers of the possibly unbounded operator $\Jt$.
\vspace{-.05in}
\begin{lem}
\label{lemma:ShiftToHigherOrderPhi}
Under Assumptions 1 and 2, we have for all $t\geq 0$
\begin{equation}
\varphi_{k}(th_n\Jt)f(\ut)= \frac{1}{k!}\ut'+h_n \frac{t}{(k+1)!}\ut'' + t^2h_n^2\frac{1}{(k+2)!}\left(\ut^{(3)}-(k+2)!\varphi_{k+2}(th_n\Jt)\frac{\partial^2 \nt }{\partial u^2}(\ut)(\ut',\ut')\right)+\O(t^3h_n^3),
\label{eqn:phi_k_expansion}\end{equation}
and furthermore,
\begin{equation}
\psi_{i1}(t\Jt)f(\ut)=\mathbb{P}_{i1} \ut'+t \mathbb{P}_{i2}\ut'' +t^2\left( \mathbb{P}_{i3} \ut^{(3)}-\left(\sum_{k=1}^s p_{i1k}\varphi_{k+2}(t \Jt)\right)\frac{\partial^2 \nt}{\partial u^2}(\ut )(\ut' , \ut ')\right)+\O(t^3)
\label{eqn:psi_1_expansion}\end{equation}
where 
\begin{equation}
\label{eqn:Pcoeff}
\mathbb{P}_{i1}=\sum_{k=1}^s \frac{p_{i1k}}{k!}, \quad \mathbb{P}_{i2}=\sum_{k=1}^s \frac{p_{i1k}}{(k+1)!}, \quad \mathbb{P}_{i3}=\sum_{k=1}^s \frac{p_{i1k}}{(k+2)!}
\end{equation}
\end{lem}
\begin{proof}
By using $f(\ut)=\ut '$, the recurrence relation 
\begin{equation}\label{eqn:phiRecurrenceRelation}
\varphi_k(z)=z\varphi_{k+1}(z)+1/k!
\end{equation} 
and formulas (\ref{eqn:J_mult_identities}) we have
\begin{eqnarray*}
\varphi_k(th_n\Jt)\ut' &=& \left(\frac{1}{k!}+th_n\Jt\varphi_{k+1}(th_n\Jt)\right)\ut '\\
&=& \frac{1}{k!}\ut'+th_n\varphi_{k+1}(th_n\Jt)\Jt \ut'\\
&=& \frac{1}{k!}\ut'+th_n\left(\frac{1}{(k+1)!}+th_n\Jt\varphi_{k+2}(th_n\Jt) \right)\ut''\\
&=&  \frac{1}{k!}\ut'+h_n \frac{t}{(k+1)!}\ut''+t^2h_n^2\varphi_{k+2}(th_n\Jt)\left(\ut^{(3)}-\frac{\partial^2 \nt }{\partial u^2}(\ut)(\ut',\ut')\right)\\
&= & \frac{1}{k!}\ut'+h_n \frac{t}{(k+1)!}\ut'' + t^2h_n^2\left(\frac{1}{(k+2)!}+th_n\Jt\varphi_{k+3}(th_n\Jt)\right)\ut^{(3)}-t^2h_n^2\varphi_{k+2}(th_n\Jt)\frac{\partial^2 \nt }{\partial u^2}(\ut)(\ut',\ut')\\
&=&\frac{1}{k!}\ut'+h_n \frac{t}{(k+1)!}\ut'' + t^2h_n^2\frac{1}{(k+2)!}\ut^{(3)}+t^3h_n^3\varphi_{k+3}(th_n\Jt)\Jt\ut^{(3)}-t^2h_n^2\varphi_{k+2}(th_n\Jt)\frac{\partial^2 \nt }{\partial u^2}(\ut)(\ut',\ut')\\
&=& \frac{1}{k!}\ut'+h_n \frac{t}{(k+1)!}\ut'' + t^2h_n^2\frac{1}{(k+2)!}\left(\ut^{(3)}-(k+2)!\varphi_{k+2}(th_n\Jt)\frac{\partial^2 \nt }{\partial u^2}(\ut)(\ut',\ut')\right)+\O(t^3h_n^3),
\end{eqnarray*}
where the last equality holds due to the boundedness of $\Jt\ut^{(3)}$. Equation \eqref{eqn:psi_1_expansion} follows directly from using these expansions for $\varphi_k(z)$ in $\psi_{i1}(z)=\sum_{k=1}^s p_{i1k}\varphi_{k}(z)$.
\end{proof}
\begin{lem}\label{lemma:V_i_expansion}
Under Assumptions 1 and 2, we have
\begin{equation}
V_i = \alpha_{i1}\mathbb{P}_{i1} \ut ' + h_n \alpha_{i1} g_{i1} \mathbb{P}_{i2} \ut''+h_n^2 \alpha_{i1}g_{i1}^2 \mathbb{P}_{i3} \ut^{(3)}+h_n^2 \Psi_{i}(h_n \Jt)\frac{\partial^2 \nt}{\partial y^2}(\ut)(\ut ' , \ut ' ) +\O(h_n^3)\label{eqn:V_i_expasion}\end{equation} 
where
\begin{equation}
\label{eqn:BigPsi}
\Psi_i(z)=\frac{1}{2!} \sum_{j=2}^{i-1}\alpha_{j1}^2\P_{j1}^2 a_{ij}(z)-g_{i1}^2\left( \sum_{k=1}^s p_{i1k}\varphi_{k+2}(g_{i1} z)\right)
\end{equation}
\end{lem}
\begin{proof}
Inserting (\ref{eqn:V_i}) into (\ref{eqn:Expan_of_r}) for $k=2$ and using Lemma \ref{lemma:ShiftToHigherOrderPhi} with $t=g_{1j}h_n$ gives
\begin{eqnarray*}
\rh_{nj}&=& \frac{h_n^2}{2!}\frac{\partial^2 \nt }{\partial u^2}(\ut)(V_j,V_j)+\O(h_n^3)\\
&=& \frac{h_n^2}{2!}\frac{\partial^2 \nt }{\partial u^2}(\ut)\left(\alpha_{j1}\varphi_1(g_{j1}h_n\Jt)f(\ut),\alpha_{j1}\varphi_1(g_{j1}h_n\Jt)f(\ut)\right)+\O(h_n^3)\\
&=& \frac{h_n^2}{2!}\frac{\partial^2 \nt }{\partial u^2}(\ut)\left(\alpha_{j1}\mathbb{P}_{j1}\ut',\alpha_{j1}\mathbb{P}_{j1}\ut'\right)+\O(h_n^3)\\
&=& \frac{h_n^2}{2!}\alpha_{j1}^2\mathbb{P}_{j1}^2\frac{\partial^2 \nt }{\partial u^2}(\ut)\left(\ut',\ut'\right)+\O(h_n^3)\\
\end{eqnarray*}
Substituting this into (\ref{eqn:V_i}) we obtain
\begin{equation*}
V_i=\alpha_{i1}\psi_{i1}(g_{i1}h_n \Jt)+\frac{h_n^2}{2!}\sum_{j=2}^{i-1}a_{ij}(h_n \Jt)\alpha_{j1}^2\mathbb{P}_{j1}^2\frac{\partial^2 \nt }{\partial u^2}(\ut)\left(\ut',\ut'\right) 
\end{equation*}
which yields the desired result if Lemma \ref{lemma:ShiftToHigherOrderPhi} is used:
\begin{eqnarray*}
V_i&=&\alpha_{i1}\mathbb{P}_{i1} \ut'+\alpha_{i1}h_ng_{i1} \mathbb{P}_{i2}\ut'' +\alpha_{i1}h_n^2g_{i1}^2\left( \mathbb{P}_{i3} \ut^{(3)}-\left(\sum_{k=1}^s p_{i1k}\varphi_{k+2}(t \Jt)\right)\frac{\partial^2 \nt}{\partial u^2}(\ut )(\ut' , \ut ')\right) \\
&&\quad+\frac{h_n^2}{2!}\sum_{j=2}^{i-1}a_{ij}(h_n \Jt)\alpha_{j1}^2\mathbb{P}_{j1}^2\frac{\partial^2 \nt }{\partial u^2}(\ut)\left(\ut',\ut'\right) +\O(h_n^3)\\
&=& \alpha_{i1}\mathbb{P}_{i1} \ut'+h_n\alpha_{i1}g_{i1} \mathbb{P}_{i2}\ut'' +h_n^2\alpha_{i1}g_{i1}^2\mathbb{P}_{i3} \ut^{(3)}\\
&&\quad+h_n^2\left(-g_{i1}^2\sum_{k=1}^s\alpha_{i1} p_{i1k}\varphi_{k+2}(t \Jt)+\frac{1}{2!}\sum_{j=2}^{i-1}a_{ij}(h_n \Jt)\alpha_{j1}^2\mathbb{P}_{j1}^2\right)\frac{\partial^2 \nt }{\partial u^2}(\ut)\left(\ut',\ut'\right) +\O(h_n^3).\\
\end{eqnarray*}
\end{proof}
We now use these expansions of $V_i$ to obtain
\begin{eqnarray*}
\frac{\partial^4 \nt}{\partial u^4}(\ut)(V_i,V_i,V_i,V_i)&=& \alpha_{i1}^4\P_{i1}^4 \frac{\partial^4 \nt}{\partial u^4}(\ut)(\ut',\ut',\ut',\ut')+\O(h_n)\\
\frac{\partial^3 \nt}{\partial u^3}(\ut)(V_i,V_i,V_i) \quad &=& \alpha_{i1}^3\P_{i1}^3 \frac{\partial^3 \nt}{\partial u^3}(\ut)(\ut ',\ut ',\ut ')+ 3 h_n \alpha_{i1}^3 g_{i1} \P_{i1}^2 \P_{i2}\frac{\partial^3 \nt}{\partial u^3}(\ut)(\ut ',\ut ',\ut '')+\O(h_n^2)\\
\frac{\partial^2 \nt}{\partial u^2}(\ut)(V_i,V_i) \quad \quad &=& \alpha_{i1}^2\P_{i1}^2 \frac{\partial^2 \nt}{\partial u^2}(\ut)(\ut', \ut')
+2 h_n \alpha_{i1}^2 g_{i1} \P_{i1}\P_{i2} \frac{\partial^2 \nt}{\partial u^2}(\ut)(\ut', \ut'') 
\\
&&\quad+2 h_n^2 \alpha_{i1}^2 g_{i1}^2 \P_{i1}\P_{i3} \frac{\partial^2 \nt}{\partial u^2}(\ut)(\ut', \ut^{(3)})+ h_n^2 \alpha_{i1}^2 g_{i1}^2 \P_{i2}^2 \frac{\partial^2 \nt}{\partial u^2}(\ut)(\ut'', \ut'')\\
&& \quad +2 h_n^2 \alpha_{i1}^2\P_{i1} \frac{\partial^2 \nt}{\partial u^2}(\ut)\left( \ut', \Psi_i(h_n \Jt)\frac{\partial^2 \nt}{\partial u^2}(\ut)(\ut', \ut')\right) +\O(h_n^3)\end{eqnarray*}
and insert these expressions into (\ref{eqn:NumericalExpan1}) with $k=4$ to get 
%
\begin{eqnarray}
\hat{u}_{n+1}&=&\ut+\beta_1\psi_{s+1 \,1}(g_{s+1 1} h_n \Jt)h_n f(\ut)+\\
&&\qquad +\sum_{i=2}^{s} b_i(h_n \Jt)\left( \frac{h_n^3}{2} \frac{\partial^2 \nt}{\partial u^2}(\ut)(V_i,V_i) +\frac{h_n^4}{3!} \frac{\partial^3 \nt}{\partial u^3}(\ut)(V_i,V_i,V_i) +\frac{h_n^5}{4!} \frac{\partial^4 \nt}{\partial u^4}(\ut)(V_i,V_i,V_i,V_i)\right)+\O(h_n^6)\nonumber
\end{eqnarray}
and ultimately
\begin{eqnarray}
\hat{u}_{n+1}&=& \ut+  \beta_1\psi_{s+1 \,1}(g_{s+1 1} h_n \Jt)h_nf(\ut)+h_n^3 \left( \sum_{i=2}^{s} b_i(h_n \Jt)\frac{\alpha_{i1}^2\P_{i1}^2}{2!} \right) \frac{\partial^2 \nt}{\partial y^2}(\ut)(\ut', \ut')\nonumber\\
&& \quad + h_n^4\left(\sum_{i=2}^{s} b_i(h_n \Jt)\frac{2}{2!}\alpha_{i1}^2g_{i1}\P_{i1}\P_{i2}\right)\frac{\partial^2 \nt}{\partial y^2}(\ut)(\ut', \ut'') \nonumber\\
&& \quad + h_n^4 \left( \sum_{i=2}^{s} b_i(h_n \Jt)\frac{1}{3!}\alpha_{i1}^3\P_{i1}^3\right)\frac{\partial^3 \nt}{\partial u^3}(\ut)(\ut ',\ut ',\ut ')\nonumber\\
&&\quad+ h_n^5\left[\left( \sum_{i=2}^{s} b_i(h_n \Jt)\frac{2}{2!} g_{i1}^2 \alpha_{i1}^2\P_{i1}\P_{i3} \right)\frac{\partial^2 \nt}{\partial y^2}(\ut)(\ut', \ut^{(3)})\right.
\nonumber\\
&& \left.\quad + \left( \sum_{i=2}^{s} b_i(h_n \Jt)\frac{1}{2!}\alpha_{i1}^2 g_{i1}^2\P_{i2}^2 \right)\frac{\partial^2 \nt}{\partial y^2}(\ut)(\ut'', \ut'') \right.\nonumber\\
&&\left.\quad +\left(\sum_{i=2}^{s} b_i(h_n \Jt)\frac{3}{3!} g_{i1}\alpha_{i1}^3\P_{i1}^2\P_{i2}  \right)\frac{\partial^3 \nt}{\partial u^3}(\ut)(\ut ',\ut ',\ut '')\right.
\nonumber\\
&& \quad \left.+\left( \sum_{i=2}^{s} b_i(h_n \Jt)\frac{1}{4!}\alpha_{i1}^4\P_{i1}^4\right)\frac{\partial^4 \nt}{\partial u^4}(\ut)(\ut',\ut',\ut',\ut')\right]\nonumber \\
&&\quad +h_n^5 \left( \sum_{i=2}^{s} b_i(h_n \Jt)\frac{2}{2!}\alpha_{i1}^2\P_{i1}\right)\frac{\partial^2 \nt}{\partial u^2}(\ut)\left( \ut', \Psi_i(h_n \Jt)\frac{\partial^2 \nt}{\partial u^2}(\ut)(\ut', \ut')\right)\label{eqn:Numerical_Expansion_final} .
\end{eqnarray}
The expression for the Taylor expansion of the exact solution we borrow directly from \cite{luan_oster}:
  \begin{equation} \label{eqn:TaulorExpanExact}
  \begin{aligned} \tilde{u}_{n+1}&=\ut+h_n\varphi_1(h_n \Jt)f(\ut)+h_n^3 \varphi_3(h_n\Jt)\frac{\partial^2 \nt}{\partial u^2}(\ut)(\ut', \ut') \\
   & \quad + h_n^4 \varphi_4(h_n \Jt) \left[ 3\frac{\partial^2 \nt}{\partial u^2}(\ut)(\ut', \ut'')+\frac{\partial^3 \nt}{\partial u^3}(\ut)(\ut ',\ut ',\ut ')\right] \\
   & \quad + h_n^5 \varphi_5(h_n \Jt)\left[ 4 \frac{\partial^2 \nt}{\partial u^2}(\ut)(\ut', \ut^{(3)})+3\frac{\partial^2 \nt}{\partial u^2}(\ut)(\ut'', \ut'')\right. \\
    &\quad\left.+6\frac{\partial^3 \nt}{\partial u^3}(\ut)(\ut ',\ut ',\ut '')+\frac{\partial^4 \nt}{\partial u^4}(\ut)(\ut',\ut',\ut',\ut')\right]+\O(h_n^6) 
   \end{aligned}
   \end{equation}
Now subtracting (\ref{eqn:TaulorExpanExact}) from (\ref{eqn:Numerical_Expansion_final}) we obtain the expression for the local error $\tilde{e}_{n+1}=\hat{u}_{n+1}-\tilde{u}_{n+1} $ in the form:
 \begin{eqnarray}
   \label{eqn:localError}
   \tilde{e}_{n+1} &&= h_n\left(\beta_1\psi_{s+1 \,1}(g_{s+1 1}h_n\Jt)-\varphi_1(h_n\Jt)\right)f(\ut)\nonumber \\
  &&+ h_n^3 \left(\sum_{i=2}^{s} b_i(h_n \Jt)\frac{\alpha_{i1}^2\P_{i1}^2}{2!} -\varphi_3(h_n\Jt)\right)\frac{\partial^2 \nt}{\partial y^2}(\ut)(\ut', \ut') \\ \nonumber
 && +h_n^4 \left(\sum_{i=2}^{s} b_i(h_n \Jt)\alpha_{i1}^2g_{i1}\P_{i1}\P_{i2}-3\varphi_4(h_n\Jt)\right)\frac{\partial^2 \nt}{\partial y^2}(\ut)(\ut', \ut'') \\ \nonumber
&& + h_n^4 \left(\sum_{i=2}^{s} \frac{1}{3!}b_i(h_n \Jt)\alpha_{i1}^{3}\mathbb{P}_{i1}^3-\varphi_4(h_n\Jt)\right)\frac{\partial^3 \nt}{\partial u^3}(\ut)(\ut ',\ut ',\ut ')\\ \nonumber
  && +h_n^5\left( \sum_{i=2}^{s} b_i(h_n \Jt)\frac{2}{2!} g_{i1}^2 \alpha_{i1}^2\P_{i1}\P_{i3}-4\varphi_5(h_n\Jt) \right)\frac{\partial^2 \nt}{\partial y^2}(\ut)(\ut', \ut^{(3)})\\ \nonumber
&&+  h_n^5\left( \sum_{i=2}^{s} b_i(h_n \Jt)\frac{1}{2!}\alpha_{i1}^2 g_{i1}^2\P_{i2}^2 -3\varphi_5(h_n\Jt)\right)\frac{\partial^2 \nt}{\partial y^2}(\ut)(\ut'', \ut'')\\ \nonumber
  && +h_n^5\left(\sum_{i=2}^{s} b_i(h_n \Jt)\frac{1}{2!} g_{i1}\alpha_{i1}^3\P_{i1}^2\P_{i2} -6\varphi_5(h_n\Jt) \right)\frac{\partial^3 \nt}{\partial u^3}(\ut)(\ut ',\ut ',\ut '')\\ \nonumber
 &&+h_n^5 \left( \sum_{i=2}^{s} b_i(h_n \Jt)\frac{1}{4!}\alpha_{i1}^4\P_{i1}^4-\varphi_5(h_n\Jt)\right)\frac{\partial^4 \nt}{\partial u^4}(\ut)(\ut',\ut',\ut',\ut')\\ \nonumber
 &&+h_n^5\left( \sum_{i=2}^{s} b_i(h_n \Jt)\frac{2}{2!}\alpha_{i1}^2\P_{i1}\right)\frac{\partial^2 \nt}{\partial u^2}(\ut)\left( \ut', \Psi_i(h_n \Jt)\frac{\partial^2 \nt}{\partial u^2}(\ut)(\ut', \ut')\right)\nonumber,
   \end{eqnarray}
with $\Psi_{i}(z)$ given by (\ref{eqn:BigPsi}).  From (\ref{eqn:localError}) we can easily read off the stiff order conditions  by ensuring the terms for a given order (three,four, or five) vanish.  These conditions are given in Table \ref{table:StiffOC}.  To eliminate the first order term in our error we must have $\beta_1\psi_{s+1 \,1}(g_{s+1 1}z)=\varphi_1(z)$. If this condition is satisfied, the resulting method will be of stiff order two. Throughout the rest of the paper we will set $\beta_1=g_{s+1}=1$ and $\psi_{s+1\, 1}=\varphi_1$. In Section \ref{sec:construction} we will show how more flexibility in the stiff order conditions for EPIRK compared to the exponential Rosenbrock methods leads to construction of more efficient techniques. 
\begin{table}[htb]
\begin{center}
\begin{tabular}{lll}
\hline 
Ref. label & \hspace{10pt}Order condition & Order \\ 
\hline 
C1   & \hspace{10pt} $\sum_{i=2}^{s} b_i(Z)\alpha_{i1}^2\P_{i1}^2 =2\varphi_3(Z)$ \hspace{10pt}& 3  \\ 

C2 &\hspace{10pt} $\sum_{i=2}^{s} b_i(Z)\alpha_{i1}^2g_{i1}\P_{i1}\P_{i2}=3\varphi_4(Z)$ & 4 \\ 

C3 &\hspace{10pt} $\sum_{i=2}^{s} b_i(Z)\alpha_{i1}^3\mathbb{P}_{i1}^3=3!\varphi_4(Z)$& 4 \\ 
 
C4 &\hspace{10pt} $\sum_{i=2}^{s} b_i(Z) g_{i1}^2 \alpha_{i1}^2\P_{i1}\P_{i3}=4\varphi_5(Z)$ & 5\\ 

C5 &\hspace{10pt} $ \sum_{i=2}^{s} b_i(Z)\alpha_{i1}^2 g_{i1}^2\P_{i2}^2 =3!\varphi_5(Z)$ & 5\\ 

C6 &\hspace{10pt} $\sum_{i=2}^{s} b_i(Z) g_{i1}\alpha_{i1}^3\P_{i1}^2\P_{i2} =4!\varphi_5(Z)$ & 5 \\ 

C7 &\hspace{10pt} $\sum_{i=2}^{s} b_i(Z)\alpha_{i1}^4\P_{i1}^4=4!\varphi_5(Z)$ & 5 \\ 

C8&\hspace{10pt} $ \sum_{i=2}^{s} b_i(Z)\alpha_{i1}^2\P_{i1}K\Psi_i(Z)=0$ & 5 \\ 
\hline 
\end{tabular}  
\end{center}
\caption{Stiff order conditions for EPIRK methods.  Note that $Z$ and $K$ are arbitrary square matrices and $\Psi_i$ is given by (\ref{eqn:BigPsi})}
\label{table:StiffOC}
\end{table}
\subsection{Convergence.}\label{sec:convergence} 

The majority of the convergence proof presented in \cite{luan_oster} for the exponential Rosenbrock (EXPRB) methods can be applied directly to the EPIRK schemes. The only exception is Lemma 4.5 in \cite{luan_oster}. In order to obtain the same result as in this lemma we need to use additional assumption on the coefficients.  Assumption 3 given below allows for less restrictive choice of the coefficients compared to EXPRB methods but enables us to proceed with the convergence proof in the same way as in \cite{luan_oster}: 

\noindent
 \begin{assumption2}{3}\label{assump:CoeffRestrictionConvergence}
 Suppose the coefficients of an EPIRK scheme satisfy one of the following for each $i$ and all $k$
  \begin{equation}
\label{eqn:CoeffRestrictionConvergence}
\alpha_{i1}p_{i1k} = g_{i1} \quad \textrm{ or } \quad \alpha_{i1}=g_{i1} \quad \textrm{ or } \quad p_{i1k}=g_{i1}.
\end{equation}
 \end{assumption2}
 %
%

Given this assumption we then need to modify Lemma 4.5 and its proof as described in the Appendix. With this modification the convergence is proved exactly as in \cite{luan_oster} and results in the following Theorem \ref{thm:convergenceEPIRK} (see Theorem 4.1 in \cite{luan_oster}). 

\begin{thm}
\label{thm:convergenceEPIRK}
Let  the initial value problem \eqref{eqn:ivp} satisfy Assumptions 1 and 2.  Consider for its numerical solution an explicit exponential propagation iterative method of Runge-Kutta type (\ref{eqn:EPIRK}) that fullfills Assumption 3 and the order conditions of Table \ref{table:StiffOC} up to order $p$ for some $3\leq p \leq 5$. Then, under the stability assumption (4.16 \cite{luan_oster}), the method converges with order $p$.  In particular, the numerical solution satisfies the error bound 
\begin{equation}
\label{eqn:convergenceErrorBound}
\norm{u_n-u(t_n)}\leq C\sum_{\nu=0}^{n-1}h_\nu^{p+1}
\end{equation}
uniformly on $t_0\leq t_n \leq T$. The constant $C$ is independent of the chosen step size sequence.
\end{thm}
Using Theorem \ref{thm:convergenceEPIRK} we now construct specific stiffly accurate methods of order four and five.

\section{Construction of new schemes}  \label{sec:construction}
In this section we demonstrate how the flexibility of coefficients in the EPIRK framework can be used to construct efficient stiffly accurate schemes.  The improvement of the computational cost comes from constructing a method particularly optimized given an algorithm for evaluating the exponential matrix function-vector products.  Evaluation of exponential-like matrix functions and vector products $\psi(A)v$ constitutes the largest computational cost of an exponential integrator. The EPIRK methods have been originally introduced to minimize the number of such estimates required per time step as well as reduce the cost of each of these evaluations by carefully selecting the exponential functions in the products \cite{Tokman, Loffeld, LoffeldTokmanPar, RainwaterTokman}.  In \cite{tokmanadapt} the authors derived particularly efficient EPIRK methods that use the adaptive Krylov algorithm to approximate products $\psi(A)v$.  While there have been several techniques introduced to evaluate $\psi(A)v$, the adaptive Krylov algorithm remains one of the most general cost efficient way to estimate these terms if no a priori information is available about the spectrum of $A$.  Thus we adopt using this algorithm in constructing the new stiffly accurate EPIRK methods here.  

We begin with a description of our method of choice, the adaptive Krylov algorithm, and discuss the structural requirements application of this method imposes on a time integrator.  A systematic approach to solving the order conditions is then offered to derive appropriate three-stage methods of order four and five.  Specific EPIRK schemes are constructed following this technique.
	
	\subsection{Adaptive Krylov algorithm and its implementations}\label{subsec:AdaptKry}
	
			The computational cost of the standard Krylov-projection based algorithm to approximate terms of type $\psi_{ij}(g_{ij}h_nJ_n)v$ scales as $\mathcal{O}(m^2)$ where $m$ is the size of the Krylov basis required to achieve a prescribed accuracy.  The value $m$, in turn, depends on the spectrum of the matrix and it is expected that the computational cost of the Krylov projection will increase with the time step size $h_n$.  Thus, for a given error tolerance it might actually be more efficient to integrate with a smaller time step rather than encounter large Krylov bases.  An alternative and more efficient approach was proposed in \cite{sidje, niesenwright}. The adaptive Krylov algorithm seeks to evaluate linear combinations of type
								\begin{equation}
					\label{eqn:adaptKrylLinearComboPhi}
					\varphi_0(A)b_0+\varphi_1(A)b_1+\varphi_2(A)b_2+\cdots+\varphi_p(A)b_p
					\end{equation}
			where $A\in\R^{N\times N}$ and $b_{i}\in\R^N$ for $i=0,\cdots,p$.  The idea is to replace computing one large Krylov subspace of size $m$ with a finite number of smaller Krylov subspaces of sizes $m_1$, $m_2$, ..., $m_K$.  In \cite{niesenwright} it was observed that expressions like (\ref{eqn:adaptKrylLinearComboPhi}) can be computed in a way that replaces the evaluation of terms like $\varphi_p(A)b_P$ with something that requires fewer Krylov vectors, like $\varphi_p(\tau_k A)b_p$ with $0<\tau_k<1$.  For example, consider the following discretization of the interval $[0,1]$, $0=t_0<t_1<\dots<t_k<t_k+\tau_k<\dots<t_K=t_{end}=1$.  Then we would have to compute $K$ Krylov projections at computational cost proportional to $\mathcal{O}(m_1^2) + \mathcal{O}(m_2^2)+ ... + \mathcal{O}(m_K^2)$ which can be more efficient than computing $\varphi_p(A)b_p$ that has computational complexity of $\mathcal{O}(m^2)$.  Obviously if $K$ gets too large, the total cost of computing $K$ smaller Krylov subspaces may exceed the cost of computing only one large Krylov basis.  Therefore the efficiency of the algorithm is dependent on the choice of step sizes $\tau_k$.  As the optimal choice varies, an algorithm was developed in \cite{niesenwright} to choose these step sizes adaptively.  Further details and information can be found in \cite{niesenwright, sidje, tokmanadapt, LoffeldTokmanPar}.

\subsubsection*{Vertical exponential-Krylov methods.} \label{subsubsec:verticalKrylov}
In \cite{tokmanadapt}, specific EPIRK methods were designed to efficiently employ the adaptive Krylov algorithm.  By enforcing the requirement $\psi_{ij}(z)=\varphi_k(z)$ for fixed $j$ and $i=2,\dots, s-1$, we can construct an $s$-stage adaptive Krylov based EPIRK scheme that requires only $s$ Krylov projections per time step\cite{tokmanadapt,LoffeldTokmanPar}.  This requirement is necessary since if a linear combination (\ref{eqn:adaptKrylLinearComboPhi}) consists of a single term $\varphi_k(tJ)b_k$, it can be computed as $\varphi_k(tJ)b_k=u(t)/t^k$ for any real value $t$.  Hence we can compute $\varphi_k(g_{ij}hJ)b_k=u(g_{ij})/g_{ij}^k$ for fixed $j$ with just one Krylov evaluation (as long as $g_{ij}$ are included in the set of times $\set{t_k}_{k=0}^{K}$).  We refer to this implementation as ``vertical exponential-adaptive-Krylov'' or ``vertical exponential-Krylov'' due to computing the ``columns'' or the terms of each stage with shared vectors $b_k$ with one Krylov subspace. Note that the last term of the last stage is not restricted to a single $\varphi$-function but rather can be any linear combination of $\varphi$-functions since it does not share a vector with any terms of the previous stages.

The flexibility and choice of $g_{ij}$ coefficients in EPIRK methods allows for further improvement in overall computational cost.  In particular, by taking these coefficients to be smaller than 1 we would effectively reduce the size of the Krylov basis since $t_{end}<1$.  When implementing vertical adaptive Krylov we have for each fixed $j=1,\dots,s$, $t_{end}=\max_{i\in[2,j+1]}g_{ij}.$  Therefore an efficient vertical Krylov EPIRK scheme has $\max_i{g_{ij}}<1$ for each $j=2,\dots,s$. The classical order conditions offer enough freedom to choose small $g$-coefficients.  However, the same approach is difficult to use to construct the stiffly accurate methods since the stiff order conditions require that $\max_i{g_{ij}}=g_{(s+1)j}=1$ as shown in Lemma \ref{lemma:gCoeff_equal_1}. 
\begin{lem}\label{lemma:gCoeff_equal_1}
An $s$-stage EPIRK method satisfying the stiff order conditions  must have $g_{(s+1)k}=1$ for $k=2,\dots, s$.
\end{lem}
 \begin{proof}
 	It must be the case that there is at least one $b_j(Z)$ such that $b_j\nequiv 0$ for some $j=2,\dots,s$, otherwise the method can only be of order two.  We can then solve conditions (C1)-(C3) for the non-zero $b_j(Z)$ functions. The resulting $b_j(Z)$ are simply linear combinations of $\varphi_3(Z)$, $\varphi_4(Z)$ functions. By substituting these solution(s) for $b_j$'s into (\ref{eqn:a_coefficients}) we find that $$\psi_{(s+1)k}(g_{(s+1)k}Z)=A_3\varphi_3(Z)+A_4\varphi_4(Z) , \quad \textrm{ for some }A_i \in \R$$
 Since this must holds for all $Z\in \R^{n\times n}$, we can conclude that $g_{(s+1)k}=1$.
\end{proof} 
Even though the stiff order conditions are restrictive with respect to the $g$-coefficients in the last stage, we still have flexibility with the $g$'s in the internal stages and will use them to reduce the computational cost.  Our approach is to modify the implementation of the adaptive Krylov algorithm from computing ``vertically'' to computing ``horizontally'' or in a ``mixed" way as described below.

\subsubsection*{Horizontal exponential-adaptive-Krylov.}
The ``horizontal'' exponential-adaptive-Krylov, or exponential-Krylov, algorithm is intended to compute all terms in each stage with one Krylov evaluation.  In contrast to the vertical Krylov, here we compute along the ``rows'' (i.e. compute $\psi_{i_0j}(z)b_j$ for fixed $i_0$ and $j=1,\dots,i_0-1$).  Since each term will have a different vector $b_j$, the only way for a $s$-stage method to require only $s$ Krylov evaluations per time-step is to enforce the condition that any non-zero exponential terms in a given stage must share the same $g_{ij}$-value. As an example let us consider the following internal stage of a five-stage EPIRK method
\begin{equation}
\label{eqn:InterStageExample}
U_{n5}=u_n+\alpha_{51}\psi_{51}(g_{51}h_nJ_n)h_nf(u_n)+\alpha_{52}\psi_{52}(g_{52}h_nJ_n)h_n\Delta r(u_n)+0\cdot\psi_{53}(g_{53}h_nJ_n)\Delta^2 r(u_n)+\alpha_{54}\psi_{54}(g_{54}h_nJ_n)\Delta^3 r(u_n)
\end{equation}
where $\psi_{51}(z)=\varphi_1(z), \psi_{52}(z)=\varphi_2(z)+\varphi_{3}(z), $ and $\psi_{54}(z)=\varphi_3(z)$.   Using the recurrence relation (\ref{eqn:phiRecurrenceRelation}) we can express (\ref{eqn:InterStageExample}) as
\begin{equation}\label{eqn:InterStageExample_2}
\begin{array}{ll}
U_{n5}&=u_n + \alpha_{51}\!\left(\left(\varphi_3(g_{51}h_nJ_n)g_{51}h_nJ_n\!+\!\frac{1}{2}\right)g_{51}h_nJ_n+1\right)h_nf(u_n) \\
&\quad\quad + \alpha_{52} \left((\varphi_3(g_{52}h_nJ_n)g_{52}h_nJ_n+\frac{1}{2})+\varphi_3(g_{52}h_nJ_n)\right)h_n\Delta r(u_n)\\
&\quad\quad  + 0\cdot h_n\Delta^2r(u_n)+\alpha_{55}\varphi_3(g_{54}h_nJ_n)h_n\Delta^3 r(u_n)\\
&=u_n + b_0+\varphi_3(g_{51}h_nJ_n)b_1 + \varphi_3(g_{52}h_nJ_n)b_2+0 + \varphi_3(g_{54}h_nJ_n)b_4
\end{array}.
\end{equation}
Since the vectors $b_1,b_2,b_4$ are not the same we must have $g_{51}=g_{52}=g_{54}$ so (\ref{eqn:InterStageExample_2}) can be written in the form
\begin{equation}\label{eqn:InterStageExample_3}
U_{n5}=u_n+\varphi_3(g_{51}h_nJ_n)(b_1+b_2+b_4).
\end{equation}
Then the adaptive Krylov algorithm can be used to compute (\ref{eqn:InterStageExample_3}) with the possibility of taking $g_{51}<1$. Note that due to Lemma \ref{lemma:gCoeff_equal_1} all projections in the vertical method require adaptive Krylov algorithm to integrate over the interval $[0,1]$. Thus the savings associated with smaller $g$ coefficients which are equivalent to reducing the integration interval in adaptive Krylov algorithm to $[0,g]$ are not possible for the vertical methods.  Thus the horizontal version of the method with coefficients $g_{i1} < 1$ carries computational savings comparable to the vertical method. 

%

\subsubsection*{Mixed exponential-adaptive-Krylov.}
We can also use a combination of both vertical and horizontal Krylov adaptation to develop an EPIRK scheme.  The idea is to compute the last stage horizontally  and the internal stages vertically or with a combination of vertical and horizontal approaches depending on what yields the most optimized scheme.  This alleviates the drawbacks of strictly implementing vertical Krylov or horizontal Krylov for stiffly accurate EPIRK schemes and opens more possibilities for customization of methods as well as improving the efficiency of a particular scheme.  After solving the order conditions below, we will construct a specific mixed exponential-adaptive-Krylov, or exponential-Krylov, method that will serve as an illustrative example of this idea.

Note that any EPIRK method can be implemented in a vertical, horizontal or mixed way, however, the schemes can be constructed to be particularly optimized for a given implementation.  Thus, for example, we will call an EPIRK scheme optimized for a mixed implementation a mixed exponential-Krylov method but in the numerical tests section we test such method with either vertical, horizontal or mixed implementation and demonstrate that the integrator optimized for the mixed implementation and applied in this way is the most efficient.

 \subsection{Solving the order conditions} \label{subsec:solveOC}
Our primary objective is to construct new and more efficient stiffly accurate EPIRK schemes.  We focus on constructing three-stage methods. Below we solve the order conditions given in Table \ref{table:StiffOC} to obtain new general classes of three-stage fourth and fifth-order EPIRK methods.  For each of these classes, the remaining free-parameters are then chosen to obtain specific schemes each targeted for a specific version (vertical, horizontal or mixed) of the adaptive Krylov algorithm. 
We begin by considering a three-stage EPIRK method in EXPRB form (\ref{eqn:3stageEPIRKinEXPRBform})
\begin{equation}
\label{eqn:EPIRK_3stage2}
  \begin{aligned}
U_{n2}&= u_n+\alpha_{21}\psi_{21}(z)(g_{21}h_nJ_n)h_nf(u_n)\\
U_{n3}&= u_n+\alpha_{31}\psi_{31}(z)(g_{31}h_nJ_n)h_nf(u_n) + \underbrace{\alpha_{32}\psi_{32}(g_{32}h_nJ_n)}_{a_{32}(h_nJ_n)}h_nr(U_{n2})\\
u_{n+1} &= u_n + \varphi_1(z)(h_nJ_n)h_nf(u_n)+\underbrace{\left(\beta_2 \psi_{42}(g_{42}h_nJ_n)-2\beta_3\psi_{43}(g_{43}h_nJ)\right)}_{b_2(h_nJ_n)}h_nr(U_{n2})+\underbrace{\beta_3 \psi_{43}(h_ng_{43})}_{b_3(h_nJ_n)}h_n r(U_{n3})
  \end{aligned},
\end{equation} 
where $\psi_{i1}(z)=p_{i11}\varphi_1(z)+p_{i12}\varphi_2(z)+p_{i13}\varphi_3(z)$ for $i=2,3$. 

\subsubsection{Fourth-order methods}
 From the conditions given in Table \ref{table:StiffOC_EPIRK_EXPRB}, (C1) and (C2)  yield solutions for $b_2(z)$ and $b_3(z)$:
\begin{equation}\label{eqn:threeStage_bSols}
\begin{array}{l}
b_2(z)=\frac{2 g_{31} \P_{32}}{\alpha _{21}^2 \P_{21} \left(g_{31} \P_{21} \P_{32}-g_{21} \P_{22} \P_{31}\right)} \varphi _3(z)-\frac{3 \P_{31}}{\alpha _{21}^2 \P_{21} \left(g_{31} \P_{21} \P_{32}-g_{21} \P_{22} \P_{31}\right)} \varphi _4(z)\\
b_3(z)=\frac{2 g_{21} \P_{22}}{\alpha _{31}^2 \P_{31} \left(g_{21} \P_{22} \P_{31}-g_{31} \P_{21} \P_{32}\right)}\varphi _3(z) -\frac{3 \P_{21}}{\alpha _{31}^2 \P_{31} \left(g_{21} \P_{22} \P_{31}-g_{31} \P_{21} \P_{32}\right)}\varphi _4(z)
\end{array},
\end{equation}
with $\P_{ij}$'s defined by (\ref{eqn:Pcoeff}).  Upon substituting (\ref{eqn:threeStage_bSols}) into (C3) we obtain 
$$\frac{2  \left(\alpha _{31} g_{21} \P_{22} \P_{31}^2-\alpha _{21} g_{31} \P_{21}^2 \P_{32}\right)\varphi _3(z)+\left(3 \P_{21} \left(2 g_{31} \P_{32}-\alpha _{31} \P_{31}^2+\alpha _{21} \P_{21} \P_{31}\right)-6 g_{21} \P_{22} \P_{31}\right)\varphi _4(z) }{g_{21} \P_{22} \P_{31}-g_{31} \P_{21} \P_{32}}=0. $$This condition is satisfied by enforcing the coefficients of $\varphi_3$ and $\varphi_4$ to be zero.  The resulting equations can then be used to solve for 
\begin{equation}
\label{eqn:threeStage_alphaSols}
\alpha_{21}=\frac{2 g_{21} \P_{22}}{\P_{21}^2}\qquad \textrm{ and } \qquad \alpha_{31}=\frac{2 g_{31} \P_{32}}{\P_{31}^2}.
\end{equation}
With (\ref{eqn:threeStage_bSols}) and (\ref{eqn:threeStage_alphaSols}) satisfied, methods of the form (\ref{eqn:EPIRK_3stage2}) define a new class of stiffly accurate fourth-order three-stage methods. The flexibility with the remaining parameters makes these methods appealing.  For example, the choice $a_{32}(Z)\equiv 0$ and $\psi_{21}(Z)=\psi_{31}(Z)=\varphi_1(Z)$ yields the structure necessary in order to construct a vertical, horizontal, or a mixed exponential-Krylov method.  These methods will require three Krylov projections per time-step when implementing the vertical or horizontal whereas it is possible to construct a mixed exponential-Krylov scheme that only requires two projections.  The removal of a whole projection each time-step can significantly reduce the overall cost.  Another computation saving feature of the fourth-order schemes is the ability to choose both $g_{21}$ and $g_{31}$.  The choice $g_{21}=\frac{1}{2}$ and $g_{31}=\frac{2}{3}$ leads to the following three-stage fourth-order method \textit{EPIRK4s3A}:
\begin{equation}
\label{eqn:epirk4s3}
  \begin{aligned}
U_{n2}&= u_n+\frac{1}{2}\varphi_1(\frac{1}{2}h_nJ_n)h_nf(u_n)\\
U_{n3}&= u_n+\frac{2}{3}\varphi_1(\frac{2}{3}h_nJ_n)h_nf(u_n) \\
u_{n+1} &= u_n + \varphi_1(h_nJ_n)h_nf(u_n)+\left(32\varphi_3(h_nJ_n)-144\varphi_4(h_nJ_n)\right)h_nr(U_{n2})+\left(-\frac{27}{2}\varphi_3(h_nJ_n)+81\varphi_4(h_nJ_n)\right)h_n r(U_{n3})
  \end{aligned}.
\end{equation} 
Another fourth-order method can be obtained similarly by taking $\psi_{21}(Z)=\psi_{31}(Z)=\varphi_2(Z)$. Different $g$-coefficients were specified but were chosen so that they are comparable to those above to obtain the  \textit{EPIRK4s3B} method
\begin{equation}
\label{eqn:epirk4s3b}
  \begin{aligned}
U_{n2}&= u_n+\frac{2}{3}\varphi_2(\frac{1}{2}h_nJ_n)h_nf(u_n)\\
U_{n3}&= u_n+\varphi_2(\frac{3}{4}h_nJ_n)h_nf(u_n) \\
u_{n+1} &= u_n + \varphi_1(h_nJ_n)h_nf(u_n)+\left(54\varphi_3(h_nJ_n)-324\varphi_4(h_nJ_n)\right)h_nr(U_{n2})+\left(-16\varphi_3(h_nJ_n)+144\varphi_4(h_nJ_n)\right)h_n r(U_{n3}).
  \end{aligned}
\end{equation} 
Note that method {\it EPIRK4s3B} lies outside of the set of exponential Rosebrock methods because it is using $\varphi_2(z)$ function in the internal stages. In the numerical experiments section we will show the performance of {\it EPIRK4s3B} is very similar with {\it EPIRK4s3A}.  This illustrates that the EPIRK form allows for more flexibility in constructing the methods.

In Section \ref{sec:numExper} it will be shown that (\ref{eqn:epirk4s3}) performs particularly well when implemented in a horizontal or a mixed exponential-Krylov way.  The flexibility in constructing this fourth order method allows building a scheme that can even be computationally favorable compared to three-stage fifth order methods as shown below.

\subsubsection{Fifth-order methods}
Our construction of methods of order five are built upon the solutions obtained above which satisfy conditions (C1)-(C3).  A fifth-order method must additionally satisfy (C4)-(C8); upon inserting (\ref{eqn:threeStage_bSols}) it is found that there is no solution which is able to satisfy these conditions for all $Z\in \R^{n\times n}$.  However, as noted in \cite{luan_oster} for EXPRB methods, with additional regularity assumptions the convergence results hold under weaker assumptions on the coefficients of the method.  Similar results can be obtained for the general EPIRK form by considering a simplified set of conditions (C4*)-(C8*) given in Table \ref{table:simplified_StiffOC_EPIRK_EXPRB} that replace conditions (C4)-(C8). 
 The regularity assumption on operators in \eqref{eqn:jacobian} needed to prove convergence in this case is the same as for EXPRB methods but stated in a less compact form as follows. \\
\begin{assumption2}{4}
\label{assump:AdditionalRegularity}
The operator $L$ and the nonlinearity $N$ from \eqref{eqn:linearizedForm} are such that 
\begin{equation}\begin{array}{c}
L(\frac{\partial^2 \nt}{\partial y^2}(\ut)(\ut', \ut^{(3)})),\quad  L(\frac{\partial^2 \nt}{\partial y^2}(\ut)(\ut'', \ut'')),\quad L(\frac{\partial^3 \nt}{\partial u^3}(\ut)(\ut ',\ut ',\ut '')),\quad
 L \frac{\partial^4 \nt}{\partial u^4}(\ut)(\ut',\ut',\ut',\ut')\\
   L(\frac{\partial^2 \nt}{\partial u^2}(\ut)\left( \ut', \Psi_i(h \Jt)\frac{\partial^2 \nt}{\partial u^2}(\ut)(\ut', \ut')\right))
\end{array}
\end{equation}
are uniformly bounded on $\X$ for all $2\leq i\leq s$.  
\end{assumption2}
\noindent Given Assumption 4, we can now prove convergence of methods that satisfy conditions (C4*)-(C8*), if the stability requirement 4.16 in \cite{luan_oster} holds. The major aspects of the proof are exactly the same as in \cite{luan_oster} but several modifications are needed as we describe below. 
\begin{thm}[Theorem 4.2 \cite{luan_oster}]\label{thm:SimplifiedCondConvergentProof}
Let the initial value problem (\ref{eqn:ivp}) satisfy Assumptions 1,2, and 4.  Consider for its numerical solution an EPIRK method (\ref{eqn:EPIRK}) which satisfies Assumption 3, the order conditions (C1)-(C3) of Table \ref{table:StiffOC} and (C4*)-(C8*) of Table \ref{table:simplified_StiffOC_EPIRK_EXPRB}.  Then, under the stability assumption (4.16 \cite{luan_oster}), the method is convergent of order 5.  In particular, the numerical solution $u_n$ satisfies the error bound 
\begin{equation}
\label{eqn:fifthOrderErrorBound}
\norm{u_n-u(t_n)}\leq C\sum_{\nu =0}^{n-1} h_\nu^6
\end{equation}
uniformly on $t_0\leq t_n \leq T$. The constant $C$ is independent of the chosen step size sequence.
\end{thm}
\begin{proof}
   We begin by defining terms of $\O(h_n^5)$ in (\ref{eqn:localError}) by \begin{equation}\label{eqn:SimpCondTHMxiOperators}
\begin{aligned}
&\xi_{5,1}(z)=\sum_{i=2}^s b_{i}(z)g_{i1}^2\alpha_{i1}^2\P_{i1}\P_{i3}-4\varphi_5(z)\\
&\xi_{5,2}(z)= \sum_{i=2}^s b_{i}(z)\frac{1}{2!}\alpha_{i1}^2g_{i1}^2\P_{i2}^2-3\varphi_5(z)\\
&\xi_{5,3}(z)=\sum_{i=2}^s b_{i}(z)\frac{1}{2!}g_{i1}\alpha_{i1}^3\P_{i1}^2\P_{i2}-6\varphi_5(z)\\
&\xi_{5,4}(z)=\sum_{i=2}^s b_i(z)\frac{1}{4}\alpha_{i1}^4\P_{i1}^4-\varphi_{5}(z),
\end{aligned}
\end{equation}
each of which can be written in the following form
$$\xi_{5,j}(z)=\sum_{i=2}^s b_i(z)C_{1,j}-C_{2,j}\varphi_5(z)  $$
where $C_{1,j},C_{2,j}\in \R$ and $j=1,\dots,4$.  Using the recurrence relation (\ref{eqn:phiRecurrenceRelation}) we have for each $j=1,\dots,5$
\begin{eqnarray}
\xi_{5,j}(z)-\xi_{5,j}(0)&=&\sum_{i=2}^s \left(b_{i}(z)-b_{i}(0)\right)C_{1,j}+\left(\varphi_5(z)-\varphi_5(0)\right)C_{2,j}\nonumber\\
&=& \sum_{i=2}^s z \widehat{b}_{i}(z)C_{1,j}+z\varphi_6(z)C_{2,j}\label{eqn:SimpCondTHMrecurrence}\\
&= & z\left(\sum_{i=2}^s \widehat{b}_i(z)C_{1,j}+\varphi_6(z)C_{2,j}\right)\nonumber\\
&= & z \widehat{\xi}_{5,j}(z)\label{eqn:SimpCondTHMBoundedOperator}
\end{eqnarray}
where $\widehat{\xi}_{5,j}$ is a bounded operator. Since the simplified conditions (C4*)-(C8*) are satisfied, we have  $\xi_{5,j}(h_n\Jt)=0+h_n\Jt\wh{\xi}_{5,j}(h_n\Jt)$.  Substituting this back into (\ref{eqn:localError}) and using Assumption 4 yields each term to be of order $\O(h_n^6)$. Therefore our error $\tilde{e}_{n+1}=\O(h_n^6)$.
\end{proof}
\begin{table}
\begin{center}
\begin{tabular}{lll}
\hline 
Ref. label &\hspace{10pt} Order Condition & Order \\ 

\hline 
(C4*) & \hspace{10pt}$\sum_{i=2}^{s} b_i(0) g_{i1}^2 \alpha_{i1}^2\P_{i1}\P_{i3}=4\varphi_5(0)$ & 5\\ 

(C5*) &\hspace{10pt}$ \sum_{i=2}^{s} b_i(0)\alpha_{i1}^2 g_{i1}^2\P_{i2}^2 =3!\varphi_5(0)$ & 5\\ 

(C6*) & \hspace{10pt}$\sum_{i=2}^{s} b_i(0) g_{i1}\alpha_{i1}^3\P_{i1}^2\P_{i2} =12\varphi_5(0)$ & 5 \\ 

(C7*) & \hspace{10pt}$\sum_{i=2}^{s} b_i(0)\alpha_{i1}^4\P_{i1}^4=4!\varphi_5(0)$ & 5 \\ 

(C8*)& $\hspace{10pt} \sum_{i=2}^{s} b_i(0)\alpha_{i1}^2\P_{i1}K\Psi_i(Z)=0$ & 5 \\ 
\hline 
\end{tabular}  
\end{center}
\caption{Simplified stiff order conditions.  Note that $Z$ and $K$ are arbitrary square matrices, $\Psi_i$ is given by (\ref{eqn:BigPsi}),and $\psi_{3,i}(Z)=\sum_{j=2}^{i-1}a_{ij}(Z)\frac{c_j^2}{2}-c_{i}^3\varphi_3(c_i Z)$}
\label{table:simplified_StiffOC_EPIRK_EXPRB}
\end{table}

As a result of Theorem \ref{thm:SimplifiedCondConvergentProof}, a stiffly accurate three-stage fifth-order EPIRK scheme can be constructed. Let us consider the solutions obtained for the fourth-order schemes and (C4*)-(C8*). We first note that the simplified conditions (C6*) and (C7*) become equivalent and are used to solve for $g_{21}$:
\begin{equation}
\label{eqn:threeStageG21sol} 
g_{21}=\frac{3 \P_{21} \left(2 \P_{31}-5 g_{31} \P_{32}\right)}{5 \P_{22} \left(3 \P_{31}-8 g_{31} \P_{32}\right)}.
\end{equation}  
Substituting (\ref{eqn:threeStageG21sol}) into simplified conditions (C4*) \& (C5*) and using Mathematica software we found that there is only one solution that yields desirable coefficients (i.e. $g_{ij}\in(0,1]$ and $\P_{ij}\in \R$) and is able to satisfy the remaining conditions
\begin{eqnarray}
\; \P_{21}=2 \P_{22},\;
 \P_{32}= -\frac{\P_{31}}{2},
 \;\P_{33}=\frac{\P_{31} \left(80 g_{31}^2 \P_{22}-225 g_{31}^2 \P_{23}+120 g_{31} \P_{22}-360 g_{31} \P_{23}+48 \P_{22}-144 \P_{23}\right)}{30 g_{31}^2 \P_{22}} \label{eqn:threeStageLargePsol1}
\end{eqnarray}
Expressions for $p_{ijk}$ coefficients are then found by replacing $\P_{ij}$'s in (\ref{eqn:threeStageLargePsol1}) with (\ref{eqn:Pcoeff}), resulting in two possibilities
\begin{eqnarray}
&p_{211}=-\frac{p_{212} \left(225 g_{31}^2 p_{311}+40 g_{31}^2 p_{312}-360 g_{31} p_{311}-60 g_{31} p_{312}+144 p_{311}+24 p_{312}\right)}{15 g_{31}^2 p_{312}},p_{213}= -2 p_{212},p_{313}=2 p_{312} \text{ or }\label{eqn:smallPexpression1}\\
&p_{212}= 0,p_{213}=0,p_{312}= 0,p_{313}= 0\label{eqn:smallPexpression2}.
\end{eqnarray}
Expressions in (\ref{eqn:smallPexpression1}) give rise to the use of general $\psi_{i1}$ functions and the ability to construct horizontal exponential-Krylov methods.  The coefficients given by (\ref{eqn:smallPexpression2}) simplify (\ref{eqn:EPIRK_3stage2}) by setting $\psi_{i1}(z)=\varphi_1(z)$.  This simplification condition provides these methods with the structure necessary for construction of a mixed exponential-Krylov scheme. For this reason we consider each case separately and construct methods specifically designed for the two different approaches - horizontal and mixed. 

Beginning with (\ref{eqn:smallPexpression1}), the remaining condition (C8*) is
\begin{equation}
\label{eqn:threestageC8weak}
b_2(0)\alpha_{21}^2\P_{21}\cdot \Psi_2(Z)+b_3(0)\alpha_{31}^2\P_{31}\cdot \Psi_{3}(Z)=0,
\end{equation}
which is satisfied by
\begin{equation}
a_{32}=\frac{2 g_{21} g_{31} p_{211} P_{32} \left(8 g_{31} P_{32}-3 P_{31}\right) \phi _{3,g_{21}}}{\alpha _{21} P_{22} P_{31}^2 \left(3 P_{21}-8 g_{21} P_{22}\right)}+\frac{2 g_{31} \phi _{3,g_{31}} \left(3 \alpha _{31} g_{31} p_{311} P_{21} P_{22} P_{31}^2-8 \alpha _{31} g_{21} g_{31} p_{311} P_{22}^2 P_{31}^2\right)}{\alpha _{21}^2 P_{21}^2 P_{22} P_{31}^2 \left(3 P_{21}-8 g_{21} P_{22}\right)}
\label{eqn:a32generalSolThreestage}
\end{equation}
\begin{table}
\begin{center}
\begin{tabular}{lll}
\hline 
Ref. label &\hspace{10pt} Order Condition & Order \\ 

\hline 
C1 &\hspace{10pt} $\sum_{i=2}^{s} b_i(Z)\alpha_{i1}^2\P_{i1}^2 =2\varphi_3(Z)$& 3 \\ 
\hline 
C2 &\hspace{10pt} $\sum_{i=2}^{s} b_i(Z)\alpha_{i1}^2g_{i1}\P_{i1}\P_{i2}=3\varphi_4(Z)$ & 4 \\ 

C3 &\hspace{10pt} $\sum_{i=2}^{s} b_i(Z)\alpha_{i1}^2\mathbb{P}_{i1}^3=3!\varphi_4(Z)$& 4 \\ 
\hline 
C4 &\hspace{10pt} $\sum_{i=2}^{s} b_i(Z) g_{i1}^2 \alpha_{i1}^2\P_{i1}\P_{i3}=4\varphi_5(Z)$ & 5\\ 

C5 & \hspace{10pt}$ \sum_{i=2}^{s} b_i(Z)\alpha_{i1}^2 g_{i1}^2\P_{i2}^2 =3!\varphi_5(Z)$ & 5\\ 

C6 &\hspace{10pt} $\sum_{i=2}^{s} b_i(Z) g_{i1}\alpha_{i1}^3\P_{i1}^2\P_{i2} =12\varphi_5(Z)$ & 5 \\ 

C7 &\hspace{10pt} $\sum_{i=2}^{s} b_i(Z)\alpha_{i1}^4\P_{i1}^4=4!\varphi_5(Z)$ & 5 \\ 
\hline 
C8& \hspace{10pt}$ \sum_{i=2}^{s} b_i(Z)\alpha_{i1}^2\P_{i1}K\Psi_i(Z)=0$ & 5 \\ 
\hline 
\end{tabular}  
\begin{tabular}{lll}
\hline 
Ref. label  & \hspace{10pt}Order Condition & Order \\ 

\hline 
C1$^\prime$ & $\sum_{i=2}^{s} b_i(Z)c_{i}^2 =2!\varphi_3(Z)$& 3 \\ 
\hline 
C2$^\prime$  & $\sum_{i=2}^{s}b_i(Z)c_{i}^3=3!\varphi_4(Z)$ & 4 \\ 

 &  & \\
\hline 
C3$^\prime$ & $\sum_{i=2}^{s} b_i(Z)c_{i}^4 =4!\varphi_5(Z)$ & 5\\ 
& & \\
& & \\ 
& & \\ \hline
C4$^\prime$& $\sum_{i=2}^{s} c_{i}b_{i}(Z)K\psi_{3,i}(Z) =0$ & 5 \\ 
\hline 
\end{tabular} 
\end{center}
\caption{Stiff order conditions for EPIRK methods and EXPRB.  Note that $Z$ and $K$ are arbitrary square matrices, $\Psi_i$ is given by (\ref{eqn:BigPsi}),and $\psi_{3i}(Z)=\sum_{j=2}^{i-1}a_{ij}(Z)\frac{c_j^2}{2}-c_{i}^3\varphi_3(c_i Z)$}
\label{table:StiffOC_EPIRK_EXPRB}
\end{table}	

\subsubsection*{Horizontal exponential-Krylov methods.}
After substituting (\ref{eqn:threeStageLargePsol1}) and (\ref{eqn:smallPexpression1}) into (\ref{eqn:a32generalSolThreestage}), we have 
\begin{eqnarray}
 a_{32}(Z)&=\frac{100 \left(3-4 g_{31}\right){}^2 g_{31}^3 p_{311} }{3 \left(4-5 g_{31}\right){}^2 \left(6 p_{311}+p_{312}\right)}\varphi _{3}(g_{31}Z)+\frac{20 \left(3-4 g_{31}\right){}^2 g_{31} \left(5 g_{31}^2 \left(45 p_{311}+8 p_{312}\right)-60 g_{31} \left(6 p_{311}+p_{312}\right)+24 \left(6 p_{311}+p_{312}\right)\right) }{9 \left(4-5 g_{31}\right){}^2 \left(6 p_{311}+p_{312}\right)} \varphi_{3}(g_{21}Z).\label{eqn:threeStage_a32}
\end{eqnarray}
Since $a_{32}(Z)$ is a function of both $g_{31}$ and $g_{21}$, a horizontal exponential-Krylov method is not possible unless one of the terms vanish (since $g_{21}\neq g_{31}$). We thus introduce a new condition by setting one of these terms to zero.  The first term in the second internal stage already uses $g_{31}$ in the evaluation of $\varphi_1(Z)$ and therefore we seek removing the term associated with $\varphi_3(g_{21}Z)$ in (\ref{eqn:threeStage_a32}).   With the use of Mathematica, only one solution was found which did not violate any of the conditions or specifications,
$$p_{312}=-\frac{9 \left(5 g_{31}-4\right){}^2 p_{311}}{4 \left(10 g_{31}^2-15 g_{31}+6\right)}.$$
With all conditions satisfied, the remaining parameters $g_{31},p_{212},$ and $p_{311}$ can be chosen freely.  In order to optimize our $g$-coefficients, we use (\ref{eqn:threeStageG21sol}) to help choose $g_{31}$.  The relationship (\ref{eqn:threeStageG21sol}) can be reduced under the conditions for a fifth-order method to $$ g_{21}=\frac{3 \left(5 g_{31}-4\right)}{5 \left(4 g_{31}-3\right)}.$$
From the plot of this relationship in Figure \ref{fig:g21g31} it can be seen that minimizing either $g_{21}$ or $g_{31}$ results in the other coefficient approaching eight tenths.  Furthermore, Figure \ref{fig:g21g31} shows that minimizing one of these coefficients is the best choice computationally.
 \begin{figure}[hbtp]
 \begin{center}
     \includegraphics[scale=0.4]{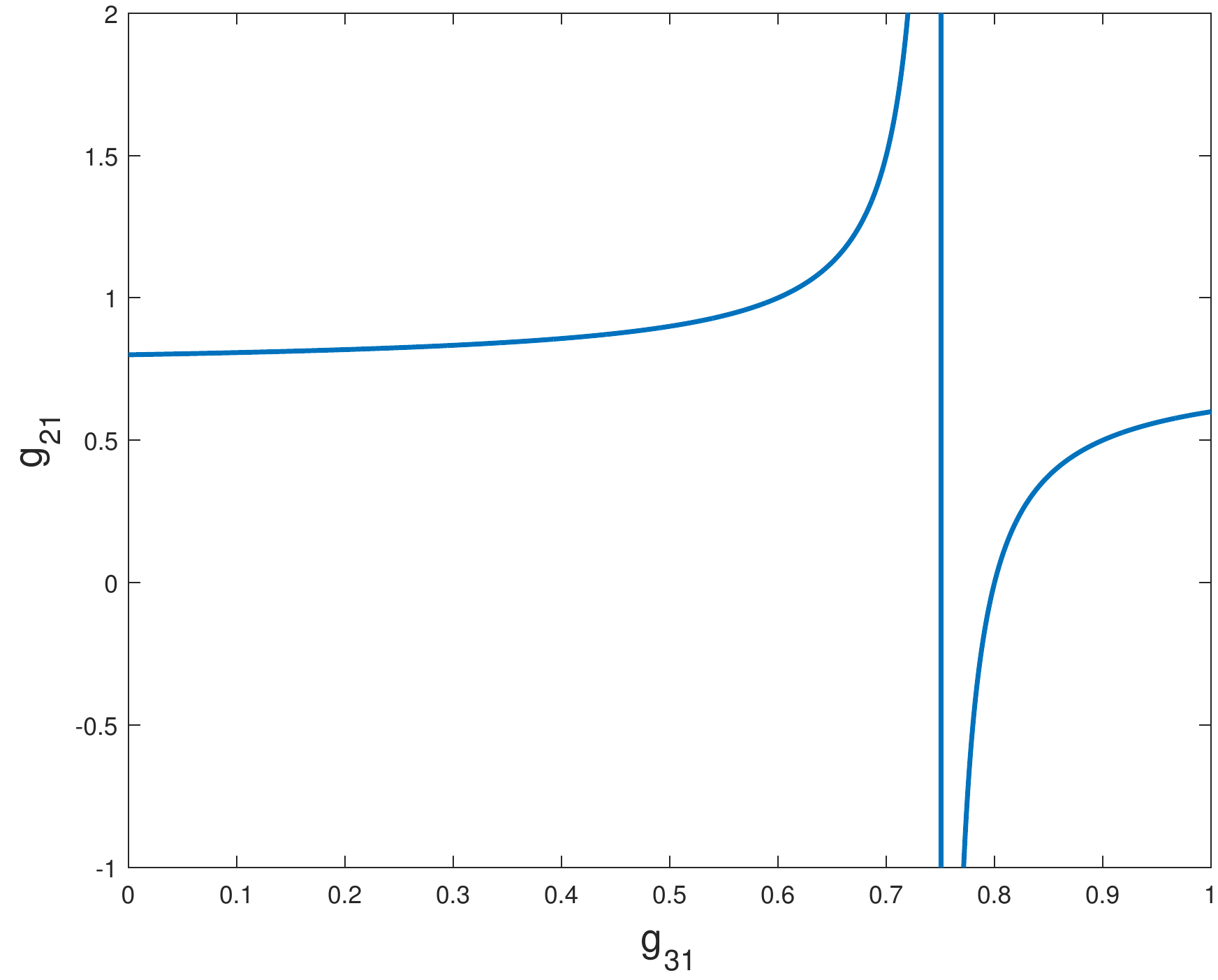}
     \end{center}
    \caption{$g_{21}$ and $g_{31}$}
    \label{fig:g21g31}
    \end{figure}
Thus let us specify $g_{31}=4/9$, $p_{212}=1$ and $p_{311}=1$ to obtain our first horizontal exponential-Krylov friendly method \textit{EPIRK5s3}-Horz:
\begin{equation}
\label{eqn:epirk5s3}
  \begin{aligned}
&U_{n2}= u_n+\frac{288}{55}\left(\varphi_2(\frac{48}{55}h_nJ_n)-2\varphi_3(\frac{48}{55}h_nJ_n)\right)h_nf(u_n)\\
&U_{n3}= u_n+\frac{212}{45}\left(\varphi_1(\frac{4}{9}h_nJ_n)-\frac{288}{53}\varphi_2(\frac{4}{9}h_nJ_n)+\frac{576}{53}\varphi_3(\frac{4}{9}h_nJ_n)\right)h_nf(u_n)+\frac{32065}{13122}\varphi_3(\frac{4}{9}h_nJ_n)h_nr(U_{n2}) \\
&u_{n+1} = u_n + \varphi_1(h_nJ_n)h_nf(u_n)+\left(-\frac{166375}{61056}\varphi_3(h_nJ_n)+\frac{499125}{27136}\varphi_4(h_nJ_n)\right)h_nr(U_{n2})\\
&\qquad +\left(\frac{2187}{106}\varphi_3(h_nJ_n)-\frac{2187}{106}\varphi_4(h_nJ_n)\right)h_n r(U_{n3})
  \end{aligned},
\end{equation}

\subsubsection*{Mixed exponential-Krylov methods.}

For a mixed exponential-Krylov method, we propose using vertical exponential-Krylov approach to compute the internal stages and horizontal exponential-Krylov method for the last stage (see Table \ref{table:EXPRB53s3}). For a three-stage method, this type of a mixed exponential-Krylov method requires $\Psi_{21}(z)=\Psi_{31}(z)=\varphi_k(z)$  for some fixed $k\in \N$.  The simplification from coefficients (\ref{eqn:smallPexpression2}) satisfies this requirement with $k=1$.  Therefore we obtain a three-stage fifth-order mixed exponential-Krylov method by additionally satisfying
$$a_{32}(Z)=\frac{2 \alpha _{31} g_{31}^2 p_{311} \varphi _{3}(g_{31}Z)}{\alpha _{21}^2 p_{211}^2}+\frac{2 \left(20 g_{31}^2-31 g_{31}+12\right) g_{31} \varphi _{3}(g_{21}Z)}{\alpha _{21} p_{211}} .$$
A further result from this simplification is that a stiffly accurate fifth-order three-stage EPIRK method is also a fifth-order three-stage EXPRB method.  Thus any three-stage EXPRB scheme is of a mixed exponential-Krylov type.  As an example and  for our numerical experiments we will consider a fifth-order three-stage method from \cite{luan_oster},  \textit{EXPRB53s3} (Table \ref{table:EXPRB53s3}).
  
\begin{table}[h]
\centering
\begin{tabular}{llll|llll}
\cline{5-5}
$U_{n2}$        & = & $u_n$ & + & \multicolumn{1}{l|}{$\frac{1}{2} \varphi_1(\frac{1}{2}h_nJ_n)h_nf(u_n)$ } &                        &                         &                        \\
       & &  &  & \multicolumn{1}{l|}{ } &                        &                         &                         \\ \cline{7-7}
$U_{n3}$        & = & $u_n$ & + & \multicolumn{1}{l|}{$\frac{9}{10}\varphi_1(\frac{9}{10} h_nJ_n)h_nf(u_n) $} & \multicolumn{1}{l|}{+} & \multicolumn{1}{l|}{$\left(\frac{27}{25}\varphi_3(\frac{1}{2}h_nJ_n)+\frac{729}{125}\varphi_3(\frac{9}{10}h_nJ_n)\right)h_nr(U_{n2})$} &              \\ \cline{5-5} \cline{7-7}          \\ \cline{5-8} 
$u_{n+1}$ & = & $u_n$ & + & \multicolumn{4}{l|}{$\varphi_1(h_nJ_n)h_nf(u_n) + \left(18\varphi_3(h_nJ_n)-60\varphi_4(h_nJ_n)\right)h_nr(U_{n2})+\left(-\frac{250}{81}\varphi_3(h_nJ_n)+\frac{500}{27}\varphi_4(h_nJ_n)\right) h_nr(U_{n3})$}    \\\cline{5-8}
\end{tabular}
\caption{{\it EXPRB53s3}  and grouping of terms for mixed exponential-adaptive-Krylov}
\label{table:EXPRB53s3}
\end{table}

 \subsection{Variable time-stepping}\label{subsec:variabletimestep}

Variable time-stepping has been used with both the vertical implementation of EPIRK and EXPRB methods in \cite{Tokman2006} and \cite{erow4} respectively.  For both of these classes of methods the approach to implementing an efficient variable step-size mechanism was to embed a lower-order error estimator into a higher-order method in such a way that both rely on the same internal stages and do not require additional Krylov projections per time step.  While this is possible for vertical Krylov implementation, the horizontal and mixed implementations are limited by computing the final stage horizontally where one Krylov projection is used for its approximation and accounts for the specific coefficients of that stage. Thus, the implementation of variable time-stepping in this manner will require an extra Krylov projection each time-step in order to calculate the error estimator.  To further reduce computational cost of the horizontal and mixed implementations with variable time-stepping the adaptive Krylov algorithm has to be modified. This is the goal of our current research but in this paper we restrict our attention to the existing adaptive Krylov algorithm as in \cite{niesenwright}.

Since the horizontal and vertical implementations require the same number of projections per time-step, the cost of an additional projection can offset any computational gains from optimized $g$-coefficients in the horizontal implementation.  However, the mixed implementation of methods like (\ref{eqn:epirk4s3}) requires fewer projections each time-step than the vertical implementation of a method with the same number of stages.  Therefore the extra Krylov evaluation for the mixed implementation would still be competitive with the vertical implementation due to now having the same number of projections each time-step.  As an example we can embed the following third-order method into {\it EPIRK4s3}
\begin{equation}
\label{eqn:embedded3rdOrderMethod1}
\widehat{u}_{n+1}=u_n + \varphi_1(h_nJ_n)h_nf(u_n)+8\varphi_3(h_nJ_n)h_nr(U_{n2})
\end{equation}
and use it as our error estimator for both vertical and mixed implementations. 

To efficiently approximate (\ref{eqn:embedded3rdOrderMethod1}), the vertical method needs to use the same Krylov bases that are computed each time-step for (\ref{eqn:epirk4s3}).  In \cite{TokmanTranquilli} methods were restricted to using single $\varphi$-functions for terms who shared the same vector.  By modifying the implementation we can account for multiple $\varphi$-functions and approximate the terms  $$\left(32\varphi_3(h_nJ_n)-144\varphi_4(h_nJ_n)\right)h_nr(U_{n2})\quad \textrm{ and } \quad 8\varphi_3(h_nJ_n)h_nr(U_{n2})$$
in \eqref{eqn:epirk4s3} and \eqref{eqn:embedded3rdOrderMethod1} using the same Krylov basis.  After computation of the Krylov basis for $\varphi_4(h_nJ_n)h_nr(U_{n2})$, an approximation of $\varphi_3(h_nJ_n)h_nr(U_{n2})$ can then be obtained by using the recurrence relation $\varphi_{k}(z)=z\varphi_{k+1}(z)+1/k!$.

\section{ Numerical Experiments}\label{sec:numExper}
The numerical experiments presented below are designed to address several objectives. First, we want to demonstrate the performance of the new stiffly accurate EPIRK schemes. Second, we will confirm our claim that implementing the horizontal and/or mixed exponential-adaptive-Krylov algorithm for stiffly accurate methods can offer significant computational savings. Third, we examine the accuracy of the integrators on problems which do not satisfy Assumptions 1 and 2. Finally, we conclude the section with tests that illustrate the performance of the variable time stepping version of the methods. 
		
The integrators employed in our numerical experiments are: {\it EPIRK4s3A}  (\ref{eqn:epirk4s3}), {\it EPIRK4s3B} (\ref{eqn:epirk4s3b}), {\it EPIRK5s3} (\ref{eqn:epirk5s3}), {\it EXPRB53s3} (Table \ref{table:EXPRB53s3}) and one classically (non-stiff) derived method {\it EPIRK5P1} from \cite{TokmanTranquilli}.  Each of the stiffly accurate methods will be implemented in its vertical, horizontal and mixed exponential-Krylov versions as follows:  
\begin{itemize}
	\item {\it EPIRK4s3A}  - vertical, horizontal, and mixed; these three implementations of the same fourth-order method demonstrate the advantages of mixed or horizontal forms;
	\item {\it EPIRK4s3B} - mixed; this is an EPIRK method that cannot be written in the exponential Rosenbrock form;
	\item {\it EPIRK5s3} - horizontal;  this fifth-order method has been derived specifically to take advantage of the horizontal form;
	\item {\it EXPRB5s3} - vertical and mixed; this fifth-order method has been designed to take advantage of a mixed form.
\end{itemize} 
The classically derived {\it EPIRK5P1} method has been included to illustrate the difference in performance compared to the stiffly accurate and particular implementation adapted schemes. 

We begin by describing the test problems  that we will be using and verifying the theoretically predicted order of all the integrators used in our experiments.  The simulations and results are then detailed.

\subsection{Test problems}\label{subsec:testproblems}
Our numerical experiments are conducted on a select subset of the test problems used in \cite{Loffeld}.  For each of these test problems, numerical comparisons between previously derived exponential schemes not included here can be found in \cite{Tokman,Loffeld,Hoch,luan_oster}.  In all of the problems presented below the $\nabla^2$ term is discretized using the standard second order finite differences. \\

            \noindent \emph{Allen-Cahn 2D.} Two-dimensional Allen-Cahn equation \cite{bates}:
            \begin{eqnarray*}
                u_t &=& \alpha\nabla^2u + u - u^3, \mbox{ }x,y \in [-1,1], t \in [0,1.0]
            \end{eqnarray*}
            with $\alpha = 0.1$, using no-flow boundary conditions and initial conditions given by $u = 0.1 + 0.1\cos(2\pi x)\cos(2\pi y)$. \\

            \noindent \emph{ADR 2D.} Two-dimensional advection-diffusion-reaction equation \cite{caliariostermann}:
            \begin{eqnarray*}
                u_t &=& \epsilon(u_{xx} + u_{yy}) - \alpha(u_x + u_y) + \gamma u(u - \tfrac{1}{2})(1 - u), \mbox{ }x,y \in [0,1], t \in [0,0.1],
            \end{eqnarray*}
            where $\epsilon = 1/100$, $\alpha = -10$, and $\gamma = 100$.  Homogeneous Neumann boundary conditions were used and the initial conditions were given by $u = 256(xy(1-x)(1-y))^2 + 0.3$. \\
            
            \noindent \emph{Brusselator 2D.} Two-dimensional Brusselator \cite{hairer2, brusselator}
        \begin{eqnarray*}
            u_t &=& 1 + u^2v - 4u + \alpha\nabla^2u, \mbox{ }x,y \in [0,1] \\
            v_t &=& 3u - u^2v + \alpha\nabla^2v \\
            \alpha &=& 0.02
        \end{eqnarray*}
        with homogeneous Neumann boundary conditions, $t\in [0,1]$, and initial values
        \begin{eqnarray*}
            u &=& 2+0.25y  \\
            v &=&  1+0.8x
        \end{eqnarray*}            
            
%

       \noindent \emph{Gray-Scott 2D.} Two-dimensional Gray-Scott \cite{grayscott} with periodic boundary conditions:
        \begin{eqnarray*}
            u_t &=& d_u\nabla^2u - uv^2 + a(1 - u), \mbox{ }x,y \in [0,1] \\
            v_t &=& d_v\nabla^2v + uv^2 - (a + b)v, 
        \end{eqnarray*}
       and $        d_u = 0.2,   d_v= 0.1,a= 0.04,  b = 0.06$. Initial conditions given by
        \begin{eqnarray*}
            u &=& 1 - e^{-150 \lbrack (x - \frac{1}{2})^2 + (y - \frac{1}{2})^2 \rbrack}, \\
            v &=& e^{-150 \lbrack (x - \frac{1}{2})^2 + 2(y - \frac{1}{2})^2 \rbrack}.
        \end{eqnarray*}

\noindent \emph{1D Semilinear parabolic.} One-dimensional semilinear parabolic problem \cite{Hoch}:
					\begin{equation*}
					\frac{\partial U}{\partial t}(x,t)-\frac{\partial^2 U}{\partial x^2}(x,t)=\int_0^1 U(x,t)dx + \Phi(x,t)
					\end{equation*}		
		  with homogeneous Dirichlet boundary conditions and for $x\in [0,1]$ and $t\in [0,1]$.  The source function $\Phi$ is chosen such that $U(x,t)=x(1-x)e^t$ is the exact solution.

\subsection{Verification of accuracy}
		The implementation of the exponential integrators was done in MATLAB.  For all the integrators, the adaptive Krylov algorithm as described in \cite{tokmanadapt} was used to compute products of matrix $\varphi$-functions and vectors.  To verify the order of the schemes, we set the tolerance to $10^{-14}$ so that the Krylov approximation error has a minimal effect on the time stepping error.  A reference solution was computed using MATLAB's $\textit{ode15s}$ integrator with absolute and relative tolerances set to $10^{-14}$ and the error was defined as the discrete infinity (maximum) norm of the difference between the computed solution and this approximation.

Figure \ref{fig:loglog1} shows the achieved order of all the methods we are considering.  For convenience we included lines (without markers) of slope three (dashed), four (dash-dotted), and five (dotted). It can be seen from Figure \ref{fig:loglog1} that the predicted order of each stiffly accurate method is attained whereas the classically derived EPIRK5P1 achieves fifth-order for all but one exception, semilinear parabolic problem (Figure \ref{fig:loglog1}(c)) where a significant order reduction is observed. Note that there is not guarantee of order reduction if a method is derived classically and does not satisfy the stiff order conditions. Whether order reduction is observed depends on the specifics of the differential operator of a given problem.  As Figure \ref{fig:loglog1} illustrates it is possible that a classical method does not exhibit order reduction even if a problem is stiff.  Figure \ref{fig:loglog1} also confirms that the same numerical solution is obtained for each version of the exponential-Krylov implementation.  

\begin{figure}[hbtp]
\begin{center}
   \subfigure[2D ADR $(N=400^2$)]{\includegraphics[scale = 0.4]{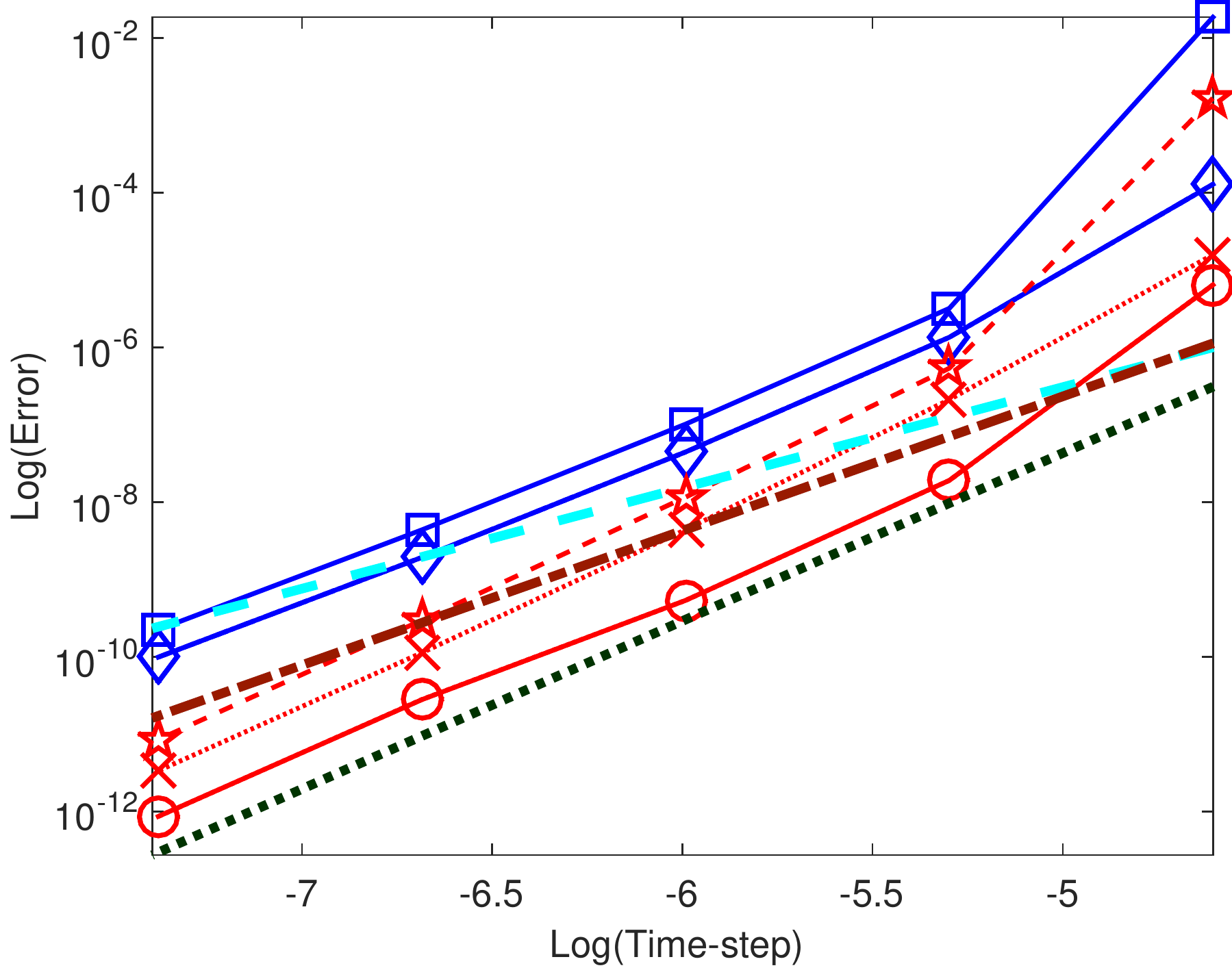}}  \hspace{10pt}
    \subfigure[2D Allen-Cahn $(N=500^2)$]{\includegraphics[scale = 0.4]{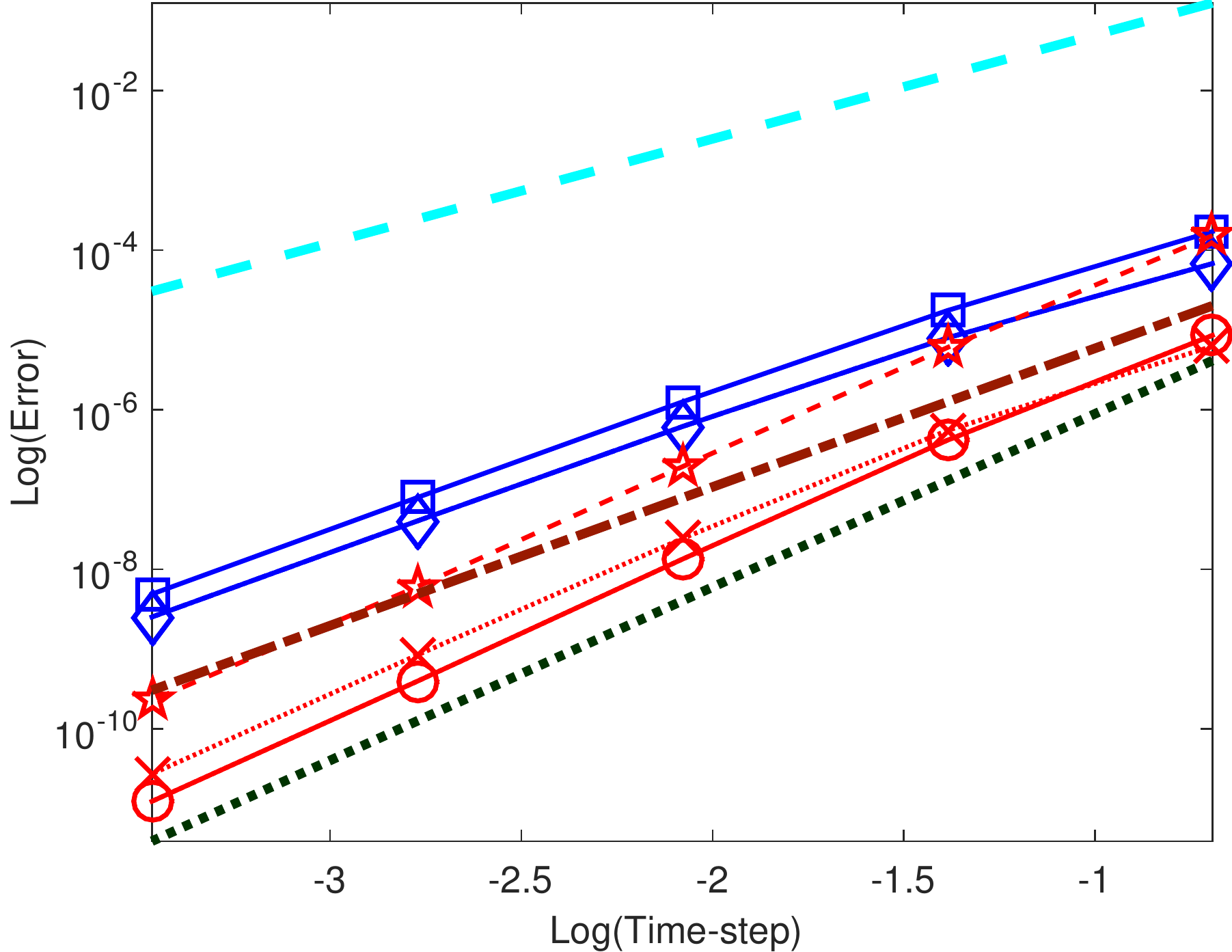}}\\
   \subfigure[1D Semilinear Parabolic N=1000]{$\begin{array}{c}\includegraphics[scale = 0.4]{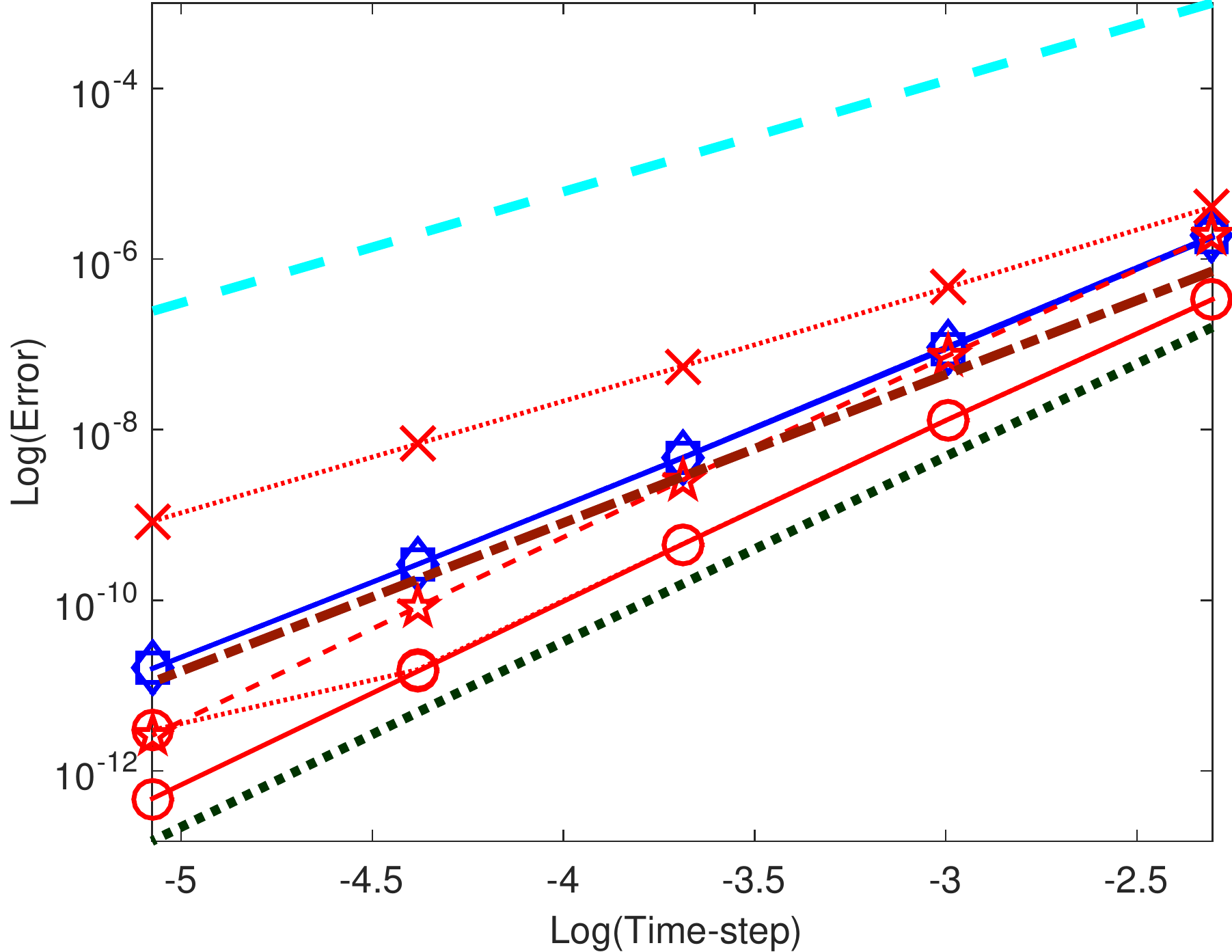}\\ \end{array}$}  \hspace{10pt} 
   \subfigure{$\begin{array}{c} {\includegraphics[scale = 0.2]{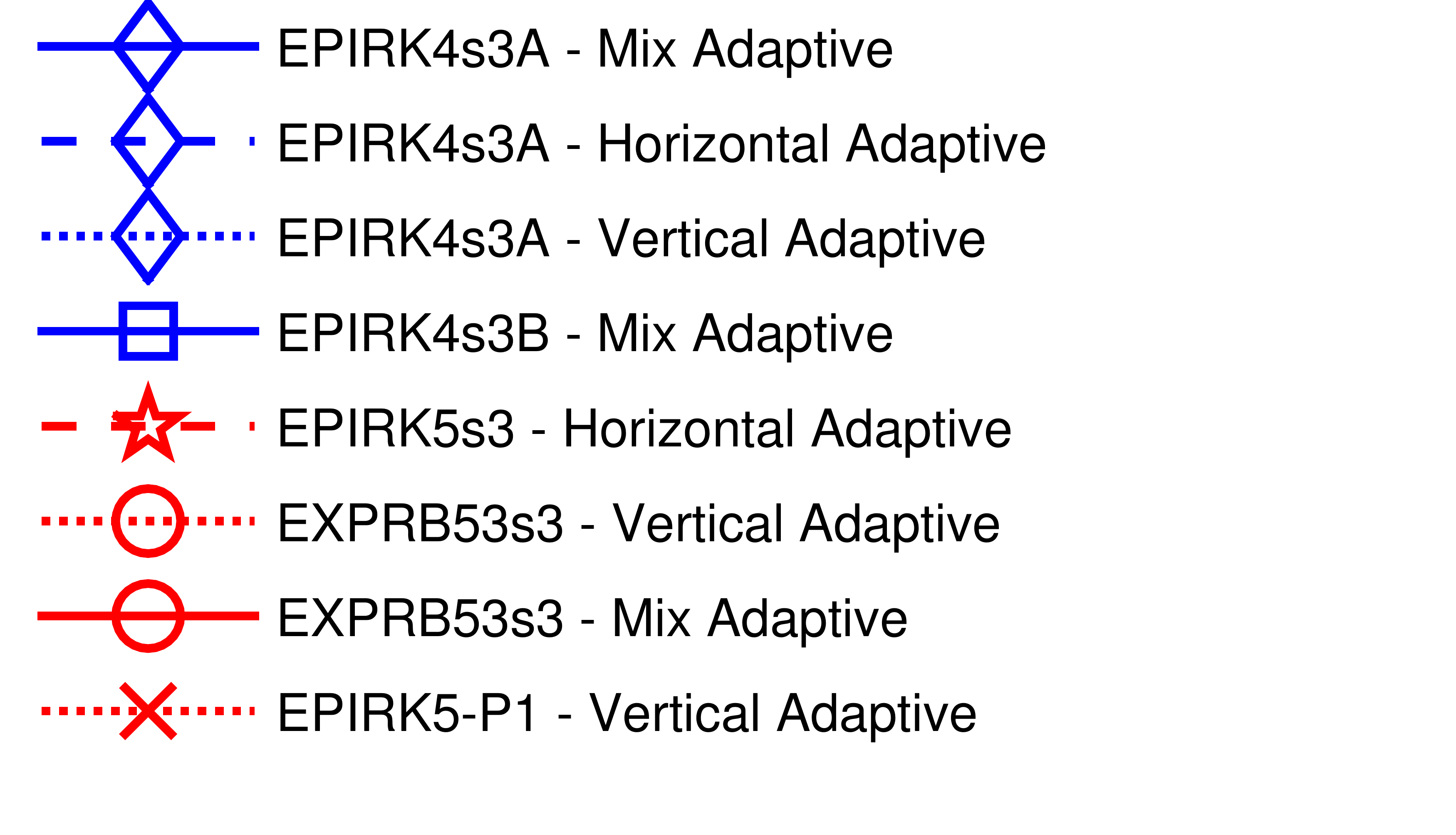}}\\ \textrm{  }\end{array}$ }  \\
  		\addtocounter{subfigure}{-1}
    \subfigure[2D Gray-Scott $(N=400^2)$]{\includegraphics[scale = 0.4]{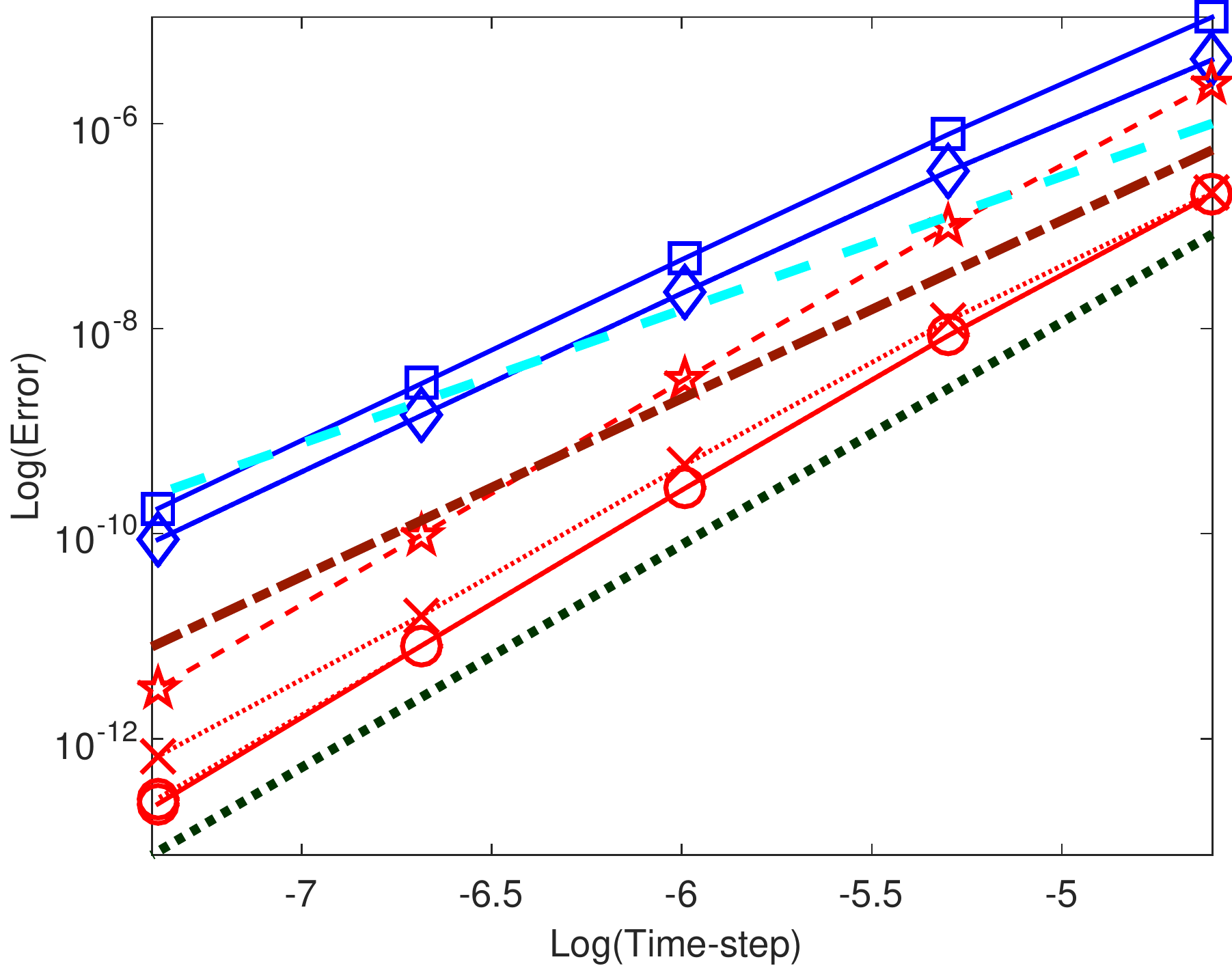}}\hspace{10pt}
    \subfigure[2D Brusselator $(N=300^2)$]{\includegraphics[scale = 0.4]{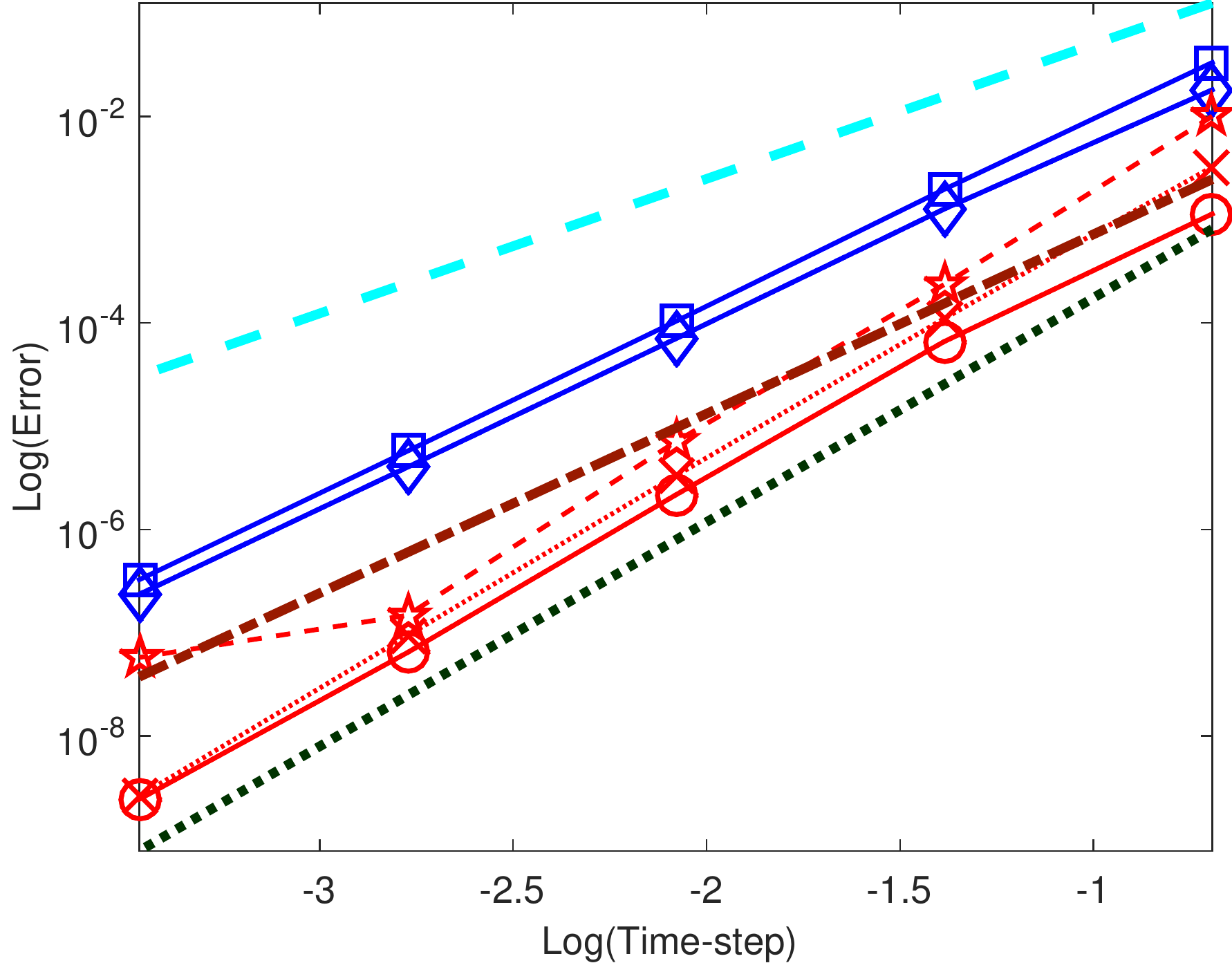}}
    \end{center}
    \caption{Log-log plots of error vs. time step size. For convenience the lines with slopes equal to three (dashed), four (dash-dotted), and five (dotted) are shown.}
    \label{fig:loglog1}
    \end{figure}

\subsection{Comparative performance}
The results of our numerical experiments are presented and analyzed to address the comparative performance of the different exponential-Krylov implementations and the new stiffly accurate EPIRK schemes themselves.  Our comparisons are based on the analysis of precision diagrams for the following simulations:
\begin{itemize}
	\item ADR: $N=400^2$ with $h=0.01, 0.005, 0.0025,0.00125,6.25e-4$,
	\item Allen-Cahn: $N=500^2$ with $h=0.5, 0.25, 0.1250, 0.0625, 0.03125$,
	\item SemilinearParabolic: $N=1000$ with $h=0.1, 0.05, 0.0250, 0.0125, 0.00625$,
	\item Gray-Scott: $N=400^2$ with $h=0.01, 0.005, 0.0025,0.00125,6.25e-4$,
	\item Brusselator: $N=300^2$ with $h=0.5, 0.25, 0.1250, 0.0625, 0.03125$,
\end{itemize}
where $N$ and $h$ correspond to the spatial discretization and time-step sizes respectively.  The precision diagrams are given in Figure \ref{fig:prec}.  Previously published performance comparisons such as \cite{lofftokEPIcompare} addressed computational issue characteristic of the EPIRK methods in general such as, for example, the C-shape of the precision graphs which is induced by the computational complexity scaling of the Krylov algorithm with respect to the size of the time step. Here we concentrate on the numerical experiments demonstrating the properties of the stiffly accurate and optimized with respect of a particular implementation EPIRK methods.

Overall, the figures verify that the performance of each method highly depends on the number of Krylov evaluations and the size of the interval $[0,g]$ (which depends on the chosen $g$-coefficients) that the adaptive Krylov method has to traverse.  For example, consider the horizontal (dashed) and vertical (dotted) implementation of the fourth-order {\it EPIRK4s3A} (diamond).  The same number of adaptive-Krylov evaluations per time-step were taken (three) but Figure \ref{fig:prec} shows a considerable difference in the overall computational cost.  By comparing the CPU times for each time-step we can easily see how much computational savings are obtained by using the horizontal adaptive-Krylov algorithm.  Table \ref{table:costcompareHorzVert}(a) displays the maximum, minimum, and average of the cost of {\it EPIRK4s3A}-Vert compared to cost of {\it EPIRK4s3A}-Horz over all time-steps.  Considering all the test problems, the vertical implementation of {\it EPIRK4s3A} costs on average 129\% of the cost of the horizontal implementation.  Similar results are also found when comparing the fifth-order {\it EXPRB53s3}-Vert with the specifically constructed {\it EPIRK5s3}-Horz (Table \ref{table:costcompare} (b)). As we predicted, these savings come from the horizontal implementations ability to make use of $g_{ij}<1$ coefficients by reducing the Krylov basis size. While this advantage should be observed for the vertical implementation of any closely related method of the same order and same number of stages, the amount of savings will depend on the coefficients of the method.

\begin{table}[htb]
\centering
\caption{Comparison of CPU times for {\it EPIRK4s3A}-Vert with {\it EPIRK4s3A}-Horz  and {\it EXPRB53s3}-Vert with {\it EPIRK5s3}-Horz (Cost of VertMethod/Cost of HorzMethod)}
\label{table:costcompareHorzVert}
\begin{tabular}{cc}
\begin{tabular}{llll}\hline 
        &                Max       & Min       & Avg          \\ \hline
{ADR:} &    137 \%   &   114	\%    &   120 \%              \\ \hline
{Allen-Cahn:} &   127  \%   &     	104\%    & 120 \%                \\ \hline
{Brusselator:} &   133\%     &    98	 \%   &     120 \%              \\ \hline
{Gray-Scott:} &    131\%    &   116\%  	    &    126\%               \\ \hline
{Semilinear Parabolic:} &   157   \%  &   143  \%	    &  152\%                  \\ \hline
\end{tabular}
&
\begin{tabular}{llll}\hline 
        &                Max       & Min       & Avg          \\ \hline
{ADR:} &    119\%  & 79\% & 107\%             \\ \hline
{Allen-Cahn:} &   108\%  & 90\% & 102\%                \\ \hline
{Brusselator:} &   107\%  & 94\% & 101\%             \\ \hline
{Gray-Scott:} &    109\%  & 102\% & 106\%             \\ \hline
{Semilinear Parabolic:} &   138\%  & 107\% & 124\%             \\ \hline
\end{tabular}\\
(a)CPU times: {\it EPIRK4s3A}-Vert/{\it EPIRK4s3A}-Horz & (b) CPU times: {\it EXPRB53s3}-Vert / {\it EPIRK5s3}-Horz
\end{tabular}
\end{table}

The vertical and horizontal implementation of {\it EPIRK4s3A} require three Krylov evaluations each time-step.  The mixed implementation of {\it EPIRK4s3A} only requires two Krylov projections and therefore it is expected this method will further increase the savings compared to the vertically implemented {\it EPIRK4s3A}.  Our numerical experiments confirm {\it EPIRK4s3A}-Mix has a clear advantage over both its horziontal and vertical implementations and can offer up to 50\% savings (compared to its vertical implementation).  The maximum/minimum/average of the per time-step comparisons are given in Table \ref{table:costcompare} and plots of CPU execution time versus error in Figure \ref{fig:prec}.  In the case where the number of Krylov evaluations are the same (i.e. fifth-order three-stage methods), the mixed implementation can still offer computational savings over the vertical implementation but is really dependent on the $g$-coefficients. For example, the difference in performance between the mixed and vertical implementations  of {\textit EXPRB53s3} is much smaller due to $g_{31}=9/10$. For this value the resulting intervals $[0,g]$ are nearly the same and no significant savings are obtained.  We will pursue further optimization of coefficients strategies for horizontal and mixed implementations in our future research.

\begin{figure}[hbtp]
\begin{center}
   \subfigure[2D ADR $(N=400^2$)]{\includegraphics[scale = 0.4]{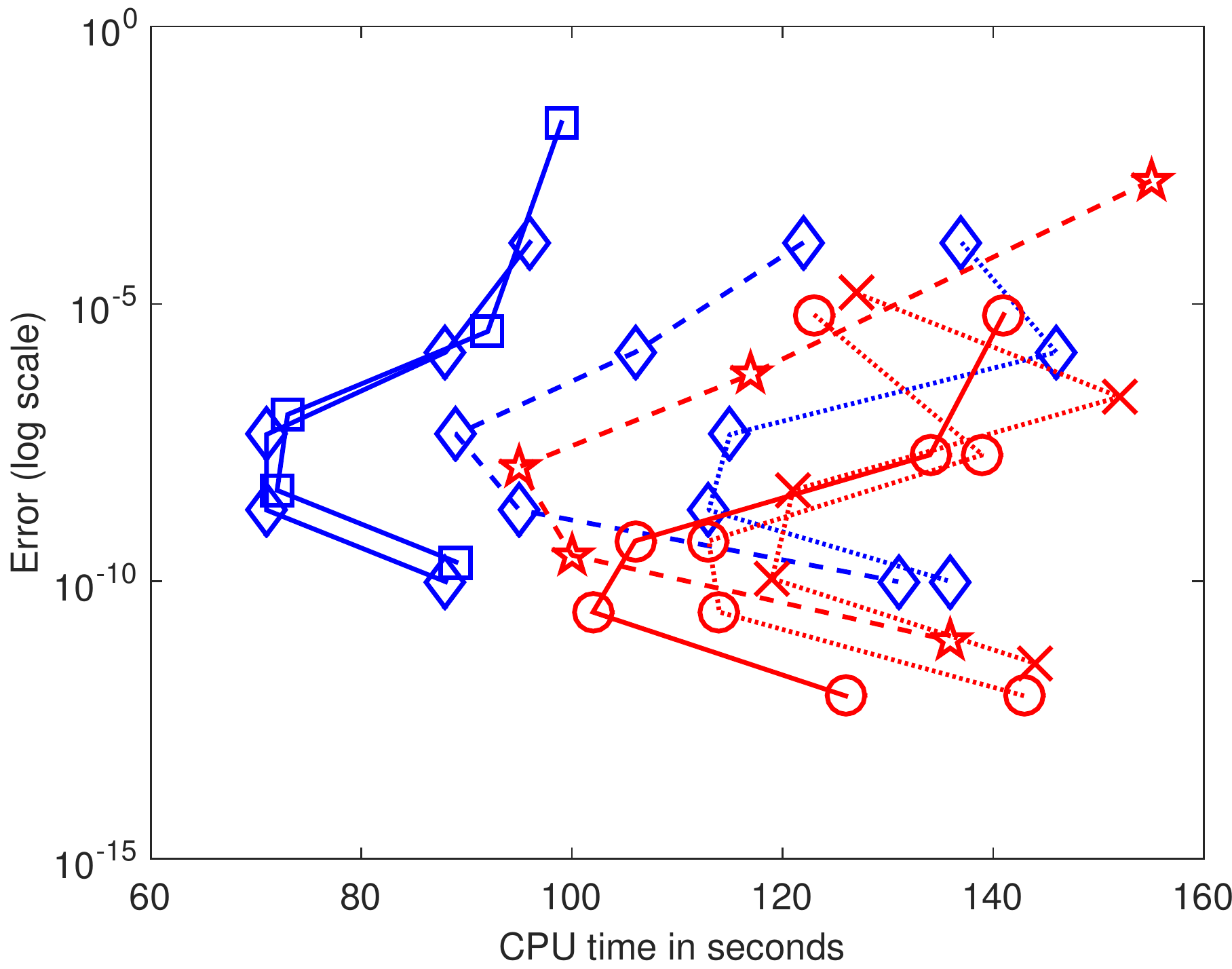}}  \hspace{10pt}
    \subfigure[2D Allen-Cahn $(N=500^2)$]{\includegraphics[scale = 0.4]{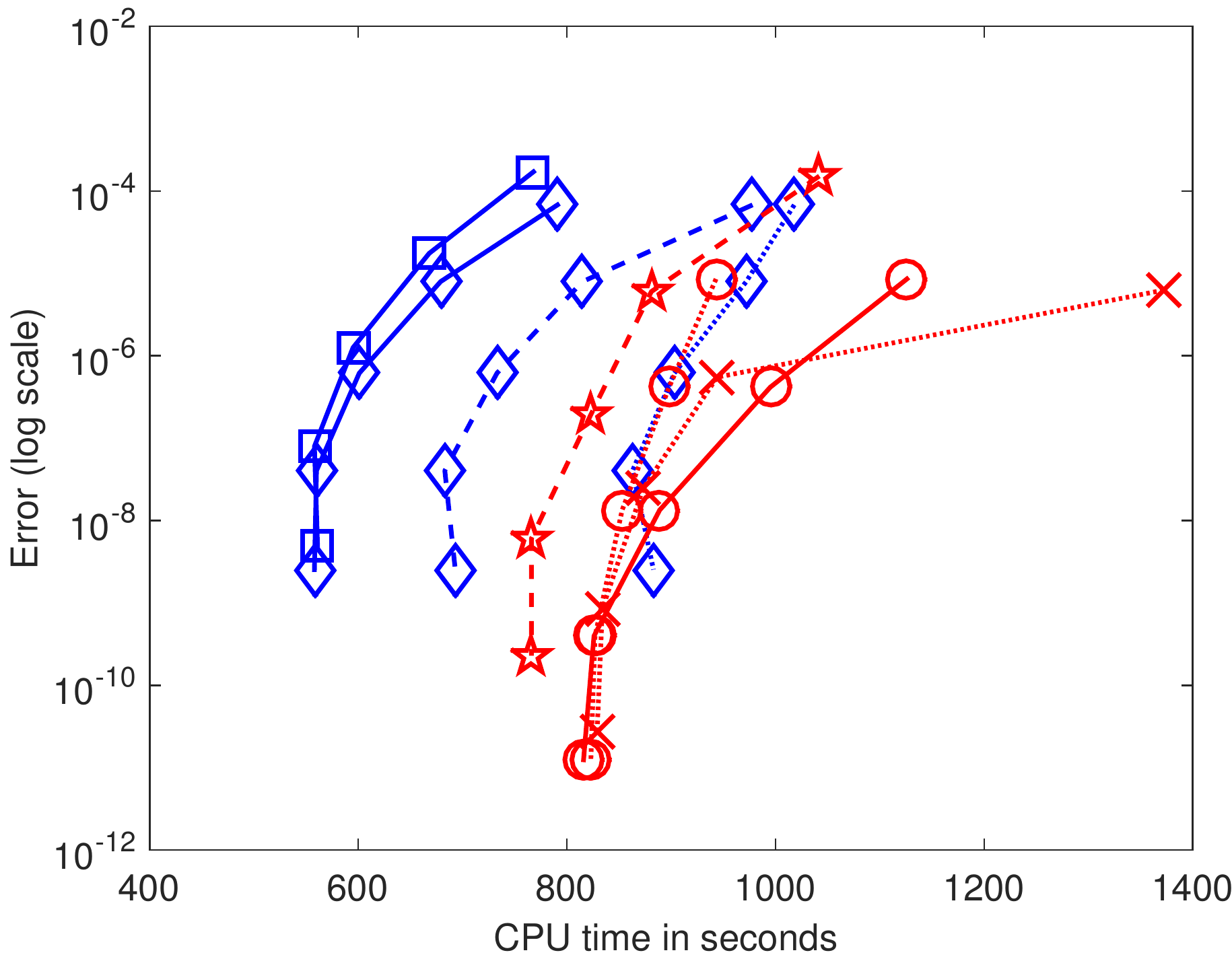}}\\
   \subfigure[1D Semilinear Parabolic N=1000]{$\begin{array}{c}\includegraphics[scale = 0.4]{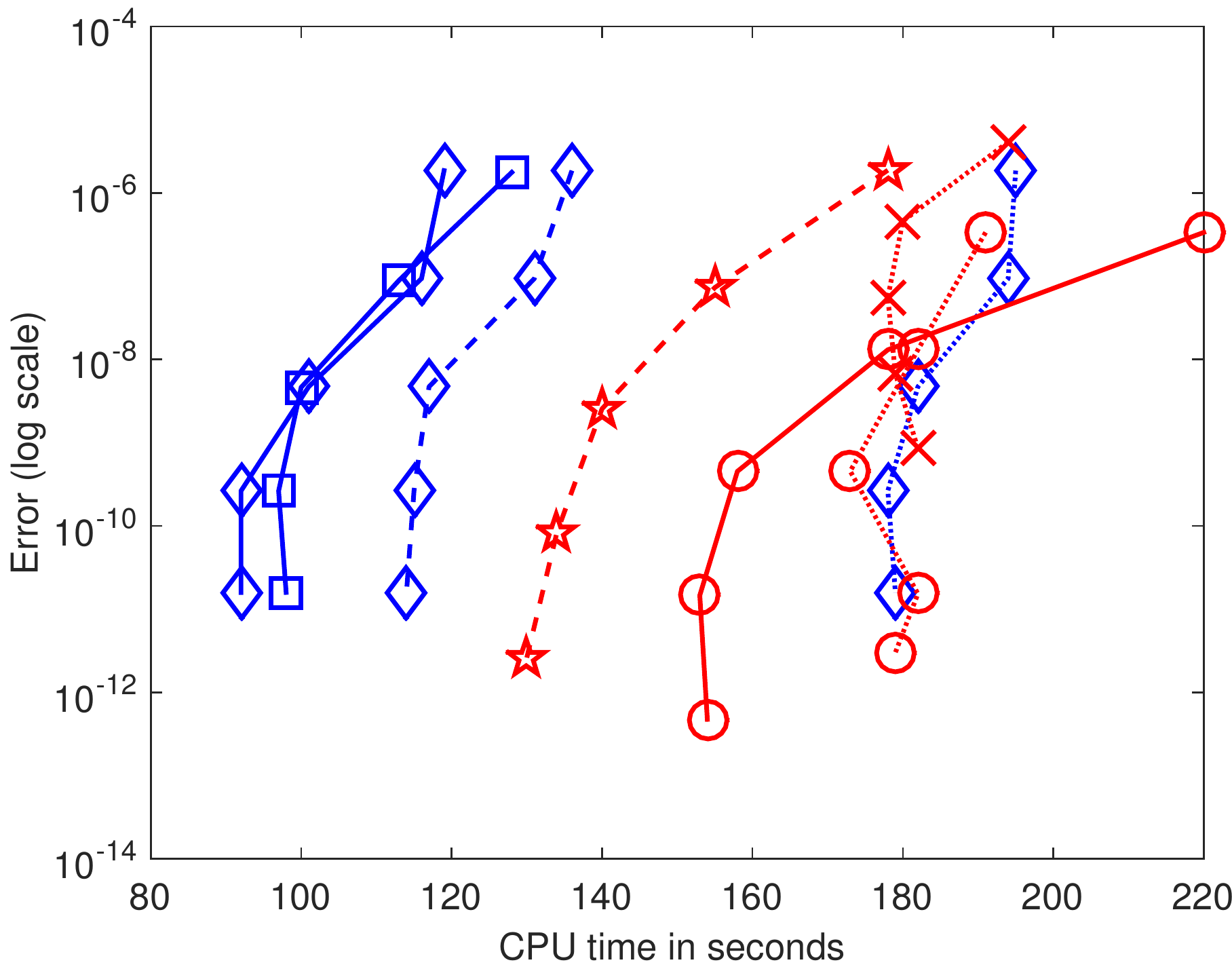}\\ \end{array}$}  \hspace{10pt} 
   \subfigure{$\begin{array}{c} {\includegraphics[scale = 0.2]{stiffpaperlegend}}\\ \textrm{ }\end{array}$ }  \\
  		\addtocounter{subfigure}{-1}
    \subfigure[2D Gray-Scott $(N=400^2)$]{\includegraphics[scale = 0.4]{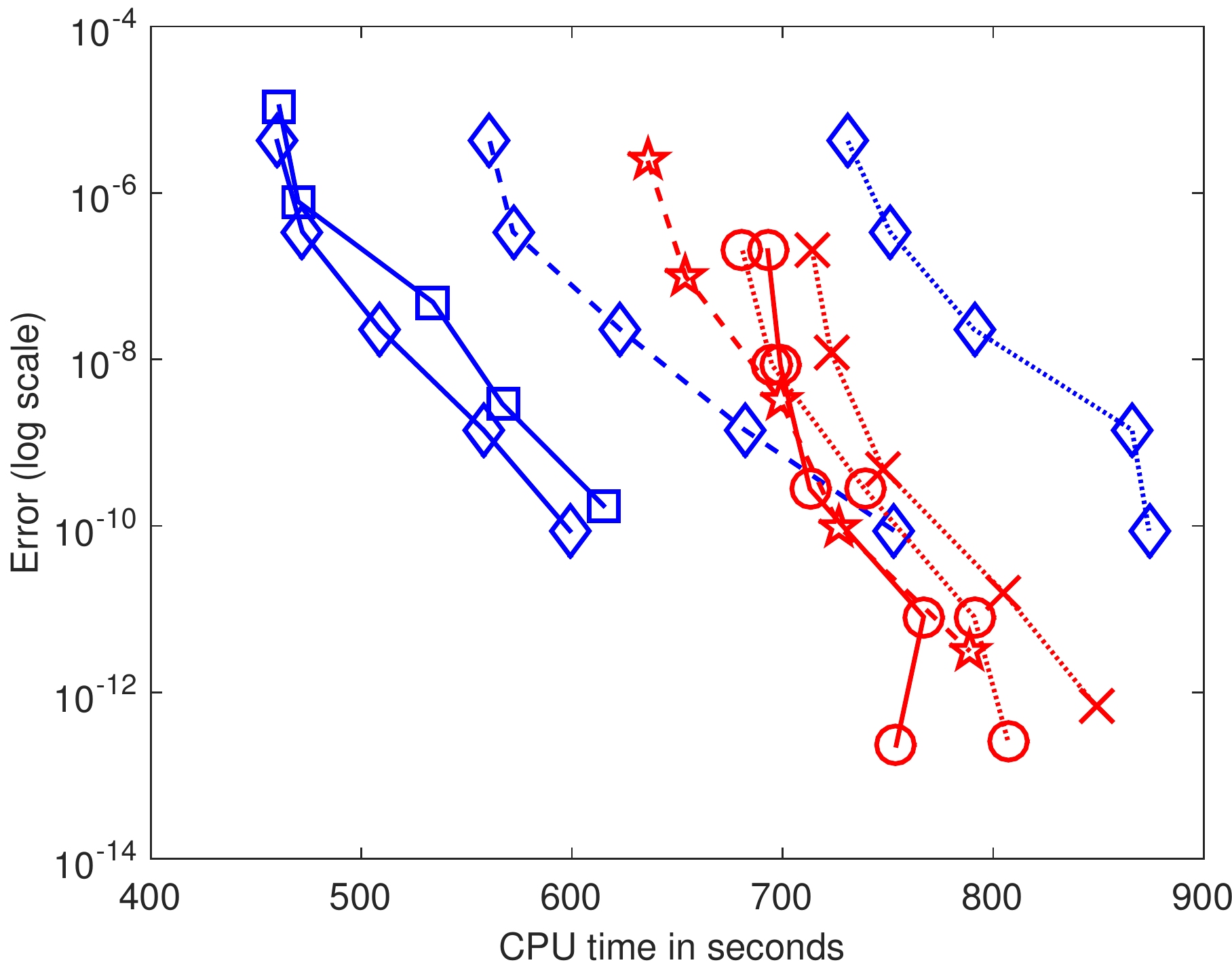}}\hspace{10pt}
    \subfigure[2D Brusselator $(N=300^2)$]{\includegraphics[scale = 0.4]{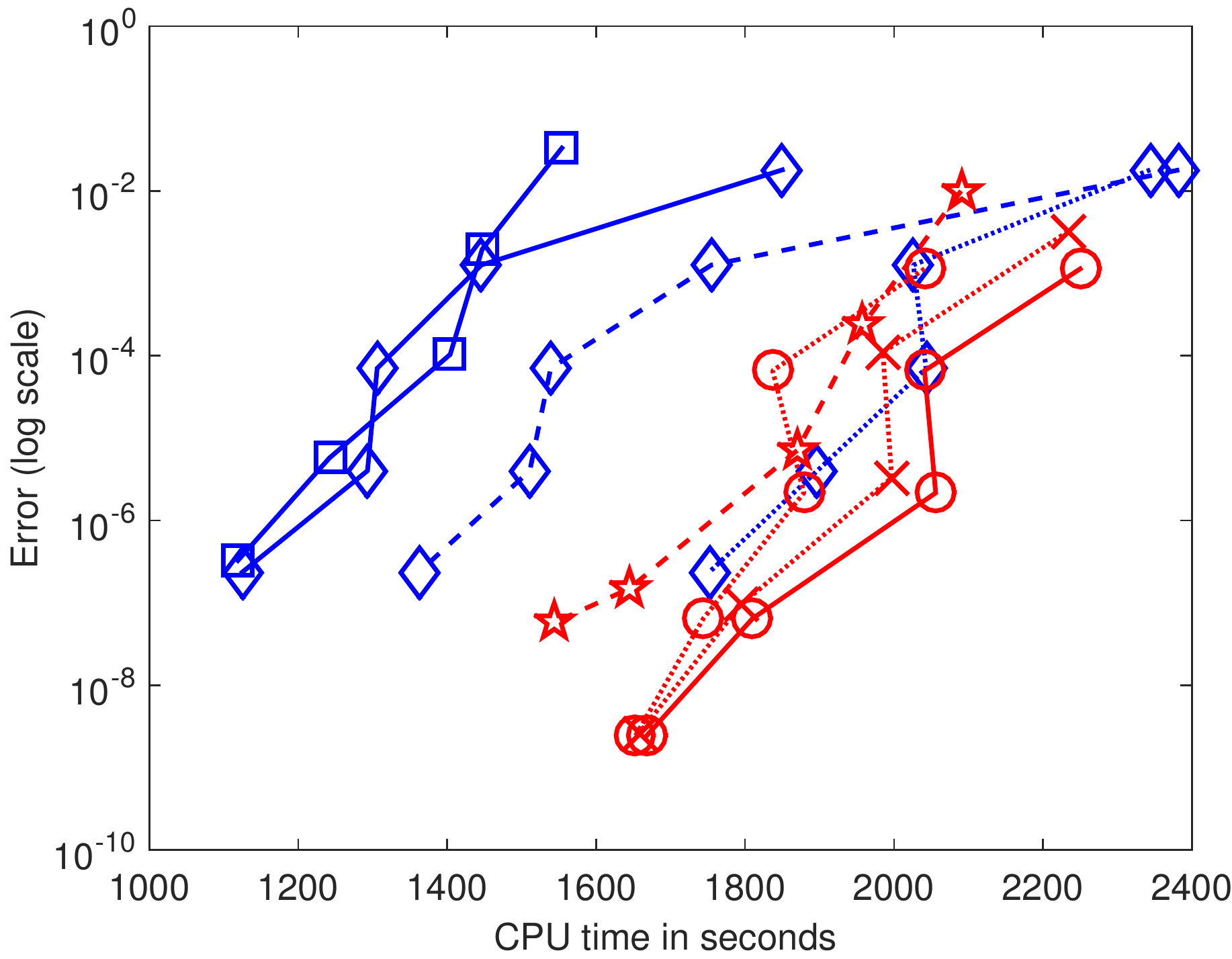}}
    \end{center}
   \caption{CPU execution time versus error for constant time step experiments}
    \label{fig:prec}
    \end{figure}

%
%

\begin{table}[htb]
\centering
\caption{Cost of horizontal and vertical implementations in comparison to mixed implementation}
\label{table:costcompare}
\begin{tabular}{lllllllll}
                  & \multicolumn{4}{c}{Comparisons to EPIRK4s3A-Mixed} &   \multicolumn{4}{c}{Comparisons to EXPRB5s3-Mixed} \\
        &                      & Max       & Min       & Avg       &  & Max              & Min              & Avg             \\ \hline
\multirow{2}{*}{ADR:}&EPIRK4s3A-Horz & 149\%  & 120\% & 131\% &  EPIRK5s3-Horz & 110\%  & 87\% & 99\% \\
 &EPIRK4s3A-Vert & 166\%  & 143\% & 157\% & EXPRB53s3-Vert & 113\%  & 87\% & 105\% \\ \hline
\multirow{2}{*}{Allen-Cahn:} &EPIRK4s3A-Horz & 124\%  & 120\% & 122\% &  EPIRK5s3-Horz & 94\%  & 89\% & 92\% \\
 &EPIRK4s3A-Vert & 158\%  & 129\% & 147\% & EXPRB53s3-Vert & 101\%  & 84\% & 94\% \\ \hline
\multirow{2}{*}{Brusselator:} &EPIRK4s3A-Horz & 129\%  & 117\% & 121\% &  EPIRK5s3-Horz & 96\%  & 91\% & 93\% \\
 &EPIRK4s3A-Vert & 156\%  & 127\% & 145\% & EXPRB53s3-Vert & 99\%  & 90\% & 94\% \\ \hline
\multirow{2}{*}{Gray-Scott:} &EPIRK4s3A-Horz & 126\%  & 121\% & 123\% &  EPIRK5s3-Horz & 105\%  & 92\% & 97\% \\
 &EPIRK4s3A-Vert & 159\%  & 146\% & 155\% & EXPRB53s3-Vert & 107\%  & 98\% & 102\% \\ \hline
\multirow{2}{*}{Semilinear Parabolic:} &EPIRK4s3A-Horz & 125\%  & 113\% & 118\% &  EPIRK5s3-Horz & 89\%  & 81\% & 86\% \\
 &EPIRK4s3A-Vert & 195\%  & 164\% & 180\% & EXPRB53s3-Vert & 119\%  & 87\% & 107\% \\ \hline                 
\end{tabular}
\end{table}

We now turn to comparing the performance of the schemes themselves. While there is no clear dominate fifth-order method in regards to computational cost we do see that {\it EXPRB53s3}  is slightly more accurate for all problems.  The more interesting comparison is that of the fourth-order method with the fifth-order schemes.  For a prescribed accuracy, the fourth-order mixed (and horizontal) {\it EPIRK4s3A} can offer significant (up to 64\%) savings in comparison to the fifth order methods (of any implementation).  In Table \ref{table:CPUtimes} we list the CPU execution times for each method and each test problem for various tolerances.  For any set tolerance we see that the mixed implementation of {\it EPIRK4s3A} can achieve this level of accuracy at a fraction of the cost of any of the fifth-order methods.   A simple justification is that conditions for a stiffly accurate fifth-order method are far more restrictive than for a fourth-order scheme.  Thus the additional flexibility of stiffly accurate fourth-order schemes allows for more customization and design of methods which optimize the efficiency.

\begin{table}[htb]
\centering
\caption{Approximate CPU times for a given accuracy}
\label{table:CPUtimes}
{\small \begin{tabular}{lllll}
\hline
    Accuracy        & \multicolumn{4}{c}{Approx. CPU times (seconds)}                          \\
  & {\it EPIRK4s3A}-mixed & {\it EPIRK5s3}-horz & {\it EXPRB53s3}-vert & {\it EXPRB53s3}-mixed \\ \hline
$10^{-6}$  &     616        &          851         &  911  &   1034         \\
$10^{-7}$  &     575           &    811                &   879     &   950    \\
$10^{-8}$ &      558         &      775         &   849     &  885    \\
\hline\\
\end{tabular}\\

(a) Allen-Cahn $N=500^2$}\\

{\small \begin{tabular}{lllll}
\\
\hline
 Accuracy          & \multicolumn{4}{c}{Approx. CPU times (seconds)}                          \\
  & {\it EPIRK4s3A}-mixed & {\it EPIRK5s3}-horz & {\it EXPRB53s3}-vert & {\it EXPRB53s3}-mixed \\ \hline
$10^{-7}$  &          74   &   107               &134    &  136         \\
$10^{-8}$  &           71  &     95              &   134      & 128       \\
$10^{-9}$ &           75  &       98         &      117    &     110   \\
\hline\\
\end{tabular}\\

(b) ADR $N=400^2$}\\
{\small \begin{tabular}{lllll}
\\
\hline
 Accuracy          & \multicolumn{4}{c}{Approx. CPU times (seconds)}                          \\
  & {\it EPIRK4s3A}-mixed & {\it EPIRK5s3}-horz & {\it EXPRB53s3}-vert & {\it EXPRB53s3}-mixed \\ \hline
$10^{-7}$  &  116            &      156          &  187  &  204     \\
$10^{-8}$  &     104        &      145            &  181      &   176   \\
$10^{-9}$ &       96     &         138         &    175    &   162     \\
\hline\\
\end{tabular}\\
(c) Semilinear Parabolic $N=1000$}

{\small \begin{tabular}{lllll}
\\
\hline
 Accuracy          & \multicolumn{4}{c}{Approx. CPU times (seconds)}                          \\
  & {\it EPIRK4s3A}-mixed & {\it EPIRK5s3}-horz & {\it EXPRB53s3}-vert & {\it EXPRB53s3}-mixed \\ \hline
$10^{-7}$  &     488         &      655          &  683  & 695      \\
$10^{-8}$  &       523      &    684              &    694    & 698     \\
$10^{-9}$ &         562   &   708               &     722   &      708  \\
\hline\\
\end{tabular}\\
(d) Gray-Scott $N=400^2$}

{\small \begin{tabular}{lllll}
\\
\hline
 Accuracy          & \multicolumn{4}{c}{Approx. CPU times (seconds)}                          \\
  & {\it EPIRK4s3A}-mixed & {\it EPIRK5s3}-horz & {\it EXPRB53s3}-vert & {\it EXPRB53s3}-mixed \\ \hline
$10^{-4}$  &      1321        &    1937            & 1861   & 2076      \\
$10^{-5}$  &     1295        &      1877            & 1859       &   2050   \\
$10^{-6}$ &    1204        &     1754             &   1847     &   1999     \\
\hline\\
\end{tabular}\\
(d) Brusselator $N=300^2$}
\end{table}

\subsubsection{Non-homogeneous boundary conditions}
As mentioned in Section \ref{subsec:analyticFramework} problems with non-homogeneous boundary conditions do not necessarily satisfy the assumptions of our framework and therefore the stiff order is not guaranteed. The purpose of this section is to show that order reduction occurs and identify how much of a reduction to expect for these problems.  We perform simulations with the following test problems: 
\begin{itemize}
	\item Allen-Cahn 2d: Neumann boundary conditions with initial and boundary values given by $$u=0.4+0.1(x+y)+0.1\sin\left(\frac{3}{2}\pi x\right)\sin\left(\frac{5}{2}\pi y\right). $$
    
	\item Brusselator 2d: Dirichlet boundary conditions with initial and boundary values given by
	$$\begin{aligned}
&	u = 1 + \sin(2\pi x)\sin(2\pi y)\\
&	v=3
	\end{aligned} $$
		\item 1D Degenerate nonlinear diffusion  \cite{sherratt}: 
	$$ \frac{\partial u}{\partial t}=\frac{\partial }{\partial x}\left[u\frac{\partial u}{\partial x}\right] +u(1-u), \qquad x\in (-23,50), \, t\in [0,50]$$
  with Dirichlet boundary conditions $u(-23,t)=1$ and $u(50,t)=0$, and  initial conditions $$u(x,0)=\left\{\begin{array}{ll} 
	1 & \textrm{ if } x<0\\
	e^{-1.3x} & \textrm{ if } x>0  . 
\end{array}	\right. $$
\end{itemize}
The same spatial discretization and time-step sizes were used for the Brusselator problem as in the previous section.  The Allen-Cahn and degenerate nonlinear diffusion problem were conducted with time-step sizes $h=0.05,0.0250,0.0125,0.00625,0.003125$  and respective discretization sizes of $N=500^2$ and $N=1000$.

Figure \ref{fig:loglogNonHom} displays the log-log plots of time-step size versus error and Table \ref{table:ApproxOrder} has the approximate order exhibited by each of the method for every test problem.  While some of the methods achieve full order for some problems, generally the results illustrate that a reduction of order is possible even if the method is stiffly accurate.  The extent of the order reduction ranges from 0.03 to 1.34. Such reduction is expected since a similar phenomenon occurs for implicit methods.  A theory presented in \cite{OsterRoche1, OsterRoche2} allows to quantify the extent of order reduction for Rosenbrock methods.  We plan to pursue development of a similar theory for exponential integrators applied to nonhomogeneous problems in our future research.

  \begin{figure}[hbtp]
    \begin{center}$
    \begin{array}{c}
    \subfigure[2D Brusselator $(N=300^2)$]{\includegraphics[scale = 0.4]{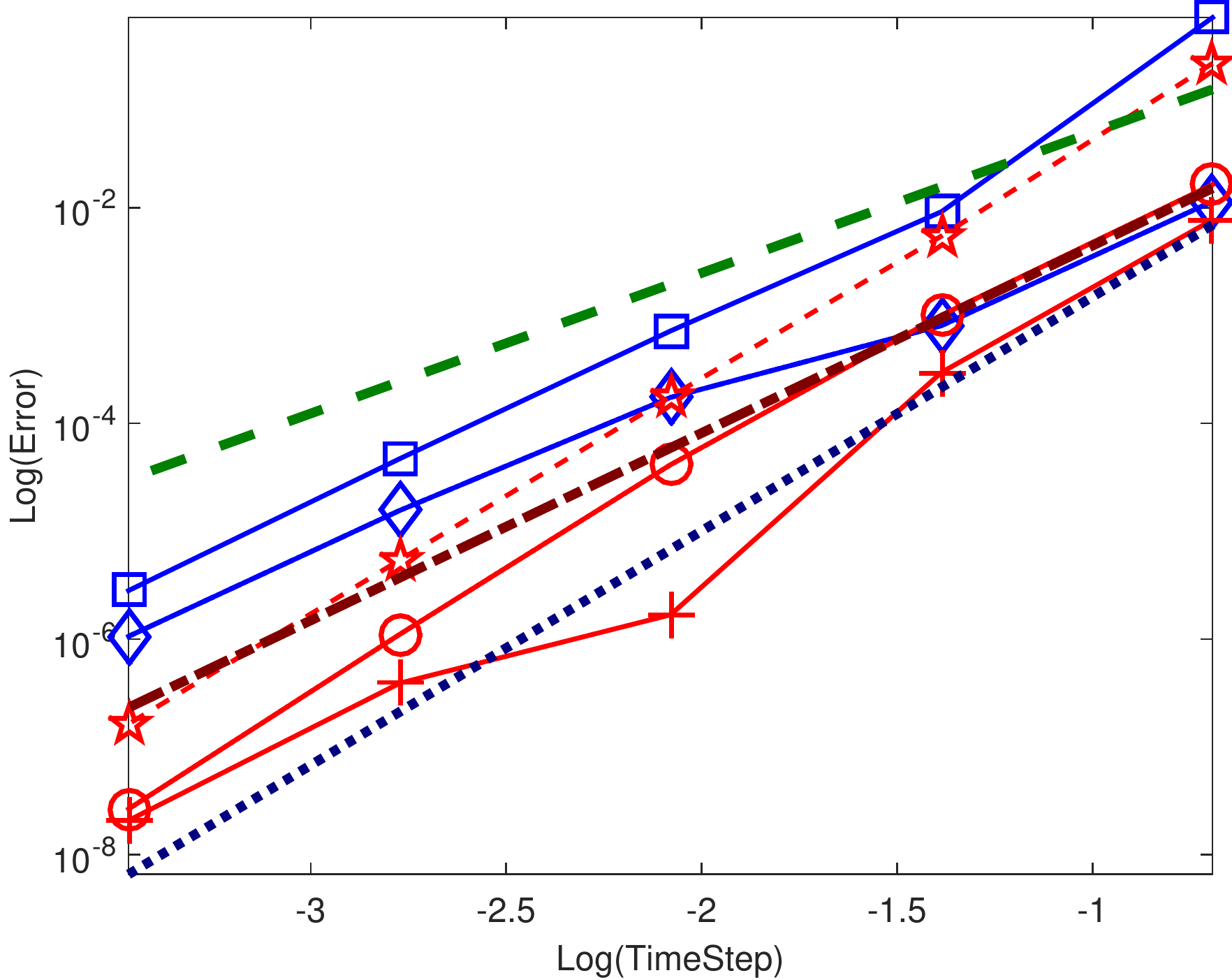}}\hspace{10pt} 
   {\includegraphics[scale = 0.2]{stiffpaperlegend}\vspace*{40pt}}\\
      \subfigure[1D Degenerate Nonlinear Diffusion ($N=1000$)]{\includegraphics[scale = 0.4]{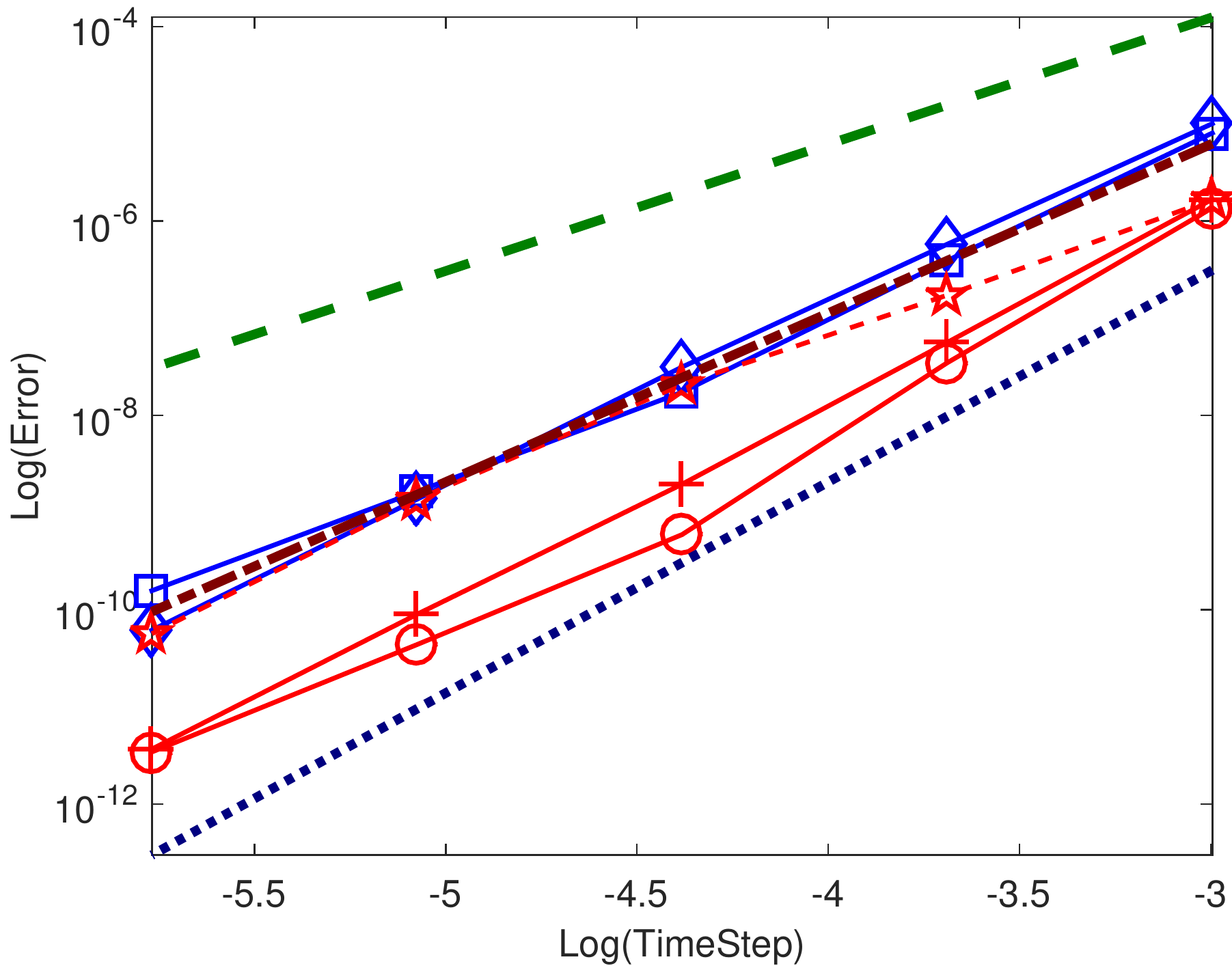}}  \hspace{10pt}
        \subfigure[2D Allen-Cahn $(N=500^2)$]{\includegraphics[scale = 0.4]{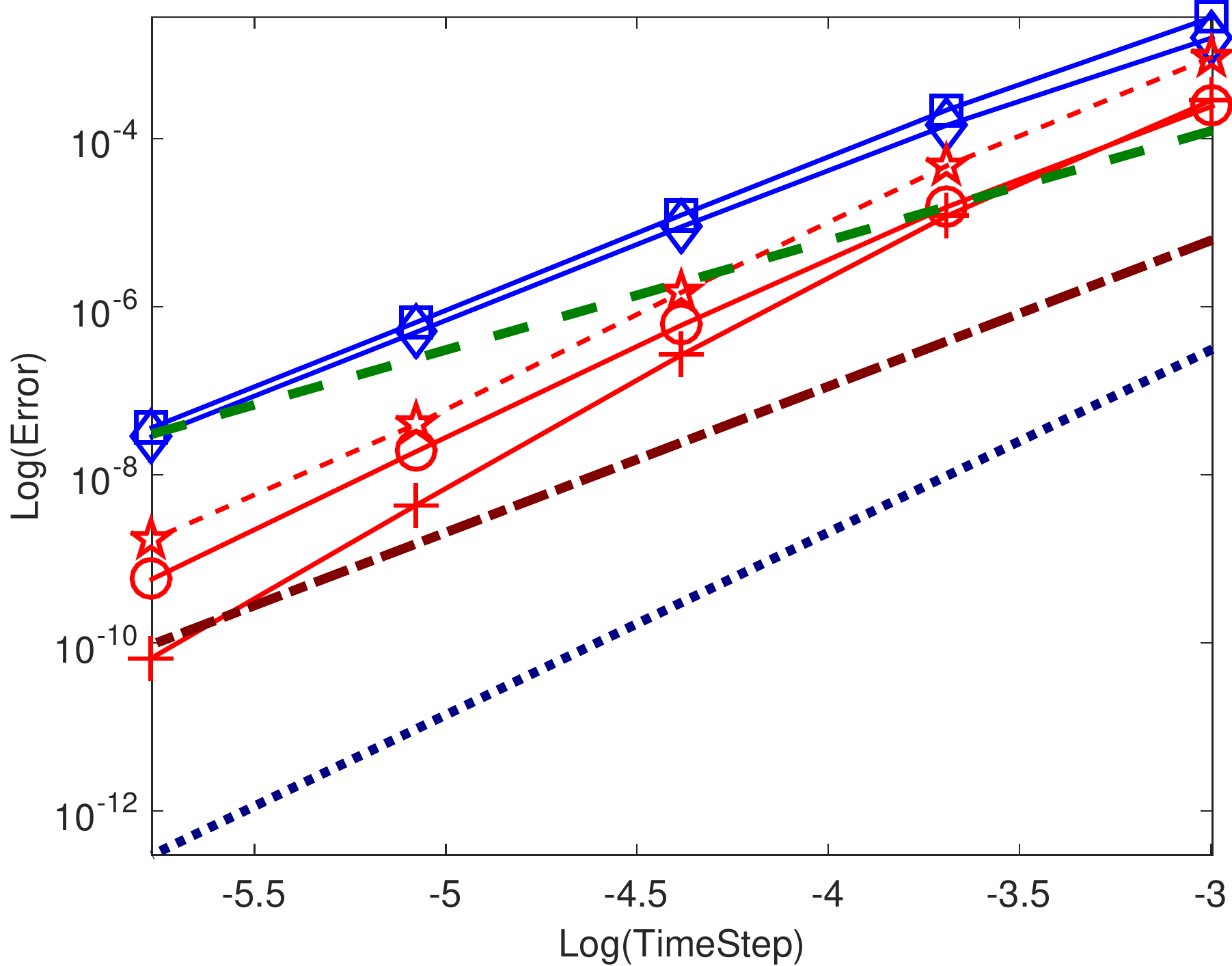}}\\
          
    \end{array}$
    \end{center}
    \caption{ Log-log plots of error vs. time step size for problems with non-homogeneous boundary conditions. For convenience the lines with slopes equal to three (dashed), four (dash-dotted), and five (dotted) are shown.}
    \label{fig:loglogNonHom}
    \end{figure}

\begin{table}[]
\centering
\caption{Approximate order of methods for problems with non-homogeneous boundary conditions}
\begin{tabular}{lccccc}
            & \textit{EPIRK4s3A} & \textit{EPIRK4s3B} & \textit{EPIRK5s3} & \textit{EXPRB53s3} & \textit{EPIRK5-P1} \\ \hline
Allen-Cahn:  & 3.97              & 4.08              & 4.84              & 4.71            & 5.56\\
Brusselator: & 3.25             & 4.29                     & 5.06              & 4.83              & 4.66               \\
Deg NL Diff:         & 4.34              & 3.92              &     3.66        & 4.68                & 4.69             
\end{tabular}
\label{table:ApproxOrder}
\end{table}

\subsection{Variable time-step comparisons}

We present here the results of our variable time-step experiments on tests problems described in Section \ref{subsec:testproblems}.  In addition to the stiffly accurate schemes from Section \ref{subsec:variabletimestep} we will also consider the fifth-order classical (non-stiff)  {\it EPIRK5-P1} method with a fourth-order error estimator \cite{lofftokEPIcompare}.  We used the same configuration for our experiments as in \cite{lofftokEPIcompare}.  For each problem, five runs were made with the following absolute and relative tolerances $Atol=Rtol=10^{-2},10^{-3},10^{-4},10^{-5},10^{-6}$. The resulting diagrams of CPU execution time versus error are displayed in Figure \ref{fig:precVar}.  The classically derived {\it EPIRK5-P1} shows to be the most efficient method across all the problems but has the potential drawback of suffering from a reduction of order as seen with the semilinear parabolic problem. This further confirms the need for more efficient stiffly accurate methods as well as illustrates the need for a more refined theory that predicts how much order reduction can be expected for a given problem and a chosen integrator.

    
\begin{figure}[hbtp]
\begin{center}
   \subfigure[2D ADR $(N=400^2$)]{\includegraphics[scale = 0.4]{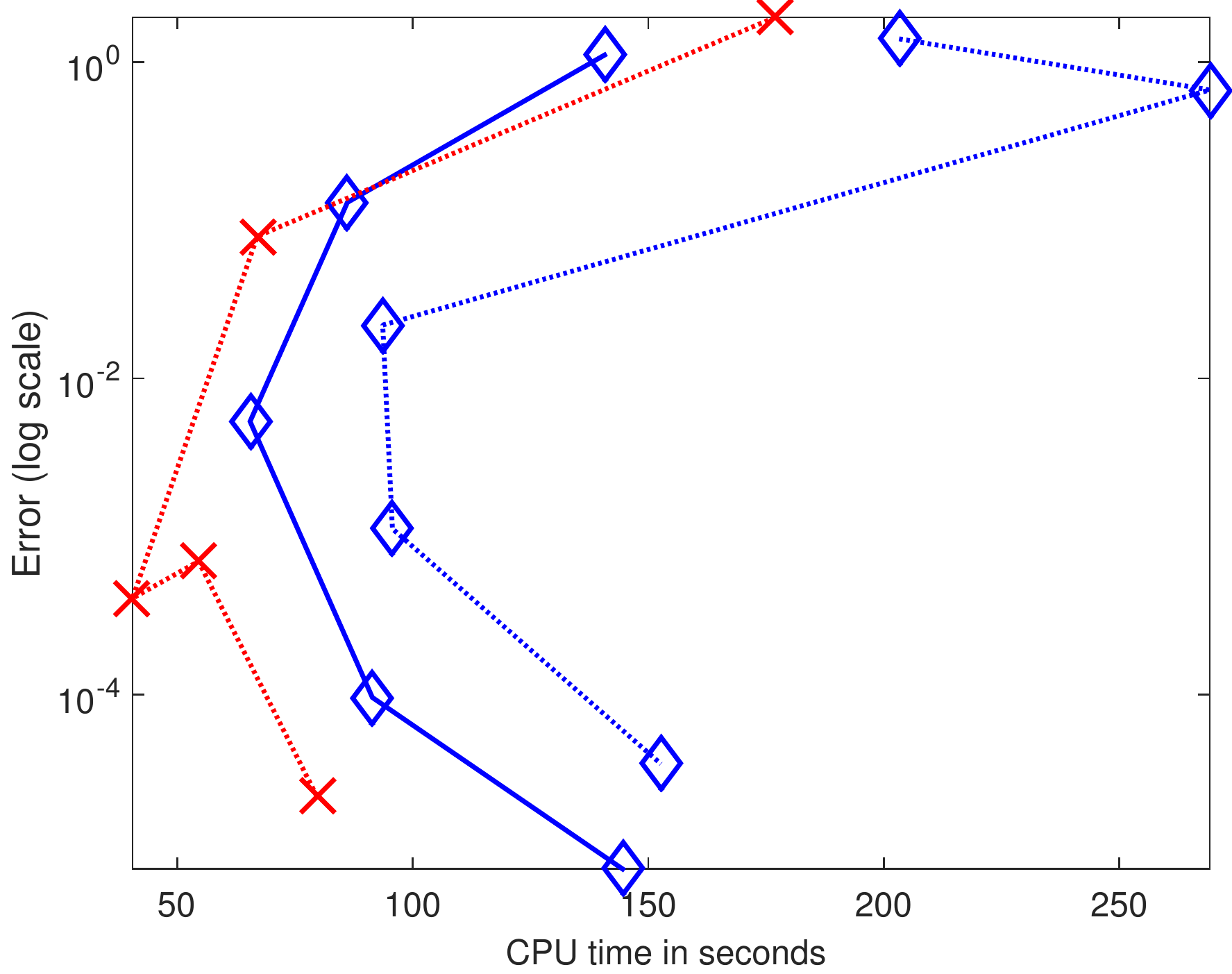}}  \hspace{10pt}
    \subfigure[2D Allen-Cahn $(N=500^2)$]{\includegraphics[scale = 0.4]{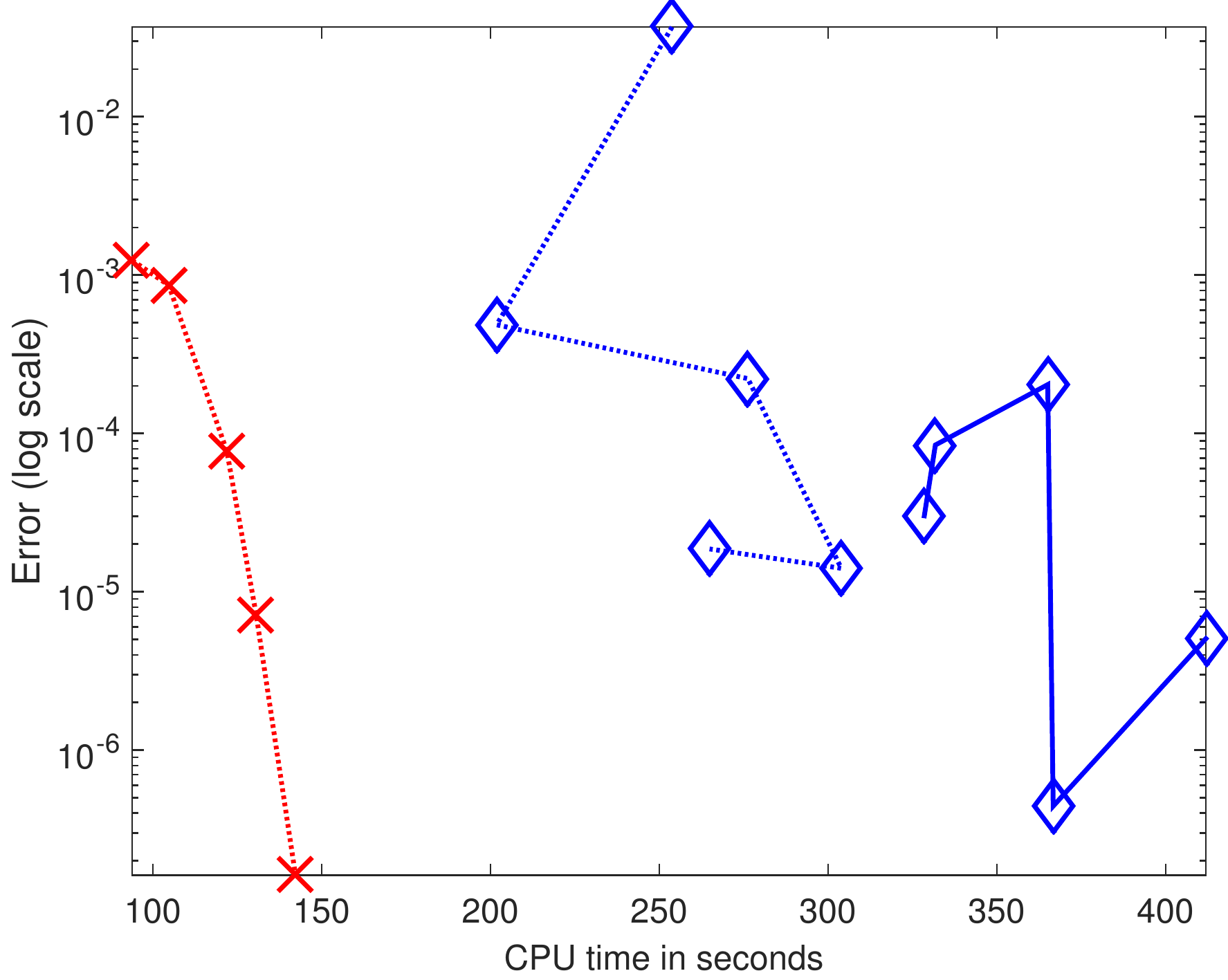}}\\
   \subfigure[1D Semilinear Parabolic N=1000]{$\begin{array}{c}\includegraphics[scale = 0.4]{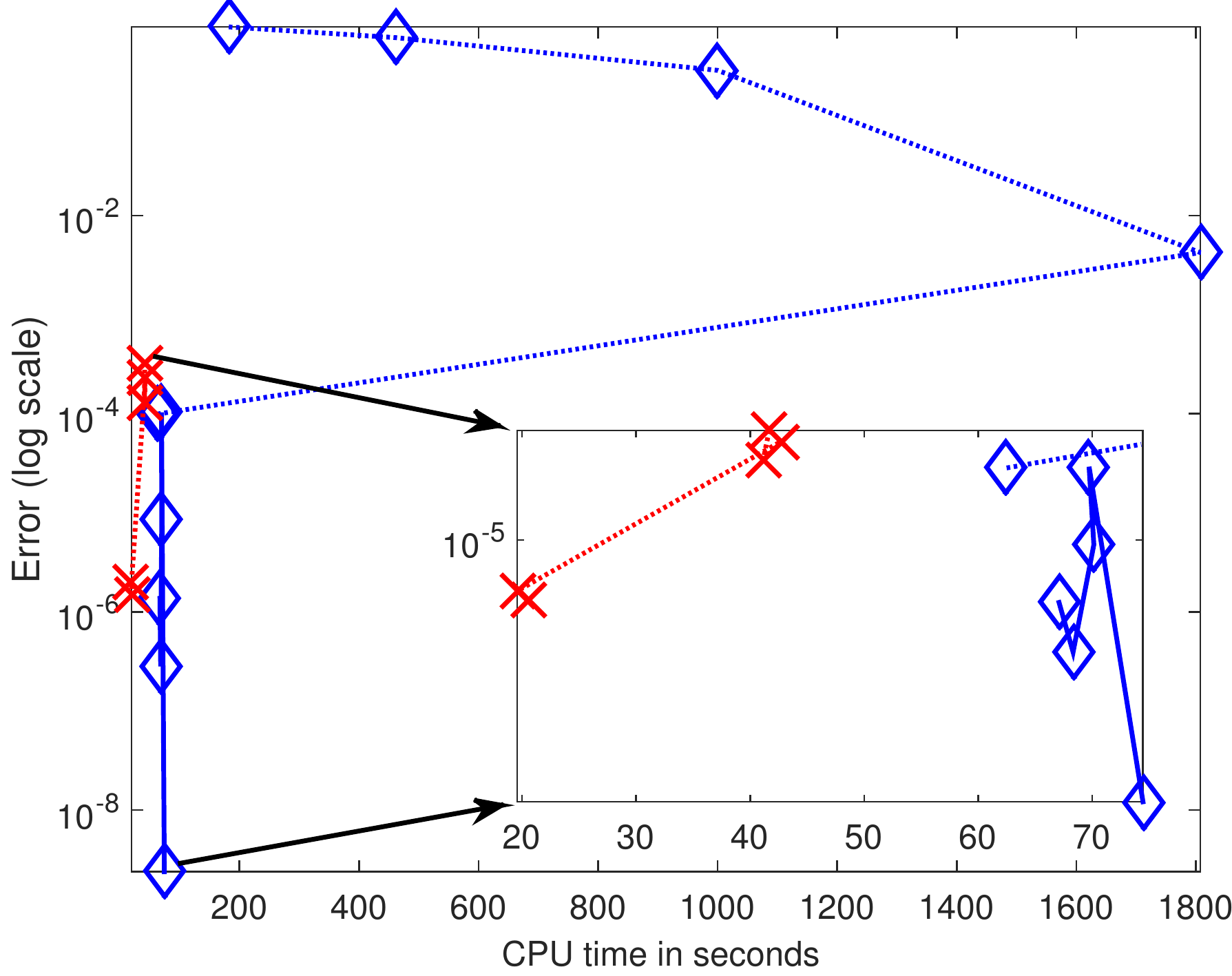}\\ \end{array}$}  \hspace{10pt} 
   \subfigure{$\begin{array}{c} {\includegraphics[scale = 0.2]{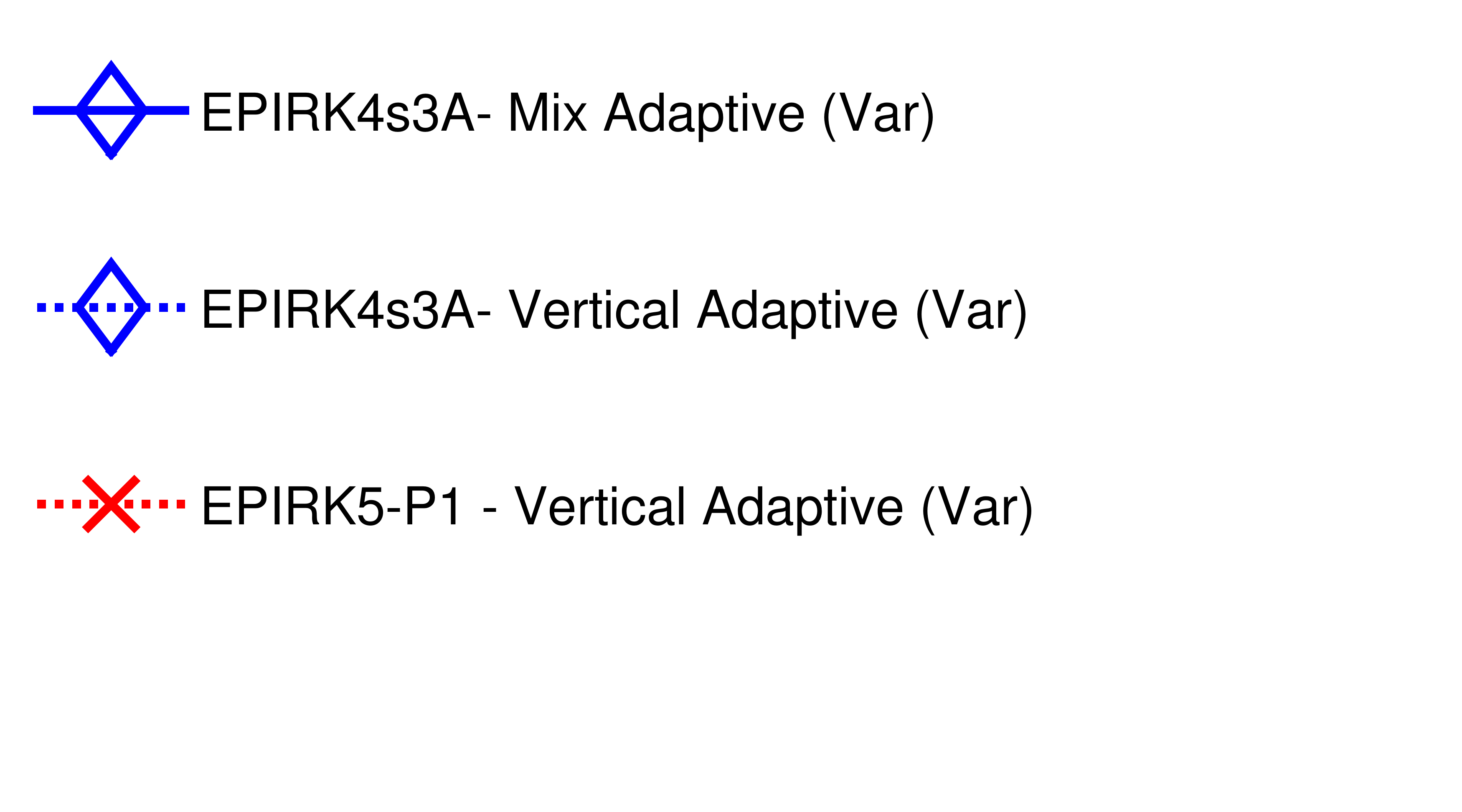}}\\ \textrm{ }\end{array}$}  \\
  		\addtocounter{subfigure}{-1}
    \subfigure[2D Gray-Scott $(N=400^2)$]{\includegraphics[scale = 0.4]{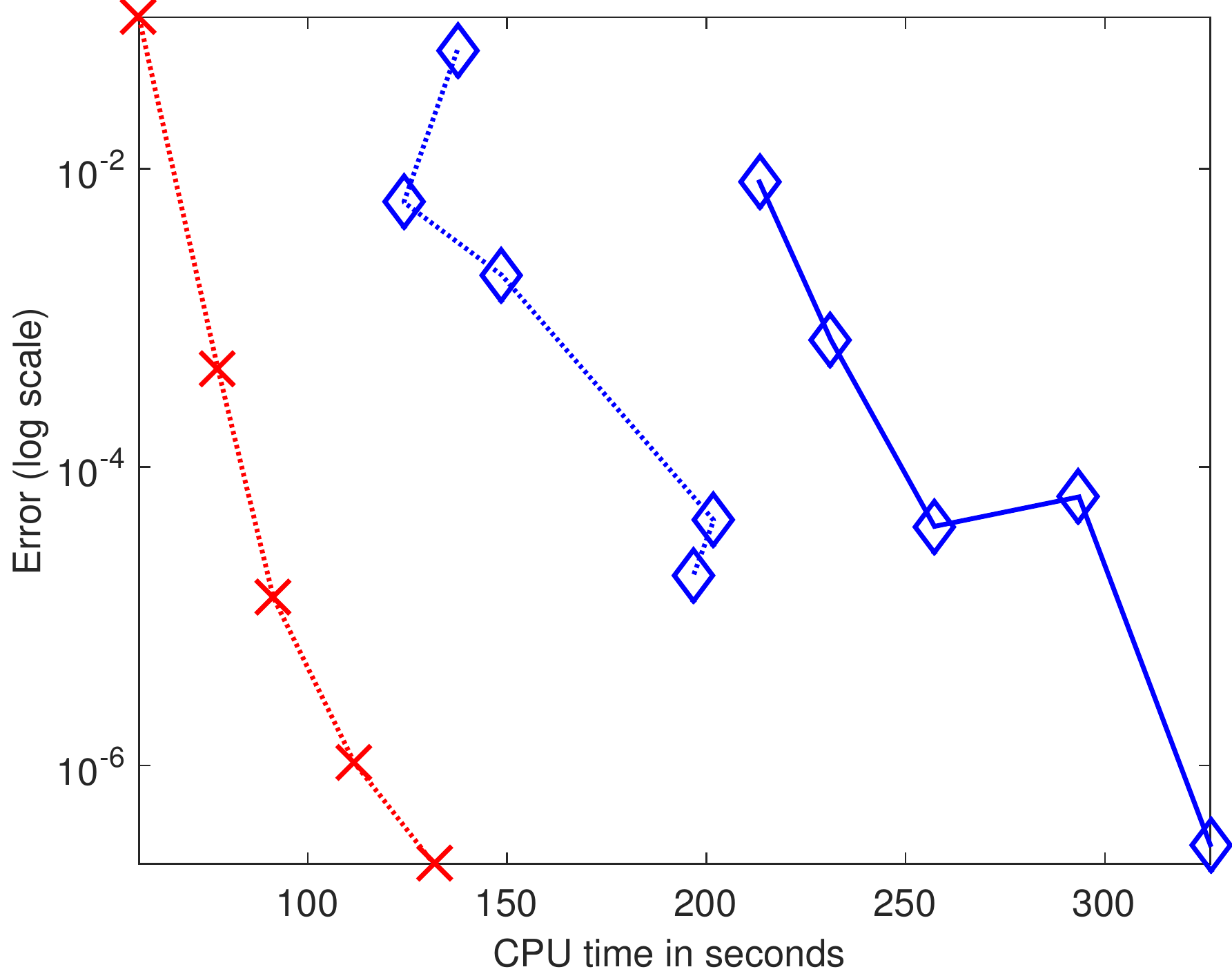}}\hspace{10pt}
    \subfigure[2D Brusselator $(N=300^2)$]{\includegraphics[scale = 0.4]{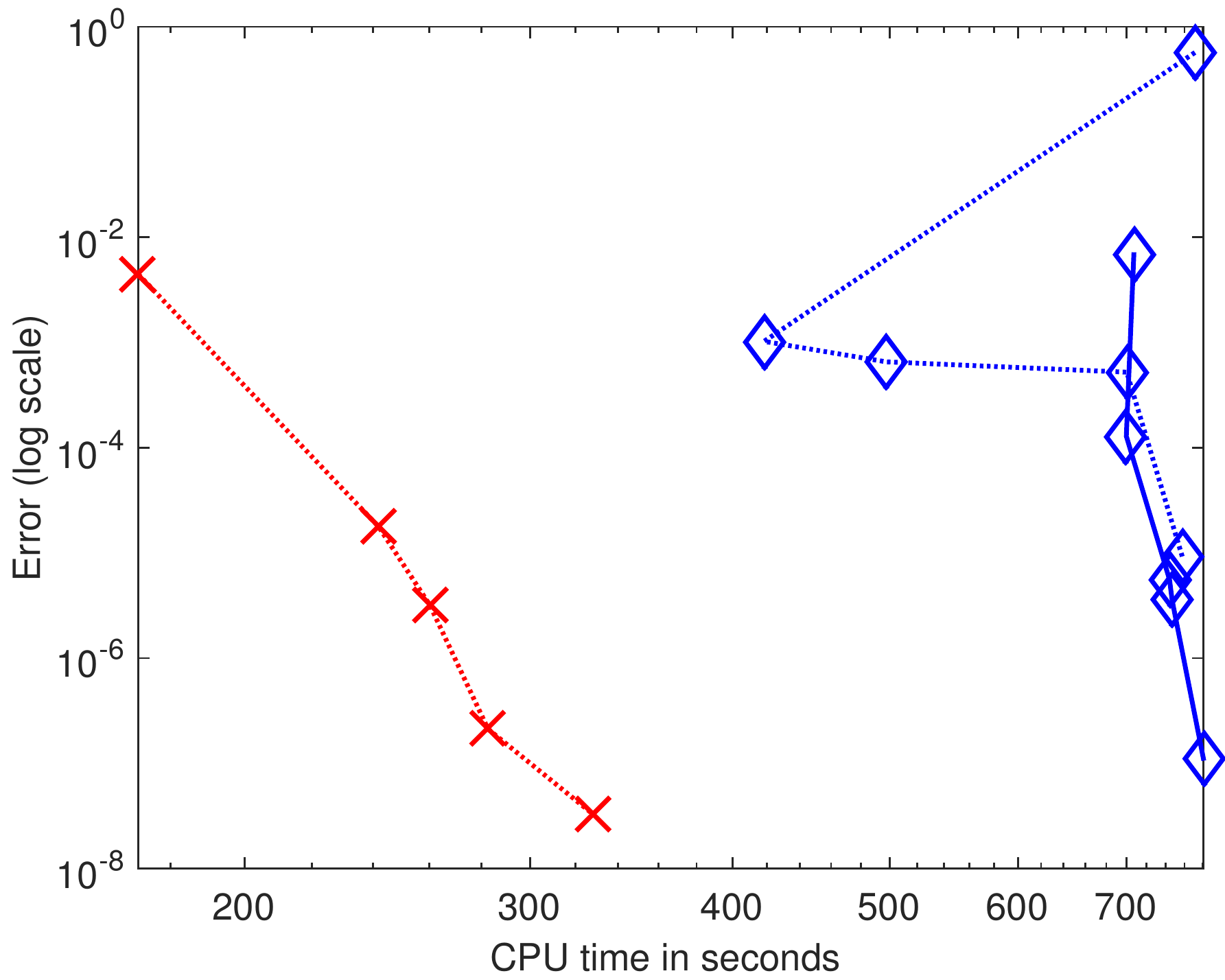}}\\
    \end{center}
    \caption{CPU execution time versus error for variable time step experiments}
    \label{fig:precVar}
    \end{figure}
    
 \section{Conclusions and future work}
We have extended the stiff order conditions and convergence theory in \cite{luan_oster} for EXPRB methods to EPIRK-type methods. We offered a different approach to solving the stiff order conditions that allows construction of efficient schemes of several types particularly when these methods are used in conjunction with the adaptive Krylov algorithm.  Using the generality of the EPRIK framework we constructed new stiffly accurate fourth and fifth-order schemes and numerically confirmed they achieved their full predicted order of accuracy on a set of test problems. Our numerical experiments further showed that the new technique of deriving horizontal or mixed EPIRK schemes does offer improved computational savings compared to previously derived (EPIRK \& EXPRB) methods.  For a given exponential method, however, the most efficient implementation will depend on its coefficients and the structure of the problem under consideration.  We are currently working on a modified adaptive Krylov algorithm that provides more computational savings for horizontal and mixed optimized EPIRK schemes. Development of better guidelines in constructing/choosing the most efficient integrator for a given problem is a goal of our future investigations.  We also plan to extend/develop the stiff order conditions theory for partitioned (or split) EPIRK, implicit-exponential type-methods. Finally, more research is needed to investigate whether stiffly accurate exponential integrators that do not suffer from order reduction can be developed for problems with non-homogeneous boundary conditions.

\section*{Acknowledgments}
This work was supported by a grant from the National Science Foundation, Computational Mathematics Program, under Grant No. 1115978.  The authors express gratitude to Dr. Vu Thai Luan and the referees for carefully reading the manuscript and providing helpful suggestions for improvements. 

\appendix
\section{}\label{sec:appendix}

In this appendix we present the details of the necessary modifications to the theory in \cite{luan_oster} to prove convergence of the stiffly accurate EPIRK methods.  The convergence proof in \cite{luan_oster} proceeds by expressing the global error $e_{n+1}=u_{n+1}-u(t_{n+1})=u_{n+1}-\tilde{u}_{n+1}$ in the following form 
%
%
%
%
\begin{equation}
\label{eqn:globalErrorExpression1}
e_{n+1}=e^{h_n\Jt}e_n + h_nP_n+\tilde{e}_{n+1},\quad e_0=0,
\end{equation} 
where
\begin{equation}
\label{eqn:Pn}
P_{n}=q_n+\sum_{i=2}^sQ_{ni}
\end{equation}
with
\begin{eqnarray}
q_{n}&=& \varphi_{1}(h_nJ_n)(N_n(u_n)-N_n(\tilde{u}_n))+(\varphi_1(h_nJ_n)-\varphi_1(h_n\Jt))f(\ut),\label{eqn:qn}\\
Q_{ni}&=&(b_{i}(h_nJ_n)-b_{i}(h_n\Jt))\widehat{r}_{ni}+b_{i}(h_nJ_n)(r(U_{ni})-\widehat{r}_{ni}).\label{eqn:Qn}
\end{eqnarray}
Lemmas 4.1 through 4.5 in \cite{luan_oster} provide bounds to the different terms in this expression.  All of these lemmas are directly applicable to the EPIRK methods except for Lemma 4.5.  Here we present a modified proof of Lemma 4.5 that accounts for the fact that EPIRK methods employ the general $\psi$-function rather than the $\varphi_1$-function as in the exponential Rosenbrock methods.   To motivate the lemma we begin by applying Lemma 4.4 in \cite{luan_oster} to \eqref{eqn:Pn} and obtain the preliminary estimate 
\begin{equation}
\label{eqn:PnPrelimBound}
\norm{P_n}\leq Ch_n \norm{e_n}+C\norm{e_n}^2 + \sum_{i=2}^s C\norm{e_n}+C\norm{e_n}^2+C\left(h_n+\norm{e_n}+\norm{\widehat{E}_{ni}}\right)\norm{\widehat{E}_{ni}},
\end{equation}
where $\widehat{E}_{ni}=U_{ni}-\wh{U}_{ni}$ is the difference between the  numerical solutions obtained from (\ref{eqn:EPIRK}) and (\ref{eqn:EPIRKalongExact}). Our desired estimate for $P_n$ is obtained by bounding $\norm{\widehat{E}_{ni}}$ in terms of $\norm{e_n}$. The bound found in \cite{luan_oster} for EXPRB methods only holds for methods whose internal stages strictly use $\varphi_1$-function in the first term.  With the additional assumption that the method satisfies Assumption 3, we prove the same bound holds for any linear combination of $\varphi$-functions.
  

\begin{lem} \label{lemma:EstimatePn} Under Assumptions 1-3, for all $i$, we have 
\begin{eqnarray}
&& \norm{\widehat{E}_{ni}}\leq C\norm{e_n}+Ch_n\norm{e_n}^2+Ch_n^5\label{eqn:EstimateEhat}\\
&& \norm{P_n}\leq C\norm{e_n}+C\norm{e_n}^2+Ch_n^6\label{eqn:EstimatePn}
\end{eqnarray}
as long as the global errors $e_n$ remain in a sufficiently small neighborhood of $0$ and $h_n \leq C_H$.
\end{lem}
\begin{proof}
Without loss of generality and for sake of presentation, we will assume the method satisfies $p_{i1k}=g_{i1}$ of Assumption 3 for each $i$ and all $k$.
We begin by proving the estimate
\begin{equation}
\label{eqn:Ehat1}
\norm{\widehat{E}_{ni}}\leq C\norm{e_n}+Ch_n\norm{e_n}^2+h_n\sum_{j=2}^{i-1} C(h_n+\norm{e_n}+\norm{\widehat{E}_{nj}})\norm{\widehat{E}_{nj}}.
\end{equation}
By adding and subtracting $\alpha_{i1}h_n p_{i1k}\varphi_k(g_{i1}h_nJ_n)f(\ut)$ for each $k=1,\dots,s$ to $\widehat{E}_{ni}$ we can then write $\widehat{E}_{ni}$ as 
\begin{equation}\label{eqn:EhatEqn1}
\begin{aligned}
\widehat{E}_{ni}&= e_n +\alpha_{i1}h_n\sum_{k=1}^{s}p_{i1k}\varphi_k(g_{i1}h_nJ_n)\left(f(u_n)-f(\tilde{u}_n)\right)+\alpha_{i1}h_n\sum_{k=1}^{s}p_{i1k}\left(\varphi_k(g_{i1}h_nJ_n)-\varphi_k(g_{i1}h_n\Jt)\right)f(\ut)+\\
&\qquad + h_n\sum_{j=2}^{i-1}a_{ij}(h_nJ_n)(r_{nj}-\hat{r}_{nj})+h_n\sum_{j=2}^{i-1}(a_{ij}(h_nJ_n)-a_{ij}(h_n \Jt))\hat{r}_{nj}.
\end{aligned}
\end{equation}
Using the identity $f(u)-f(\ut)=J_ne_n + N_n(u_n)-g_n(\ut)$, \eqref{eqn:CoeffRestrictionConvergence} and recurrence relation (\ref{eqn:phiRecurrenceRelation}), \eqref{eqn:EhatEqn1} can be expressed as
\begin{equation}
\hspace*{-15pt}\begin{aligned}
\widehat{E}_{ni}
&=e_n + \alpha_{i1}\sum_{k=1}^{s}\left[h_ng_{i1}\varphi_k(g_{i1}h_nJ_n)J_ne_n+h_ng_{i1}\varphi_k(g_{i1}h_nJ_n)(N_n(u_n)-N_n(\ut))\right]+\\
&\qquad+\alpha_{i1}g_{i1}h_n\sum_{k=1}^s(\varphi_k(g_{i1}h_nJ_n)-\varphi_k(g_{i1}h_n\Jt))f(\ut) + h_n\sum_{j=2}^{i-1}a_{ij}(h_nJ_n)(r_{nj}-\hat{r}_{nj})+h_n\sum_{j=2}^{i-1}(a_{ij}(h_nJ_n)-a_{ij}(h_n \Jt))\hat{r}_{nj}\\
&= e_n + \alpha_{i1}\sum_{k=1}^s\left[\left(\varphi_{k-1}(g_{i1}h_n J_n)-1/k!\right)e_n+\alpha_{i1}g_{i1}h_n\varphi_k (g_{i1}h_nJ_n)(N_n(u_n)-N_n(\ut))\right]+\;\;\;\\
&\qquad + \alpha_{i1}g_{i1}h_n\sum_{k=1}^s(\varphi_k(g_{i1}h_nJ_n)-\varphi_k(g_{i1}h_n\Jt))f(\ut) + h_n\sum_{j=2}^{i-1}a_{ij}(h_nJ_n)(r_{nj}-\hat{r}_{nj})+h_n\sum_{j=2}^{i-1}a_{ij}(h_nJ_n)-a_{ij}(h_n \Jt))\hat{r}_{nj}
\end{aligned}.
\end{equation}
The estimate \eqref{eqn:Ehat1} then follows from the positive-scalability and sub-additivity of the norm, the estimates of Lemmas 4.1 \& 4.3 in \cite{luan_oster}, boundedness of $f(\ut)=\tilde{u}_n'$ and $\varphi_k(h_nJ)$ (and $a_{ij}(h_nJ_n)$).
Now we can prove \eqref{eqn:EstimateEhat}. Since $e_n$ is assumed to remain in a sufficiently small neighborhood of $0$, there exists $0<\delta <1$ such that $\norm{e_n}<\delta$ for all $n$.  This implies that for each $n$, $\norm{e_n}^2\leq \norm{e_n}$ and furthermore shows that 
\begin{equation}
\label{eqn:EstimateEn2}
\norm{\widehat{E}_{n2}}\leq C\norm{e_n}+Ch_n\norm{e_n}^2\leq C\norm{e_n}+Ch_n\norm{e_n}
\end{equation}
by using \eqref{eqn:Ehat1} with $i=2$.  Assuming $\norm{\widehat{E}_{n i-1}}\leq C_1\norm{e_n}+C_2h_n\norm{e_n}^2$  we obtain
\begin{eqnarray}
\norm{\widehat{E}_{ni}} &\leq &C \norm{e_n}+C h_n \norm{e_n}^2
+h_n\sum_{j=2}^{i-1} C(h_n+\norm{e_n}+\norm{\hat{E}_{nj}})\norm{\hat{E}_{nj}}\\
& \leq &C \norm{e_n}+C h_n \norm{e_n}^2+h_n\sum_{j=2}^{i-1} C(h_n+\norm{e_n}+ (C_1 \norm{e_n}+C_2 h_n \norm{e_n}^2))( C_1 \norm{e_n}+C_2 h_n \norm{e_n}^2).\nonumber
\end{eqnarray}
By expanding the terms and using the assumption that $\norm{e_n}<\delta$ we arrive at 
\begin{eqnarray}
\norm{\widehat{E}_{ni}} &\leq & (C+h_n^2C_1)\norm{e_n}+(C +C_2h_n^2+C_1+C_1^2+C_2h_n+2C_2C_1h_n^2+C_2^2h_n^2)h_n\norm{e_n}^2\\
& \leq& (C+C_H^2C_1)\norm{e_n}+(C +C_2C_M^2+C_1+C_1^2+C_2C_M+2C_2C_1C_M^2+C_2^2C_M^2)h_n\norm{e_n}^2\\
&\leq& \mathbf{C} \norm{e_n}+\mathbf{C}h_n\norm{e_n}^2
\end{eqnarray}
where $C_M=\max(C_H,1)$ and $\mathbf{C}=\max((C+C_H^2C_1),(C +C_2C_M^2+C_1+C_1^2+C_2C_M+2C_2C_1C_M^2+C_2^2C_M^2),1)$.
The estimate (\ref{eqn:EstimatePn}) now follows from (\ref{eqn:PnPrelimBound}) and (\ref{eqn:EstimateEhat}).
\end{proof}

\bibliographystyle{amsplain}
\providecommand{\bysame}{\leavevmode\hbox to3em{\hrulefill}\thinspace}
\providecommand{\MR}{\relax\ifhmode\unskip\space\fi MR }
\providecommand{\MRhref}[2]{%
  \href{http://www.ams.org/mathscinet-getitem?mr=#1}{#2}
}
\providecommand{\href}[2]{#2}

\end{document}